\documentclass[11pt,reqno]{amsart}
\pdfoutput=1
\usepackage[most]{tcolorbox}
\usepackage{enumerate}
\usepackage{amsmath}
\usepackage{amsfonts}
\usepackage{amsthm}
\usepackage{amssymb}
\usepackage{hyperref}
\usepackage{bm}
\usepackage[left=1in,right=1in,top=1in,bottom=1in,bindingoffset=0cm]{geometry}
\usepackage[spacing=true,kerning=true,babel=true,tracking=true]{microtype}
\usepackage[foot]{amsaddr} 
\usepackage{enumitem}
\usepackage{float}
\usepackage[labelfont=bf]{caption}
\usepackage{caladea}
\usepackage[T1]{fontenc}
\usepackage[euler-digits,small]{eulervm}
\usepackage[bb=pazo,scr=esstix]{mathalfa}
\usepackage[backend=biber, style=alphabetic, sorting=nyt, maxnames=100,backref=true]{biblatex}
\usepackage{tikz}
\usetikzlibrary{decorations.pathmorphing}
\title{\large Searching for an Intruder on Graphs\\ and Their Subdivisions}
\date{}

\author{Anton~Bernshteyn}
\address[Anton Bernshteyn]{\normalfont School of Mathematics, Georgia Institute of Technology, Atlanta, GA, USA}
\email{bahtoh@gatech.edu}

\author{Eugene Lee}
\address[Eugene Lee]{\normalfont Department of Mathematical Sciences, Carnegie Mellon University, Pittsburgh, PA, USA}
\email{eleehuaj@andrew.cmu.edu}

\thanks{Research of the first named author is partially supported by the NSF grant DMS-2045412.}

\newcommand{\probend}{\vspace*{0.8\baselineskip} \hrule\hrule \vspace*{0.2\baselineskip} \hrule}

\newcommand{\abs}[1]{\left|#1\right|}
\newcommand{\pr}[1]{\left(#1\right)}
\newcommand{\FC}{\mathsf{FC}}
\newcommand{\PC}{\mathsf{PC}}

\newcommand{\s}{\mathsf{in}}
\newcommand{\gwt}[1]{\bm{#1}}
\newcommand{\pw}{\mathsf{pw}}
\renewcommand{\leq}{\leqslant}
\renewcommand{\le}{\leqslant}
\renewcommand{\ge}{\geqslant}
\renewcommand{\geq}{\geqslant}
\newcommand{\0}{\varnothing}

\newcommand{\concat}{{^\smallfrown}}
\newcommand{\bemph}[1]{{\normalfont#1}} 
\newcommand{\ep}[1]{\bemph{(}#1\bemph{)}} 

\newcommand{\emphd}[1]{{\fontseries{b}\selectfont#1}}

\hypersetup{
    colorlinks=true,
    linkcolor=blue,
    filecolor=magenta,      
    urlcolor=cyan,
}

\allowdisplaybreaks

\newtheoremstyle{bfnote}%
{}{}%
{\slshape}{}%
{\bfseries}{\bfseries.}%
{ }%
{\thmname{#1}\thmnumber{ #2}\thmnote{ \ep{\normalfont{}#3}}}

\newtheoremstyle{claim}%
{}{}%
{\slshape}{}%
{\bfseries}{.}%
{ }%
{\thmname{#1}\thmnumber{ #2}\thmnote{ \ep{\normalfont{}#3}}}

\theoremstyle{bfnote}
\newtheorem{theo}{Theorem}[section]
\newtheorem*{theo*}{Theorem}
\newtheorem{prop}[theo]{Proposition}
\newtheorem{lemma}[theo]{Lemma}
\newtheorem{claim}[theo]{Claim}
\newtheorem*{claim*}{Claim}
\newtheorem{corl}[theo]{Corollary}

\newtheorem*{corl*}{Corollary}

\theoremstyle{definition}
\newtheorem{defn}[theo]{Definition}
\newtheorem*{defn*}{Definition}
\newtheorem{exmp}[theo]{Example}

\newtheorem*{exmp*}{Example}
\newtheorem{prob}[theo]{Problem}

\theoremstyle{remark}
\newtheorem*{ques*}{Question}
\newtheorem*{remk*}{Remark}

\theoremstyle{claim}
\newcounter{ForClaims}[section]

\newtheorem{smallclaim}{Claim}[ForClaims]

\newenvironment{claimproof}{\noindent{\color{purple!20}$\circ$}\hspace{1em}}{\hfill{\color{purple!20}$\bullet$}\smallskip}

\makeatletter
\newcommand{\neutralize}[1]{\expandafter\let\csname c@#1\endcsname\count@}
\makeatother

\newenvironment{theocopy}[1]
{\renewcommand{\thetheo}{\ref{#1}}%
	\neutralize{theo}\phantomsection
	\begin{theo}}
	{\end{theo}}
	
\newenvironment{defncopy}[1]
{\renewcommand{\thetheo}{\ref{#1}}%
	\neutralize{theo}\phantomsection
	\begin{defn}}
	{\end{defn}}

\setlist{topsep=3pt,itemsep=3pt}
\usepackage{titlesec}
\renewcommand{\thesubsection}{\arabic{section}.\Alph{subsection}}
\titleformat{\section}[block]{\large\bfseries}{\thesection.}{1ex}{}
\titleformat{\subsection}[block]{\bfseries}{\thesubsection.}{1ex}{}
\titleformat{\subsubsection}[runin]{\itshape}{\bfseries\upshape\thesubsubsection.}{1ex}{}[.---]
\titlespacing*{\section}{0pt}{*3}{*1}
\titlespacing*{\subsection}{0pt}{*3}{*1}
\titlespacing*{\subsubsection}{0pt}{*1.5}{*0}

\renewbibmacro{in:}{}

\renewbibmacro*{volume+number+eid}{%
    \printfield{volume}%
    \setunit*{\addnbspace}
    \printfield{number}%
    \setunit{\addcomma\space}%
    \printfield{eid}}

\DeclareFieldFormat[article]{volume}{\textbf{#1}\space}
\DeclareFieldFormat[article]{number}{\mkbibparens{#1}}

\DeclareFieldFormat{journaltitle}{#1,}
\DeclareFieldFormat[thesis]{title}{\mkbibemph{#1}\addperiod}
\DeclareFieldFormat[article, unpublished, thesis]{title}{\mkbibemph{#1},}
\DeclareFieldFormat[book]{title}{\mkbibemph{#1}\addperiod}
\DeclareFieldFormat[unpublished]{howpublished}{#1, }

\DeclareFieldFormat{pages}{#1}

\DeclareFieldFormat[article]{series}{Ser.~#1\addcomma}

\begin{document}

\begin{abstract} In this paper we analyze a variant of the pursuit-evasion game on a graph $G$ where the intruder occupies a vertex, is allowed to move to adjacent vertices or remain in place, and is `invisible' to the searcher, meaning that the searcher operates with no knowledge of the position of the intruder. On each stage, the searcher is allowed to inspect an arbitrary set of $k$ vertices. The minimum $k$ for which the searcher can guarantee the capture of the intruder is called the \textit{inspection number} of $G$. We also introduce and study the \textit{topological inspection number}, a quantity that captures the limiting behavior of the inspection number under subdivisions of $G$. Our central theorem provides a full classification of graphs with topological inspection number up to $3$.
\end{abstract}
\maketitle

\section{Introduction}
All graphs in this paper are finite, undirected, and simple. A class of graph-theoretic problems that has been the topic of much interest is that of pursuit-evasion games, where the object is to capture an intruder who is allowed to move within a graph in some manner. 
The general nature of this problem lends itself to many variants; whether the intruder occupies vertices or edges of the graph, the amount of information revealed to the searchers as well as the constraints on their movement are all facets of the problem that can be modified to ask different questions. The version we analyze in this paper is the following:
\begin{defn}[Zero-Visibility Search Game]
For a positive integer $k$, the \emphd{zero-visibility $k$-search game} is a two-player game played on a graph $G$, where one player serves as the \emphd{searcher}, while the other plays as the \emphd{intruder}. Both players have full knowledge of the graph.

The intruder starts the game by occupying a vertex of their choice. At all times, the location of the intruder is unknown to the searcher (in other words, the intruder is `invisible').
The players alternate turns, beginning with the searcher. On each turn, the searcher `inspects' an arbitrary $k$-element set of vertices, and if the intruder is presently at any of those vertices, the game ends and the searcher wins. Otherwise, on the intruder's turn, they can choose to move to an adjacent vertex or remain at the same vertex. We emphasize that the searcher does not know if and where the intruder moves. 

A winning strategy for the searcher is a finite sequence of moves which guarantees that the intruder will always be caught
, regardless of the intruder's starting location and moves. The \emphd{inspection number} of the graph, denoted $\s(G)$, is the minimum $k$ such that a winning strategy for the searcher exists.\label{defn:s(G)}
\end{defn}

This problem has a natural equivalent formulation in terms of curbing the spread of an infection on the vertices of $G$. Suppose that each vertex can be in one of two states: \emph{cleared} or \emph{contaminated}. Initially, all vertices are contaminated. On every move, we may clear any $k$ vertices. After this, every cleared vertex that has a contaminated neighbor becomes contaminated again. The {inspection number} of $G$ is then equal to the smallest $k$ for which all vertices of $G$ can be cleared in finitely many moves (see Proposition~\ref{prop:searching_clearing}).

Very similar pursuit-evasion games have been considered previously. To\v{s}i\'{c} \cite{Tosic} introduced the \emph{zero-visibility Cops \&{} Robber game}, which differs from ours in that instead of examining arbitrary sets of $k$ vertices, the searcher controls $k$ tokens that occupy vertices of $G$ and can only be moved along edges. This model was also studied by Berger, Gilbers, Gr\"une, and Klein \cite{Berger} and Dereniowski, Dyer, Tifenbach, and Yang \cite{Dereniowski}. 
Another related model was studied in \cite{ACDL} by Aydinian, Cicalese, Deppe, and Lebedev. In their model, the searcher can only learn whether the intruder is located at one of the examined vertices \ep{but not at which vertex specifically}. 
The game investigated in \cite{Haslegrave} by Haslegrave, in \cite{Britnell} by Britnell and Wildon, and in \cite{Abramovskaya} by Abramovskaya, Fomin, Golovach, and Pilipczuk, called in the latter the \emph{Hunters \&{} Rabbit game}, is almost the same as ours except that the intruder is required to move to an adjacent vertex on every move.
For further references, see the surveys \cite{AlspachSurvey} by Alspach, \cite{FominSurvey} by Fomin and Thilikos, and \cite{BonatoYangSurvey} by Bonato and Yang.

It is easy to see that the complete graph $K_n$ has inspection number $n$, paths have inspection number $2$, and cycles have inspection number $3$. We also compute the inspection number of rectangular grid graphs:
\begin{theocopy}{theo:gr}
For each $n, m \ge 2$, the inspection number of the $n \times m$ grid graph is $\min\{n, m\} + 1$.\label{theo:grids}
\end{theocopy}

Trees can have arbitrarily high inspection numbers:

\begin{theocopy}{theo:trees}
For every $k$, there is a tree $T$ of maximum degree $3$ with $\s(T) > k$.\label{theo:trees_copy}
\end{theocopy}

Our proof of Theorem~\ref{theo:trees_copy} is surprisingly intricate. In general, proving lower bounds on the inspection number appears challenging.

An interesting aspect of pursuit-evasion games is whether or not they admit monotonic optimal strategies. 
Here a strategy is \textit{monotonic} if 
cleared vertices never become contaminated again. 
Some pursuit-evasion games admit optimal monotonic strategies \cite{Bienstock, LaPaugh}, while some do not \cite{Yang}. 
In their investigation of the zero-visibility Cops \&{} Robber game, Dereniowski et al.~\cite{Dereniowski} established sharp bounds on the number of cops needed for a monotonic winning strategy in terms of the pathwidth of the underlying graph. 
We prove the following result in the same vein:

\begin{theocopy}{theo:mono}
    For every connected graph $G$, the minimum $k$ such that the searcher has a monotonic winning strategy in the zero-visibility $k$-search game on $G$ \ep{the {\upshape\emphd{monotonic inspection number}}} is equal to $\pw(G) + 1$, where $\pw(G)$ denotes the pathwidth of $G$.\label{theo:mono_copy}
\end{theocopy}

Theorem~\ref{theo:mono} implies that $\s(G) \leq \pw(G) + 1$ for all connected graphs $G$. The $n \times m$ grid graph meets this bound with equality, since it has pathwidth $\min\{n, m\}$ \cite{Ellis} (note that the proof in the cited paper is via considering another pursuit-evasion game) and inspection number $\min\{n, m\} + 1$ by Theorem~\ref{theo:grids}. On the other hand, there are graphs $G$ with $\s(G) < \pw(G) + 1$ (and hence without a monotonic optimal search strategy). One such example is shown in Fig.~\ref{fig:K4}. Our further results imply that there are graphs with inspection number $3$ and arbitrarily high pathwidth (see Theorem \ref{theo:trees1}).

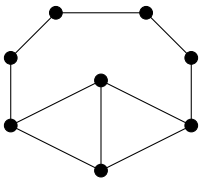
\begin{figure}[t]

\begin{center}    

\begin{tikzpicture}[scale=1.2]
    \filldraw (0,0) circle (2pt);
    \filldraw (1,0.5) circle (2pt);
    \filldraw (1,-0.5) circle (2pt);
    \filldraw (2,0) circle (2pt);
    
    \filldraw (0,0.75) circle (2pt);
    \filldraw (0.5,1.25) circle (2pt);
    \filldraw (1.5,1.25) circle (2pt);
    \filldraw (2,0.75) circle (2pt);
    
    \draw (0,0) -- (1,0.5) -- (1,-0.5) -- (2,0) -- (2,0.75) -- (1.5, 1.25) -- (0.5, 1.25) -- (0,0.75) -- (0,0) -- (1,-0.5) (1,0.5) -- (2,0);
\end{tikzpicture}

\caption{This graph (a subdivision of $K_4$) has pathwidth $3$ and hence monotonic inspection number $4$, but its ordinary inspection number is $3$ (see Example~\ref{ex:k4subdiv}).\label{fig:K4}}
\end{center}
\end{figure}

Another interesting direction is to study pursuit-evasion games on subdivisions of graphs, see, e.g., \cite{Delcourt}. Edge subdivisions can both increase and decrease the inspection number. For example, the complete graph $K_4$ has inspection number $4$ (see Example~\ref{exp:complete}). The subdivision of $K_4$ shown in Fig.~\ref{fig:K4} has inspection number $3$, but it also has a further subdivision with inspection number $4$ (see Theorem~\ref{res:1}). We hence wish to investigate the ``limiting'' behavior of the inspection number under edge subdivisions, captured precisely in the following notion:

\begin{defncopy}{defn:top_search} The \emphd{topological inspection number} of a graph $G$, denoted $\s_t(G)$, is the minimum $k$ such that for every subdivision $H$ of $G$, there is a further subdivision $H'$ of $H$ with $\s(H') \leq k$.

Equivalently, $\s_t(G)$ is the minimum $k$ such that for all $\ell$, $G$ has a subdivision $H$ with $\s(H) \le k$ in which every edge of $G$ is subdivided at least $\ell$ times.
\end{defncopy}

The topological inspection number can deviate significantly from the inspection number. For instance, by Theorem~\ref{theo:trees}, trees can have arbitrarily high inspection numbers, but we shall prove the following:

\begin{theo}
    The topological inspection number of every tree is at most $3$.\label{theo:trees1}
\end{theo}

Theorem~\ref{theo:trees1} also implies that there are graphs with inspection number $3$ and arbitrarily high pathwidth, since the pathwidth does not decrease under edge subdivisions and there are trees of arbitrarily high pathwidth.


Theorem~\ref{theo:trees1} is a corollary of our main result, which provides a complete characterization of graph with topological inspection number at most $3$. It is not hard to see that the only connected graphs with topological inspection number at most $2$ are paths (see Proposition~\ref{prop:paths}). 
The case of graphs with topological inspection number at most $3$ 
turns out to be significantly more involved:

\begin{theocopy}{theo:main}
Let $G$ be a connected graph with $\abs{V(G)} \ge 2$. The following are equivalent:
\begin{enumerate}[label=\ep{\normalfont\arabic*}]
    \item\label{item:ts3} $G$ has topological inspection number at most $3$.
    \item\label{item:families} $G$ does not contain as a subgraph any element of $\mathcal{F}_1, \mathcal{F}_2$, or $\mathcal{F}_3$ \ep{see Definition \ref{defn:families}}.
    \item\label{item:GSP} $G$ has a simple generalized series-parallel decomposition \ep{see Definitions \ref{defn:GSP}--\ref{defn:simple}}.
\end{enumerate}\label{theo:classification}
\end{theocopy}

This equivalence gives two different ways of detecting if a graph has topological inspection number at most $3$. Item \ref{item:families} describes such graphs by means of an explicit list of forbidden subgraphs. The forbidden subgraphs form three infinite families. The first family, $\mathcal{F}_1$, comprises the subdivisions of $K_4$. The other two families, $\mathcal{F}_2$ and $\mathcal{F}_3$, are described in Definition \ref{defn:families}. Representative examples of each family of forbidden subgraphs are depicted in Fig.~\ref{fig:representatives}. All the graphs in these families contain cycles, so the equivalence \ref{item:ts3} $\Longleftrightarrow$ \ref{item:families} yields Theorem~\ref{theo:trees1}.

\begin{figure}[t]
\begin{center}
\begin{tikzpicture}[x=0.75pt,y=0.75pt,yscale=-1,xscale=1]
\draw    (105.31,142.67) .. controls (109.14,116.67) and (113.52,113.67) .. (125.57,103.67) ;
\draw    (145.83,148.67) .. controls (143.64,120.67) and (137.07,111.67) .. (125.57,103.67) ;
\draw    (125.57,187.67) .. controls (135.43,179.67) and (143.64,165.67) .. (145.83,148.67) ;
\draw    (125.57,187.67) .. controls (110.79,181.67) and (107.5,171.67) .. (105.31,142.67) ;
\draw    (105.31,142.67) .. controls (117.36,155.67) and (130.5,158.67) .. (145.83,148.67) ;
\draw    (125.57,187.67) .. controls (57.83,191.67) and (58.83,102.67) .. (125.57,103.67) ;
\draw    (215.38,144.4) .. controls (200.36,130.16) and (215.38,104.95) .. (238.99,101.67) ;
\draw    (216.45,145.5) .. controls (218.6,132.35) and (223.97,113.72) .. (238.99,101.67) ;
\draw    (236.84,187.13) .. controls (215.38,187.13) and (194.99,160.84) .. (216.45,145.5) ;
\draw    (236.84,187.13) .. controls (223.97,172.89) and (217.53,160.84) .. (216.45,145.5) ;
\draw    (238.99,143.3) .. controls (227.19,133.44) and (230.41,111.53) .. (238.99,101.67) ;
\draw    (236.84,187.13) .. controls (225.04,177.27) and (230.41,153.17) .. (238.99,143.3) ;
\draw    (238.99,143.3) .. controls (249.72,133.44) and (251.87,114.82) .. (238.99,101.67) ;
\draw    (236.84,187.13) .. controls (250.8,177.27) and (251.87,154.26) .. (238.99,143.3) ;
\draw    (265.43,144.4) .. controls (274.1,124.09) and (265.43,104.95) .. (236.84,101.67) ;
\draw    (264.13,145.5) .. controls (261.53,132.35) and (255.04,113.72) .. (236.84,101.67) ;
\draw    (236.84,187.13) .. controls (262.83,187.13) and (278.39,162.44) .. (264.13,145.5) ;
\draw    (236.84,187.13) .. controls (252.44,172.89) and (262.83,160.84) .. (264.13,145.5) ;
\draw    (324.67,131.01) .. controls (323.78,121.13) and (327.72,111.13) .. (341.98,106.67) ;
\draw    (324.67,131.01) .. controls (334.28,128.13) and (339.53,118.13) .. (343.8,106.67) ;
\draw    (341.98,154.13) .. controls (326.4,149.13) and (325.09,143.13) .. (324.67,131.01) ;
\draw    (341.98,154.13) .. controls (342.16,141.13) and (335.59,134.13) .. (324.67,131.01) ;
\draw    (360.05,131.01) .. controls (360.71,117.74) and (353.55,109.13) .. (342.41,106.67) ;
\draw    (360.05,131.01) .. controls (348.55,125.13) and (345.55,117.13) .. (342.41,106.67) ;
\draw    (342.41,154.13) .. controls (355.55,151.13) and (360.55,144.13) .. (360.05,131.01) ;
\draw    (342.41,154.13) .. controls (347.55,141.13) and (350.55,137.13) .. (360.05,131.01) ;
\draw    (381.67,158.01) .. controls (380.78,148.13) and (384.72,138.13) .. (398.98,133.67) ;
\draw    (381.67,158.01) .. controls (391.28,155.13) and (396.53,145.13) .. (400.8,133.67) ;
\draw    (398.98,181.13) .. controls (383.4,176.13) and (382.09,170.13) .. (381.67,158.01) ;
\draw    (398.98,181.13) .. controls (399.16,168.13) and (392.59,161.13) .. (381.67,158.01) ;
\draw    (417.05,158.01) .. controls (417.71,144.74) and (410.55,136.13) .. (399.41,133.67) ;
\draw    (417.05,158.01) .. controls (405.55,152.13) and (402.55,144.13) .. (399.41,133.67) ;
\draw    (399.41,181.13) .. controls (412.55,178.13) and (417.55,171.13) .. (417.05,158.01) ;
\draw    (399.41,181.13) .. controls (404.55,168.13) and (407.55,164.13) .. (417.05,158.01) ;
\draw    (343.8,106.67) .. controls (383.8,76.67) and (359.41,211.13) .. (399.41,181.13) ;
\draw    (342.41,154.13) .. controls (357.55,183.13) and (374.55,113.13) .. (400.8,133.67) ;

\draw (104,197.4) node [anchor=north west][inner sep=0.75pt]    {$\mathcal{F}_{1}$};
\draw (232,201.4) node [anchor=north west][inner sep=0.75pt]    {$\mathcal{F}_{2}$};
\draw (366,202.4) node [anchor=north west][inner sep=0.75pt]    {$\mathcal{F}_{3}$};
\end{tikzpicture}
\caption{A representative example of each class of forbidden subgraphs. The family $\mathcal{F}_1$ comprises the subdivisions of $K_4$, while the other two families are more complicated (see Definition \ref{defn:families}).}\label{fig:representatives}
\end{center}
\end{figure}
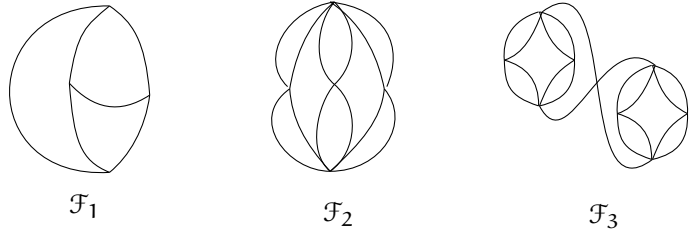

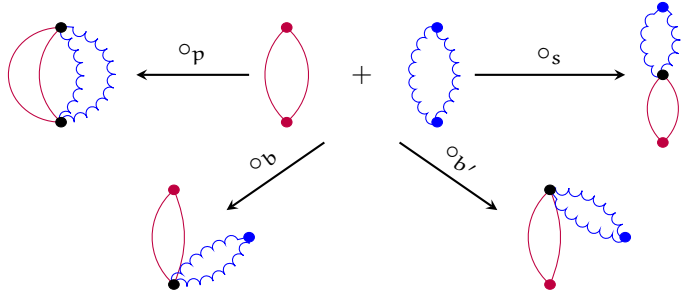
\begin{figure}[t]
    \centering
    \begin{tikzpicture}[yscale=0.9]
       \node at (0,0) {$+$};
       
       \draw[bend left=45,color=purple] (-1,-0.7) to (-1,0.7);
       \draw[bend right=45,color=purple] (-1,-0.7) to (-1,0.7);
       \filldraw[color=purple] (-1,-0.7) circle (2pt);
       \filldraw[color=purple] (-1,0.7) circle (2pt);
       
       \draw[bend left=45,style={decorate,decoration=bumps},color=blue] (1,-0.7) to (1,0.7);
       \draw[bend right=45,style={decorate,decoration=bumps},color=blue] (1,-0.7) to (1,0.7);
       \filldraw[color=blue] (1,-0.7) circle (2pt);
       \filldraw[color=blue] (1,0.7) circle (2pt);
       
       \draw[color=purple] (-4,-0.7) arc (270:90:0.7);
       \draw[bend left=45,color=purple] (-4,-0.7) to (-4,0.7);
       \draw[bend right=45,style={decorate,decoration=bumps},color=blue] (-4,-0.7) to (-4,0.7);
       \draw[style={decorate,decoration=bumps},color=blue] (-4,-0.7) arc (-90:90:0.7);
       \filldraw (-4,-0.7) circle (2pt);
       \filldraw (-4,0.7) circle (2pt);
       
       \draw[bend left=45,color=purple] (4,-1) to (4,0);
       \draw[bend right=45,color=purple] (4,-1) to (4,0);
       \draw[bend left=45,style={decorate,decoration=bumps},color=blue] (4,0) to (4,1);
       \draw[bend right=45,style={decorate,decoration=bumps},color=blue] (4,0) to (4,1);
       \filldraw[color=purple] (4,-1) circle (2pt);
       \filldraw (4,0) circle (2pt);
       \filldraw[color=blue] (4,1) circle (2pt);
       
       \draw[bend left=45,color=purple] (-2.5,-3.1) to (-2.5,-1.7);
       \draw[bend right=20,color=purple] (-2.5,-3.1) to (-2.5,-1.7);
       \draw[bend left=20,style={decorate,decoration=bumps},color=blue] (-2.5,-3.1) to (-1.5,-2.4);
       \draw[bend right=45,style={decorate,decoration=bumps},color=blue] (-2.5,-3.1) to (-1.5,-2.4);
       \filldraw[color=purple] (-2.5,-1.7) circle (2pt);
       \filldraw (-2.5,-3.1) circle (2pt);
       \filldraw[color=blue] (-1.5,-2.4) circle (2pt);
       
       \draw[bend left=45,color=purple] (2.5,-3.1) to (2.5,-1.7);
       \draw[bend right=20,color=purple] (2.5,-3.1) to (2.5,-1.7);
       \draw[bend left=20,style={decorate,decoration=bumps},color=blue] (3.5,-2.4) to (2.5,-1.7);
       \draw[bend right=45,style={decorate,decoration=bumps},color=blue] (3.5,-2.4) to (2.5,-1.7);
       \filldraw (2.5,-1.7) circle (2pt);
       \filldraw[color=purple] (2.5,-3.1) circle (2pt);
       \filldraw[color=blue] (3.5,-2.4) circle (2pt);
       
       \draw [thick,-stealth] (1.5,0) to node[midway,above]{$\circ_s$} (3.5,0);
       \draw [thick,-stealth] (-1.5,0) to node[midway,above]{$\circ_p$} (-3,0);
       \draw [thick,-stealth] (-0.5,-1) to node[midway,above,sloped]{$\circ_b$} (-1.8,-2);
       \draw [thick,-stealth] (0.5,-1) to node[midway,above,sloped]{$\circ_{b'}$} (1.8,-2);
    \end{tikzpicture}
    \caption{Pictorial representations of the four permissible operations defining the class of generalized series-parallel graphs. A decomposition is \textit{simple} if each of the $\circ_p$ operations in it has one of the subgraphs involved containing an edge separating its \textit{terminals}, the two distinguished vertices.}\label{fig:gsp}
\end{figure}

Item \ref{item:GSP} gives a constructive characterization of graphs with topological inspection number at most $3$. Roughly speaking, it says that a graph with topological inspection number at most $3$ can be inductively built up from individual edges via certain permissible operations. The four permissible operations are the \emph{series operation} $\circ_s$, the \emph{parallel operation} $\circ_p$, and the two \emph{branch operations} $\circ_b$ and $\circ_{b'}$. These operations are illustrated in Fig.~\ref{fig:gsp}. Graphs that can be constructed from individual edges using these operations are called \emph{generalized series-parallel} \ep{\emph{GSP} for short}; see Definition~\ref{defn:GSP} for details.

It is well-known that $2$-connected $K_4$-subdivision-free graphs are \emph{series-parallel}, i.e., they can be built from individual edges using only the operations $\circ_s$ and $\circ_p$ \cite[Theorem 2]{Duffin}. In general, connected $K_4$-subdivision-free graphs are GSP (see \S\ref{subsec:K4-free}). Since subdivisions of $K_4$ belong to the forbidden family $\mathcal{F}_1$, item \ref{item:families} of Theorem~\ref{theo:classification} implies that every connected graph $G$ with $\s_t(G) \leq 3$ must be GSP. However, as the graphs in $\mathcal{F}_2$ and $\mathcal{F}_3$ demonstrate, not all GSP graphs have topological inspection number at most $3$. The main point of \ref{item:GSP} is that a graph $G$ with $\s_t(G) \leq 3$ must admit a GSP decomposition that is \emph{simple}, which means that every time the parallel operation is used, at least one of the graphs it is applied to has to contain an edge separating the two distinguished vertices. The details are given in Definition~\ref{defn:simple}.

Our proof of the equivalence \ref{item:ts3} $\Longleftrightarrow$ \ref{item:GSP} is constructive; that is, given a graph $G$ with $\s_t(G) \leq 3$, it explicitly describes a subdivision of $G$ with a successful search strategy. 

Our characterization of graphs with topological inspection number at most $3$ via forbidden subgraphs shows that every such graph must be planar, since every non-planar graph contains a subdivision of $K_4$. Furthermore, some forbidden graphs in $\mathcal{F}_1 \cup \mathcal{F}_2 \cup \mathcal{F}_3$ are planar. On the other hand, none of the forbidden subgraphs are outerplanar, so every outerplanar graph has topological inspection number at most $3$. Also, there exist graphs which are not outerplanar but have topological inspection number $3$ ($K_{2, 3}$ is such an example). Thus, graphs with topological inspection number at most $3$ form a natural class strictly between outerplanar and planar graphs.

The rest of the paper is organized as follows.
\begin{itemize}
    \item Section \ref{sec:prelim} introduces the notation to be used in the rest of the paper. 
    \item Section \ref{sec:comp} computes bounds on the inspection number for several special classes of graphs. In particular, it contains the proofs of Theorems~\ref{theo:grids} and \ref{theo:trees}.
    \item Section \ref{sec:mono} investigates monotonic search strategies and proves Theorem~\ref{theo:mono}.
    \item Section \ref{sec:ts} introduces the topological inspection number and classifies graphs with topological inspection number at most $2$.
    \item Section \ref{sec:defns} introduces the necessary definitions for the statement of Theorem \ref{theo:classification}.
    \item Section \ref{sec:forb_lower} proves that each graph in $\mathcal{F}_1 \cup \mathcal{F}_2 \cup \mathcal{F}_3$ has topological inspection number at least $4$ (this is implication \ref{item:ts3} $\Longrightarrow$ \ref{item:families} of Theorem~\ref{theo:classification}).
    \item Section \ref{sec:forb_dec} proves that every graph without any forbidden subgraph has a simple generalized series-parallel decomposition (this is implication \ref{item:families} $\Longrightarrow$ \ref{item:GSP} of Theorem~\ref{theo:classification}).
    \item Section \ref{sec:dec_ts} proves that every graph with a simple generalized series-parallel decomposition has topological inspection number at most $3$ (this is implication \ref{item:GSP} $\Longrightarrow$ \ref{item:ts3} of Theorem~\ref{theo:classification}).
    \item Section \ref{sec:conclusion} discusses several future directions of study, as well as some open problems.
\end{itemize}

\subsubsection*{Acknowledgments}

We are very grateful to the anonymous referees for carefully reading this paper and making many helpful suggestions.

\section{Preliminaries}\label{sec:prelim}

\subsection{Graph-theoretic notation}

Given a graph $G$, we denote by $V(G)$ and $E(G)$ its vertex and edge sets, respectively. If $S \subseteq V(G)$, $N_G(S)$ denotes the set of neighbors of $S$ in $G$. We use $G[S]$ to denote the subgraph of $G$ induced by $S$, and $G - S$ is the subgraph induced by $V(G) \setminus S$. We call a set $S$ of vertices \emphd{$G$-connected} if $G[S]$ is connected. For $v \in V(G)$ and $r \in \mathbb{N}$, $B_G(v, r)$ denotes the \emphd{ball} of radius $r$ around $v$ in $G$, i.e., the set of all the vertices of $G$ at distance at most $r$ from $v$. The \emphd{boundary} of $S$, denoted $\partial_G(S)$, is the set of all the vertices in $S$ that have a neighbor in $V(G) \setminus S$. Note that if $H$ is a subgraph of $G$ and $S \subseteq V(H)$, then $\partial_H(S) \subseteq \partial_G(S)$.

We say a set $S \subseteq V(G)$ \emphd{separates} a collection of nonempty sets of vertices $\mathcal{F} \subseteq \mathcal{P}(V(G) \setminus S)$ if every connected component of $G - S$ has nonempty intersection with at most one member of $\mathcal{F}$. We naturally extend this notion to separating vertices by identifying vertices with singleton sets.

\begin{lemma} If $G$ is $k$-connected and $S \subseteq V(G)$ satisfies $\abs{S} \ge k$, then $\abs{\partial_G(S)} \ge k$ or $S = V(G)$.\label{claim:connectedness_lemma}
\end{lemma}
\begin{proof} Note that $\partial_G(S)$ separates $S \setminus \partial_G(S)$ from $V(G) \setminus S$. If $\abs{\partial_G(S)} < \abs{S} < V(G)$, then these two sets are nonempty, and hence $\abs{\partial_G(S)} \ge k$. Otherwise, we have $S = V(G)$ or $\abs{\partial_G(S)} = \abs{S} \ge k$.
\end{proof}





\subsection{Searches}








Let us now introduce some notation useful for describing search strategies.

\begin{defn}[Searches] A $k$\emphd{-search} of length $\ell$ on a graph $G$ is a sequence $\mathcal{S} = \{S_{t} \subseteq V(G)\}_{t=1}^\ell$ for some $\ell \in \mathbb{N}$ where $\abs{S_t} \le k$ for each $t$. Note the convention of denoting searches by calligraphic letters and the individual sets by the same non-calligraphic letter with a subscript. We refer to $S_t$ (in the context of a search $\mathcal{S}$) as the \emphd{set of vertices searched at turn $t$}. We write $\abs{\mathcal{S}} = \ell$ to denote the length of $\mathcal{S}$. We sometimes refer to $k$ as the \emphd{search size}. We will occasionally refer to $k$-searches as \emphd{searches} when the specific search size is not important.

Given an \emphd{initial set of cleared vertices} $A \subseteq V(G)$ (which is usually taken to be $\0$), each $k$-search $\mathcal{S}$ of length $\ell$ has two associated sequences of sets of vertices $\{\PC_t(\mathcal{S}, A)\}_{t=1}^\ell$ and $\{\FC_t(\mathcal{S}, A)\}_{t=0}^\ell$, the \emphd{pre-cleared} and \emphd{fully cleared} vertices respectively after each turn. These are defined recursively by:
\begin{align*}
\FC_0(\mathcal{S}, A) &= A; \\
\PC_{t+1}(\mathcal{S}, A) &= \FC_t(\mathcal{S}, A) \cup S_{t+1}; \\
\FC_{t+1}(\mathcal{S}, A) &= \PC_{t+1}(\mathcal{S}, A)\setminus \partial_G(\PC_{t+1}(\mathcal{S}, A)).
\end{align*}
When $A = \0$, we will typically omit it from the notation, thus writing $\PC_t(\mathcal{S})$ and $\FC_t(\mathcal{S})$ for $\PC_t(\mathcal{S}, \0)$ and $\FC_t(\mathcal{S}, \0)$ respectively. We use $\FC_*(\mathcal{S}, A)$ (resp. $\PC_*(\mathcal{S}, A)$) to denote the set $\FC_\ell(\mathcal{S}, A)$ (resp. $\PC_\ell(\mathcal{S}, A)$) of fully cleared (resp.~pre-cleared) vertices after the final turn. 

We say a $k$-search $\mathcal{S}$ is \emphd{successful} if $\FC_*(\mathcal{S}) = V(G)$. 
\label{defn:search}
\end{defn}
\begin{prop}
The inspection number of a graph $G$ \ep{as given by Definition~\ref{defn:s(G)}} is equal to the minimum $k$ such that there exists a successful $k$-search on $G$. \label{prop:searching_clearing} 
\end{prop}
\begin{proof} We can consider the natural correspondence between strategies for the searcher and $k$-searches by identifying the sets of vertices inspected at turn $i$ with the $i$th term of the $k$-search. A simple inductive argument shows that given this correspondence, the set of all the possible locations of the intruder after the searcher's $i$th turn is precisely $V(G) \setminus \PC_i(\mathcal{S})$ and the set of all the possible locations of the intruder after the intruder's $i$th turn is precisely $V(G) \setminus \FC_i(\mathcal{S})$. The searcher's strategy is winning if and only if the set of possible locations is reduced to $\0$ by a finite sequence of moves, which is equivalent to the search being successful.
\end{proof}

It is clear that for any subgraph $H$ of $G$, we have $\s(H) \le \s(G)$, as any successful search on $G$, with the vertex sets appropriately restricted, is also a successful search on $H$.


\begin{exmp}\label{exp:complete}
$\s(K_n) = n$.
\end{exmp}
\begin{proof}
Suppose $\mathcal{S}$ is a $k$-search on $K_n$ for some $k < n$. Note that if $\FC_i(\mathcal{S}) = \0$, then $\PC_{i+1}(\mathcal{S}) \ne V(K_n)$, so $\FC_{i+1}(\mathcal{S}) = \0$ as well. Hence $\mathcal{S}$ cannot be successful. On the other hand, checking all the vertices on the first move gives a successful $n$-search, and hence $\s(K_n) = n$.
\end{proof}

\section{Some computations}\label{sec:comp}

\subsection{Sets with small boundary}

Our primary tool for proving lower bounds on the inspection number is the following proposition:
\begin{prop}
Let $G$ be a connected graph. If $\s(G) \le k$, then for each $1 \le i \le \abs{V(G)}$, there exists some $C \subseteq V(G)$ such that $\abs{\partial_G(C)} < k$ and $i - k < \abs{C} < i$.\label{prop:boundary}
\end{prop}
\begin{proof} Let $\mathcal{S}$ be a successful $k$-search on $G$. For each $1 \le i \le \abs{V(G)}$, consider the minimal $t$ such that $\abs{\PC_t(\mathcal{S})} \ge i$. We claim that $C = \PC_{t-1}(\mathcal{S})$ works. Indeed, $\abs{\PC_{t-1}(\mathcal{S})} < i$, so we can write
\begin{align*}
\abs{\PC_{t-1}(\mathcal{S})} &< i \le \abs{\PC_{t}(\mathcal{S})} \le \abs{\FC_{t-1}(\mathcal{S})} + k = \abs{\PC_{t-1}(\mathcal{S})} - \abs{\partial_G(\PC_{t-1}(\mathcal{S}))} + k. 
\end{align*}
Thus, we have $\abs{\partial_G(\PC_{t-1}(\mathcal{S}))} < k$. Furthermore,
\[i - k \le \abs{\PC_{t-1}(\mathcal{S})} - \abs{\partial_G(\PC_{t-1}(\mathcal{S}))} < \abs{\PC_{t-1}(\mathcal{S})} < i,\]
where we use that $\partial_G(\PC_{t-1}(\mathcal{S}))$ is nonempty. 
\end{proof}

\subsection{Grid graphs}

Here we compute the inspection number of rectangular grid graphs.

\begin{defn}[Grid Graphs]
An $n \times m$ \emphd{grid graph}, denoted $\Gamma_{n, m}$, comprises the vertices
\[\{(i, j) \in \mathbb{Z}^2 \mid 0 \le i < n, 0 \le j < m\},\]
with vertices $(a, b)$ and $(c, d)$ adjacent if and only if $\abs{a - c} + \abs{b - d} = 1$.

Visually the first coordinate indexes the row of a vertex, and the second one the column. 
Rows increase in index from bottom to top, and columns increase in index from left to right.
\end{defn}

\begin{theo}
For each $n, m \ge 2$, $\s(\Gamma_{n, m}) = \min\{n, m\} + 1$.\label{theo:gr}
\end{theo}

\begin{proof}\stepcounter{ForClaims} \renewcommand{\theForClaims}{\ref{theo:gr}}
Let $G = \Gamma_{n,m}$. Suppose, without loss of generality, that $n \ge m$, so our goal is thus to show that $\s(G) = m + 1$. For convenience, we relabel the vertices of $G$, writing $(a, b)$ as $v_{am + b}$ (see Fig.~\ref{fig:grid}).

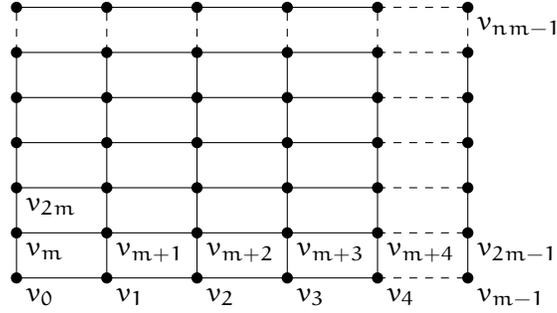
\begin{figure}[t]
\begin{center}
\begin{tikzpicture}[scale=1.2]
\draw (0, 3) -- (4, 3);
\draw (5, 0) -- (5, 2.5);
\draw[dashed] (4,3) -- (5, 3) -- (5, 2.5);
\foreach \x in {0, 1, ..., 4}
{
    \draw[dashed] (\x, 2.5) -- (\x, 3);
    \draw (\x, 0) -- (\x, 2.5);
}
\foreach \y in {0, 0.5, ..., 2.5}
{
    \draw[dashed] (4, \y) -- (5, \y);
    \draw (0, \y) -- (4, \y);
}

\foreach \y in {0, 0.5, ..., 3} {
    \foreach \x in {0, 1, ..., 5} {
        \node[fill=black,circle,inner sep=1.5pt] at (\x, \y){};
    }
}

\node[anchor=north west] at (0, 0){$v_0$};
\node[anchor=north west] at (1, 0){$v_1$};
\node[anchor=north west] at (2, 0){$v_2$};
\node[anchor=north west] at (3, 0){$v_3$};
\node[anchor=north west] at (4, 0){$v_4$};
\node[anchor=north west] at (5, 0){$v_{m-1}$};
\node[anchor=north west] at (0, 0.5){$v_m$};
\node[anchor=north west] at (1, 0.5){$v_{m+1}$};
\node[anchor=north west] at (2, 0.5){$v_{m + 2}$};
\node[anchor=north west] at (3, 0.5){$v_{m + 3}$};
\node[anchor=north west] at (4, 0.5){$v_{m + 4}$};
\node[anchor=north west] at (5, 0.5){$v_{2m-1}$};
\node[anchor=north west] at (0, 1){$v_{2m}$};
\node[anchor=north west] at (5, 3){$v_{nm - 1}$};
\end{tikzpicture}
\end{center}
\caption{The graph $\Gamma_{n,m}$.}\label{fig:grid}
\end{figure}

To prove that $\s(G) \le m + 1$, consider the $(m+1)$-search $\mathcal{S}$ defined via
\[S_t := \{v_{t+\ell} \mid 0 \le \ell \le m\} \quad \text{for all } 1 \le t \le m(n-1).\]
Note that for each $v_i$, its neighbors $v_j$ satisfy $j \le i + m$. If for some $t$, $\{v_i \mid i \le t\} \subseteq \FC_t(\mathcal{S})$, then we have
$\{v_i \mid i \le t + m + 1\} \subseteq \PC_{t+1}(\mathcal{S})$,
and hence
$\{v_i \mid i \le t + 1\} \subseteq \FC_{t+1}(\mathcal{S})$.
Induction thus shows that the search is successful, so $\s(G) \le m + 1$, as desired.

We now wish to show that no successful $m$-search exists. We shall analyze the sets with boundaries of size less than $m$ and then use Proposition~\ref{prop:boundary} to conclude the argument.

\begin{smallclaim} Any set $S \subseteq V(G)$ with $\abs{\partial_G(S)} \le m - 1$ satisfies \[\abs{S} \le \frac{m(m-1)}{2} \quad \text{or} \quad \abs{S} \ge mn - \frac{(m-2)(m-1)}{2}.\]\label{claim:grid_bdry}
\end{smallclaim}\vspace*{-15pt}
\begin{claimproof} Denote the $i$th column and $j$th row of $G$ by $c_i$ and $r_j$ respectively. For brevity, let a row or column be \emphd{empty/full} if it is disjoint from/contained in $S$ respectively. Note that each row or column is either empty, full, or contains a boundary vertex of $S$.

If there exist both an empty column and a full column, then every row is neither empty nor full and hence there are at least $n$ boundary vertices of $S$, a contradiction. Similarly, if there is neither an empty column nor a full column, then every column contains a boundary vertex of $S$ and so there are at least $m$ boundary vertices of $S$, a contradiction.

Now suppose there exists an empty column $c_k$ but no full columns. Take some other arbitrary column $c_i$ and consider the segments of each row between $c_i$ and $c_k$. For $\abs{c_i \cap S}$ of them, they contain a vertex in $c_i$, and none of them contain a vertex in $c_k$, so they contain at least $\abs{c_i \cap S}$ boundary vertices of $S$. Furthermore, every column not contained between $c_i$ and $c_k$ inclusive is disjoint from each of these segments. If $m(i)$ is the number of nonempty columns not contained between $c_i$ and $c_k$ inclusive, then 
$|\partial_G(S)| \ge \abs{c_i \cap S} + m(i)$,
as such columns are neither empty nor full and hence contain at least one boundary vertex of $S$ each. Thus,
\begin{equation}\label{eq:sum}
m - 1 - m(i)\ge \abs{c_i \cap S}.\tag{$*$}
\end{equation}
Now let $I = \{i \mid c_i \text{ is nonempty}\}$. Then
\[
    \sum_{i \in I} m(i) \ge \binom{\abs{I}}{2}, 
\]
as for each pair of distinct columns $c_i, c_j$, either $c_i$ is not between $c_j$ and $c_k$ or $c_j$ is not between $c_i$ and $c_k$. Summing \eqref{eq:sum} over $I$ gives
\[\abs{S} = \sum_{i \in I} \abs{c_i \cap S} \le \sum_{i \in I} (m - 1 - m(i)) \le \abs{I}(m-1) - \binom{\abs{I}}{2} = \frac{\abs{I}(2m - \abs{I} - 1)}{2} \le \frac{(m-1)m}{2}.\]

The remaining case where there exists a full column $c_k$ is analogous, except we get the boundary vertices on the segments for which $c_i$ does not contain a vertex of $S$, and thus the boundary vertex in $c_i$ is distinct from the boundary vertices of each segment. The analog of \eqref{eq:sum} is hence
\[|\partial_G(S)| \ge \abs{c_i \setminus S} + 1 + m'(i) \quad\implies\quad m - 2 - m'(i) \ge \abs{c_i \setminus S},\]
where $m'(i)$ is the number of non-full columns not contained between $c_i$ and $c_k$. Via the same bound for the sum of $m'(i)$, we get
\[
|V(G)\setminus S| \leq \sum_{i \in J} (m - 2 - m'(i)) \leq \frac{|J|(2m-|J|-3)}{2} \leq \frac{(m-2)(m-1)}{2},
\]
where $J = \{i \mid c_i \text{ is not full}\}$. Therefore, $\abs{S} \ge mn - \frac{(m-2)(m-1)}{2}$, as desired.
\end{claimproof}

Claim~\ref{claim:grid_bdry} implies that there does not exist a set $S$ with $|\partial_G(S)| < m$ and
\[\frac{m(m-1)}{2} < \abs{S} < \frac{m(m-1)}{2} + m = m^2 - \frac{(m-2)(m-1)}{2} \le mn - \frac{(m-2)(m-1)}{2},\]
contradicting Proposition \ref{prop:boundary} and completing the proof.
\end{proof}


\subsection{Trees}

In this subsection we show that trees can have arbitrarily large inspection number.

\begin{theo}
For every $k$, there is a tree $T$ of maximum degree $3$ with $\s(T) > k$.\label{theo:trees}
\end{theo}
\begin{proof}\stepcounter{ForClaims} \renewcommand{\theForClaims}{\ref{theo:trees}}
Take any $N > \log_2(k) + 2$ and consider the perfect binary tree $T$ of depth $L = 20kN$, i.e., a rooted tree in which every non-leaf vertex has exactly two children and the distance from the root to every leaf is equal to $L$. 
We wish to compute the possible sizes of subsets of $V(T)$ with fewer than $k$ boundary vertices and then apply Proposition~\ref{prop:boundary} to conclude that $\s(T) > k$.

Consider an arbitrary subset $S \subseteq V(T)$ with $\abs{\partial_T(S)} < k$. For each $0 \le i \le L$, let $S_i$ be the set of the vertices of $S$ at depth $L - i$ (so $S_0$ comprises the leaf vertices in $S$, while $S_L$ will be the root if it is in $S$). We define two subsets $A_i$, $B_i \subseteq S_i \cap \partial_T(S)$ by
\begin{align*}
    A_i &= \{x \in S_i \cap \partial_T(S) \mid \text{the parent of $x$ is not in $S$}\};\\
    B_i &= \{x \in S_i \cap \partial_T(S) \mid \text{not every child of $x$ is in $S$}\}.
\end{align*}
This gives us a recursive bound on the size of $S_i$ for $i < L$:
\[2\abs{S_{i+1}} + \abs{A_i} - 2\abs{B_{i+1}} \le \abs{S_i} \le 2\abs{S_{i+1}} + \abs{A_i} - \abs{B_{i+1}},\]
noting that every vertex in $B_{i+1}$ has either $0$ or $1$ child in $S$. Iterating this bound gives the following relation between the sizes of any two $S_i$s:
\begin{equation}\label{eq:i_ell}
2^\ell \abs{S_{i+\ell}} + \sum_{j=0}^{\ell - 1} 2^{j}\pr{ \abs{A_{i+j}} - 2 \abs{B_{i+j + 1}}}\le \abs{S_i} \le 2^\ell \abs{S_{i+\ell}} + \sum_{j=0}^{\ell - 1} 2^{j} \pr{\abs{A_{i+j}} - \abs{B_{i+j + 1}}}. \tag{$\spadesuit$}
\end{equation}
We denote by $R(i; \ell)$ the `slice' $\bigcup_{j=0}^\ell S_{i+j}$. Summing \eqref{eq:i_ell} gives:
\begin{align*}\label{eq:i_ell_sum}\pr{2^{\ell + 1} - 1}\abs{S_{\ell}} &+ \sum_{j=0}^{\ell - 1} \pr{2^{j+1}-1}\pr{\abs{A_{j}} - 2\abs{B_{j+1}}} \le \abs{R(0; \ell)} \\
&\le \pr{2^{\ell + 1} - 1}\abs{S_{\ell}} + \sum_{j=0}^{\ell - 1} \pr{2^{j+1}-1}\pr{\abs{A_{j}} - \abs{B_{j+1}}}.\tag{$\heartsuit$}
\end{align*}


By Proposition~\ref{prop:boundary}, if $\s(T) \leq k$, then there exists a set $S \subseteq V(T)$ with $\abs{\partial_T(S)} < k$ such that
\begin{equation}\label{eq:S_size}
    \abs{\abs{S} - \sum_{j = 1}^{4k+1} 2^{L - jN}} < k.\tag{$\diamondsuit$}
\end{equation}
For the remainder of the proof, we fix any such set $S$.

\begin{smallclaim}
    For each $1 \leq a \leq 2k$, there is $1 \leq j \leq 2N$ such that $\abs{A_{L-2aN-j}}$ or $\abs{B_{L-2aN-j}}$ is nonzero.\label{claim:tree_slice}
\end{smallclaim}
\begin{claimproof}
Note that $\abs{S} = \abs{R(0; L-2aN - 1)} + \abs{R(L-2aN; 2aN)}$. Since $R(L - 2aN; 2aN)$ is a subset of a perfect binary tree of depth $2aN$, 
\[\abs{R(L-2aN; 2aN)} \le 2^{2aN + 1} - 1 \le 2^{5kN}.\]
Suppose for the sake of contradiction that
\begin{equation}\label{eq:contradiction}
    \abs{A_{L-2aN -j}} = \abs{B_{L-2aN-j}} = 0 \quad \text{for all } 1 \le j \leq 2N.\tag{$\clubsuit$}
\end{equation}
Let $K = \pr{2^{L - 2aN} - 1}\abs{S_{L-2aN - 1}}$. Then, by \eqref{eq:i_ell_sum}, we have
\begin{align*}
    \abs{R(0; L-2aN-1)} &\geq K  + \sum_{j=0}^{L-2aN-2} \pr{2^{j+1}-1}\pr{\abs{A_{j}} - 2\abs{B_{j+1}}} \\
    [\text{by \eqref{eq:contradiction}}]\qquad&\geq K - \sum_{j=0}^{L- 2(a+1)N - 2} \pr{2^{j+1}-1}2\abs{B_{j+1}} \\
    [\text{since $\textstyle\sum\abs{B_i} \le \abs{\partial_T(S)} < k$}]\qquad&\geq K - 2^{L-2(a+1)N}k.
\end{align*}
Similarly,
\begin{align*}
    \abs{R(0; L-2aN-1)} &\leq K  + \sum_{j=0}^{L-2aN-2} \pr{2^{j+1}-1}\pr{\abs{A_{j}} - \abs{B_{j+1}}} \\
    [\text{by \eqref{eq:contradiction}}]\qquad&\leq K + \sum_{j=0}^{L - 2(a+1)N - 1} \pr{2^{j+1}-1}\abs{A_{j}} \\
    [\text{since $\textstyle\sum\abs{A_i} \le \abs{\partial_T(S)} < k$}]\qquad&\leq K + 2^{L-2(a+1)N}k.
\end{align*}
Since $\abs{S_{L-2aN - 1}} \leq 2^{2aN + 1} \leq 2^{5kN}$, $K$ is within $2^{5kN}$ of a multiple of $2^{L - 2aN}$. The maximum difference that $\abs{S}$ could attain from a multiple of $2^{L - 2aN}$ is therefore
\begin{align*}
    \abs{R(L-2aN; 2aN)} + \abs{S_{L-2aN-1}} + 2^{L - 2(a+1)N}k &\leq 2^{5kN} + 2^{5kN} + 2^{L - 2(a+1)N}k \\
    &< 2^{L - 2(a+1)N + \log_2(k) + 1}.
\end{align*}
On the other hand, by \eqref{eq:S_size}, $\abs{S}$ differs from a multiple of $2^{L - 2aN}$ at least by
\[
    2^{L-(2a+1)N} - k >  2^{L-(2a+1)N - 1}, 
\]
which is a contradiction since $N > \log_2(k) + 2$.
\end{claimproof}

Applying Claim~\ref{claim:tree_slice} to each $1 \leq a \leq 2k$ yields $2k$ distinct values $i$ such that $|A_i|$ or $|B_i|$ is nonzero, and therefore $\sum \abs{A_i} + \sum \abs{B_i} \ge 2k$. This is a contradiction since $\sum \abs{A_i}$, $\sum \abs{B_i} \leq \abs{\partial_T(S)} < k$.
\end{proof}

\section{Monotonic searches}\label{sec:mono}

\begin{defn}[Monotonic Searches]
A search $\mathcal{S}$ is called \emphd{monotonic} if $\FC_i(\mathcal{S}) \subseteq \FC_j(\mathcal{S})$ for every $i \le j$. Let $\s_m(G)$ be the minimum $k$ for which there exists a successful monotonic $k$-search on $G$. 
\end{defn}

Let us also recall the standard definition of the pathwidth of a graph. 

\begin{defn}[Pathwidth]
A \emphd{path decomposition} of a graph $G$ is a sequence of subsets $\{X_i \subseteq V(G)\}_{i=1}^\ell$ fulfilling the following two criteria:
\begin{itemize}
    \item For each edge $\{a, b\} \in E(G)$, there exists some $i$ such that $\{a, b\} \subseteq X_i$.
    \item For each $1 \le i < j < k \le \ell$ and $v \in V(G)$, if $v \in X_i$ and $v \in X_k$, then $v \in X_j$.
\end{itemize}
The \emphd{width} of such a path decomposition is $\max_{1 \le i \le \ell} \abs{X_i} - 1$ (if the decomposition is empty we set the width to be 0). The \emphd{pathwidth} of a graph, denoted $\pw(G)$, is the minimum width among all its path decompositions.
\end{defn}

\begin{prop} For any connected graph $G$, $\s_m(G) \le \pw(G) + 1$. \end{prop}
\begin{proof} The statement is trivial if $G$ has only one vertex, so we may assume that  
$G$ does not have isolated vertices, and thus every vertex must be included in at least one set in each path decomposition of $G$. 

Let $\mathcal{S} = \{S_i \subseteq V(G)\}_{i=1}^k$ be a path-decomposition of $G$ of width $\pw(G)$, i.e., $\abs{S_i} \le \pw(G) + 1$ for each $i$, for each edge of $G$ there exists $i$ such that $S_i$ contains both its endpoints, and for all $i \le j \le k$, $S_i \cap S_k \subseteq S_j$. Note that for each $v \in V(G)$, the set of indices $i$ for which $v \in S_i$ forms a consecutive sequence of integers. We claim that $\mathcal{S}$ is a successful monotonic ($\pw(G) + 1$)-search of $G$.

First, we show by induction that for each $t$, if $v \not \in S_i$ for all $i > t$, then $v \in \FC_t(\mathcal{S})$. Suppose the claim holds for all $t < t_0$, and suppose $v \not \in S_i$ for every $i > t_0$. Consider any $u \in N_G(v)$, and let $t_u$ be maximal such that $u \in S_{t_u}$. If $t_u \ge t_0$, then $u \in S_{t_0}$ as $\{u, v\} \subseteq S_i$ for some $i \le t_0$, and so $u \in \PC_{t_0}(\mathcal{S})$. Otherwise, $u \not \in S_i$ for each $i > t_0 - 1$, so by the induction hypothesis $u \in \FC_{t_0 - 1}(\mathcal{S}) \subseteq \PC_{t_0}(\mathcal{S})$. A similar argument gives that $v \in \PC_{t_0}(\mathcal{S})$. Hence $v$ and each of its neighbors are in $\PC_{t_0}(\mathcal{S})$, so $v \in \FC_{t_0}(\mathcal{S})$, completing the induction. This shows that the search is successful.

Now we show that the search is monotonic. Suppose that $v \in \FC_t(\mathcal{S}) \setminus \FC_{t+1}(\mathcal{S})$. Then there exists some $u \in N_G(v)$ such that $u \in \PC_t(\mathcal{S}) \setminus \PC_{t+1}(\mathcal{S})$. Since $u \in \PC_t(\mathcal{S})$, there is some $t_0 \leq t$ such that $u \in S_{t_0}$. On the other hand, since $u \not \in \PC_{t+1}(\mathcal{S})$, we have $u \not \in S_{t+1}$. This implies that $u \not \in S_{t_1}$ for all $t_1 \geq t+1$, and hence $u \not \in \FC_\ast(\mathcal{S})$, contradicting the fact that the search $\mathcal{S}$ is successful.
\end{proof}

\begin{prop} For any connected graph $G$, $\s_m(G) \ge \pw(G) + 1$.
\end{prop}
\begin{proof} Let $\mathcal{S}'$ be a successful monotonic search. We can modify each $S'_i$ to $S_i$ by removing the vertices which are not in $\FC_{i}(\mathcal{S}') \cup N_G(\FC_i(\mathcal{S}'))$, as well as removing the vertices in $\FC_{i-1}(\mathcal{S}')$. Note that $\FC_i(\mathcal{S}) = \FC_i(\mathcal{S}')$ for each $i$, so $\mathcal{S}$ is still successful and monotonic. 

For each pair of adjacent vertices $\{v, w\}$, consider the minimal $t$ for which after turn $t$, either $v$ or $w$ is in $\FC_t(\mathcal{S})$. Since $v$, $w \not\in \FC_{t-1}(\mathcal{S})$ but $\{v, w\} \subseteq \PC_t(\mathcal{S})$, we must have $\{v, w\} \subseteq S_t$. Hence for each such pair there exists some $S_t$ containing both. Now suppose $v \in S_i$, $S_j$ for some $i < j$. By the modification we performed, $v \not \in \FC_{j-1}(\mathcal{S})$, and so by monotonicity $v \not \in \FC_t(\mathcal{S})$ for any $t < j$. In particular, $v \not \in \FC_i(\mathcal{S})$, yet $v \in S_i$. Hence, again by our modification, $v$ has a neighbor in $\FC_i(\mathcal{S})$. This neighbor must remain fully cleared, so $v \in \PC_{t}(\mathcal{S})$ for each $t \ge i$. Hence $v \in S_t$ for each $i < t < j$, and so $\mathcal{S}$ is a path-decomposition of $G$ of width $\max \abs{S_i} - 1$, proving the claim.
\end{proof}

These two propositions combined yield the following result: 

\begin{theo}
    For every connected graph $G$, $\s_m(G) = \pw(G) + 1$.\label{theo:mono}
\end{theo}

Note that this implies that $\s(G) \le \pw(G) + 1$. However, this inequality can be strict; in fact, we will see that the difference $\s_m(G) - \s(G)$ can be arbitrarily large.

\section{The topological inspection number}\label{sec:ts}

\begin{defn}[Subdivisions]
Given a graph $G$, an \emphd{edge subdivision} on an edge $\{v, w\} \in E(G)$ is the operation where we add a new vertex $u$, and replace the edge with a pair of edges $\{v, u\}$ and $\{u, w\}$. A \emphd{subdivision} of $G$ is a graph which can be obtained from $G$ by a (possibly empty) finite sequence of edge subdivisions.
\end{defn}

\begin{exmp}
There exist a graph $G$ and a subdivision $H$ of $G$ such that $\s(H) \ne \s(G)$.\label{ex:k4subdiv}
\end{exmp}
\begin{proof}
Consider the subdivision of $K_4$ shown in Fig.~\ref{fig:k4sub}. 

\begin{figure}[h]
\begin{center}
\begin{tikzpicture}[scale=0.1]
\node [draw,fill=white,circle,minimum size=12pt,outer sep=0pt,inner sep=0pt] (D) at (50,-14) {$D$};
\node [draw,fill=white,circle,minimum size=12pt,outer sep=0pt,inner sep=0pt] (B) at (42,-28) {$B$};
\node [draw,fill=white,circle,minimum size=12pt,outer sep=0pt,inner sep=0pt] (C) at (58,-28) {$C$};
\node [draw,fill=white,circle,minimum size=12pt,outer sep=0pt,inner sep=0pt] (A) at (50,-42) {$A$};
\node [draw,fill=white,circle,minimum size=12pt,outer sep=0pt,inner sep=0pt] (I4) at (30,-14) {$I_4$};
\node [draw,fill=white,circle,minimum size=12pt,outer sep=0pt,inner sep=0pt] (I3) at (22,-21) {$I_3$};
\node [draw,fill=white,circle,minimum size=12pt,outer sep=0pt,inner sep=0pt] (I1) at (30,-42) {$I_1$};
\node [draw,fill=white,circle,minimum size=12pt,outer sep=0pt,inner sep=0pt] (I2) at (22, -33) {$I_2$};
\draw (A) -- (B) -- (C) -- (D) -- (I4) -- (I3) -- (I2) -- (I1) -- (A) -- (C) (B) -- (D);
\end{tikzpicture}
\end{center}
\caption{A subdivision of $K_4$ with inspection number $3$.}\label{fig:k4sub}
\end{figure}
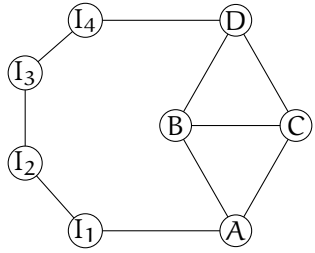

The following table lists a successful 3-search $\mathcal{S}$ of this graph:

\begin{center}
\begin{tabular}{c|c|c|c}
$i$ & $S_i$ & $\PC_i(\mathcal{S})$ & $\FC_i(\mathcal{S})$\\
\hline
1 & $\{A, I_1, I_2\}$ & $\{A, I_1, I_2\}$ & $\{I_1\}$ \\
2 & $\{A, I_2, I_3\}$ & $\{A, I_1, I_2, I_3\}$ & $\{I_1, I_2\}$ \\
3 & $\{A, I_3, I_4\}$ & $\{A, I_1, I_2, I_3, I_4\}$ & $\{I_1, I_2, I_3\}$ \\
4 & $\{A, I_4, D\}$ & $\{A, I_1, I_2, I_3, I_4, D\}$ & $\{I_1, I_2, I_3, I_4\}$  \\
5 & $\{A, B, C\}$ & $\{A, I_1, I_2, I_3, I_4, D, B, C\}$ & $\{A, I_1, I_2, I_3\}$ \\
6 & $\{B, C, D\}$ & $\{A, I_1, I_2, I_3, D, B, C\}$ &  $\{A, I_1, I_2, B, C\}$ \\
7 & $\{D, I_3, I_4\}$ & $V(G)$ & $V(G)$ \\
\end{tabular}
\end{center}

\noindent However, $\s(K_4) = 4$ by Example~\ref{exp:complete}, proving the claim.
\end{proof}

As explained in the introduction, we wish to investigate the ``limiting'' behavior of the inspection number under edge subdivisions, captured precisely in the following notion:

\begin{defn}[Topological inspection number]
The \emphd{topological inspection number} of a graph $G$, denoted $\s_t(G)$, is the minimum $k$ such that for every subdivision $H$ of $G$, there is a further subdivision $H'$ of $H$ with $\s(H') \leq k$.

Equivalently, $\s_t(G)$ is the minimum $k$ such that for all $\ell$, $G$ has a subdivision $H$ with $\s(H) \le k$ in which every edge of $G$ is subdivided at least $\ell$ times.\label{defn:top_search}
\end{defn}

It is clear that $\s_t(H) = \s_t(G)$ for every subdivision $H$ of $G$. We observe that $\s_t(G)$ is finite:

\begin{prop}
    $\s_t(G) \le \abs{V(G)} + 2$.
\end{prop}
\begin{proof}
    Let $H$ be a subdivision of $G$. Enumerate the vertices of $V(G)$ as $v_1$, \ldots, $v_k$, and the edges in $E(G)$ as $e_1$, \ldots, $e_\ell$. Then each vertex in $V(H)$ is either in $V(G)$, or is the result of some series of subdivisions on an edge $e_i$. We may order these latter vertices in order of the index of the subdivided edge they lie on, and break ties by distance to the lower-indexed endpoint of the subdivided edge (i.e., if $a$, $b \in V(H)$ lie on subdivided edge $e_i$ with endpoints $v_c$, $v_d$ with $c < d$, then whichever of $a$ or $b$ lies closer to $v_c$ will come first in the order). Using this order we can enumerate the vertices in $V(H) \setminus V(G)$ as $w_1$, \ldots, $w_m$. Note that each $w_i$'s neighbors are contained in $V(G) \cup \{w_{i-1}, w_{i+1}\}$.
    
    Now consider the sequence of sets given by $S_t = V(G) \cup \{w_t, w_{t+1}\}$. Each $w_i$ appears in either a single set or two consecutive sets while each $v_i$ appears in every set. Furthermore, note that every pair of vertices with at least one vertex in $V(G)$ appears at least once, and every pair of vertices of the form $\{w_i, w_{i+1}\}$ appears once as well. Since every pair of neighbors must have one of these forms, every pair of neighbors appears together in a set of this sequence. This sequence thus witnesses $\pw(H) \le \abs{V(G)} + 1$, and hence, by Theorem~\ref{theo:mono}, $\s(H) \le \abs{V(G)} + 2$. Since $H$ is a subdivision of itself, we have $\s_t(G) \le \abs{V(G)} + 2$.
\end{proof}


Evidently, the only connected graph $G$ with $\s_t(G) = 1$ is the one-vertex graph $K_1$. Next we characterize the graphs $G$ with $\s_t(G) = 2$:

\begin{prop}
If $G$ is a connected graph with $\s_t(G) \le 2$, then $G$ is a path.\label{prop:paths}
\end{prop}
\begin{proof} Suppose $G$ contains as a subgraph a cycle $C$. We wish to show that $\s(G) \ge 3$. Suppose otherwise and take a successful $2$-search $\mathcal{S}$ of $G$. Let $v \in V(C)$ be such that the first index $i$ with $v \in \FC_i(\mathcal{S})$ is minimal among all $v \in V(C)$. This implies that $v$ and its two neighbors (which are distinct) in $C$ are all contained in $\PC_i(\mathcal{S})$. However, by minimality of $i$, none of these vertices are contained in $\FC_{i-1}(\mathcal{S})$, and so all three must be contained in $S_i$, a contradiction. Since subdivisions of cyclic graphs remain cyclic, we must have $G$ acyclic.

Now suppose, toward a contradiction, that $G$ contains a vertex $v_0$ with $\deg(v_0) \ge 3$. Let $G'$ be an arbitrary subdivision of $G$ where each edge incident to $v_0$ is subdivided at least once. Let $v_1$, $v_2$, $v_3 \in N_{G'}(v_0)$ be distinct. Each of $v_1$, $v_2$, $v_3$ has a unique neighbor apart from $v_0$, and these neighbors are pairwise distinct. Let these neighbors be $w_1$, $w_2$, $w_3$ respectively. Let $H$ be the subgraph of $G'$ induced on $\{v_0, v_1, v_2, v_3, w_1, w_2, w_3\}$. It is straightforward to check that every subset $S \subseteq V(H)$ of size $|S| = 4$ has $|\partial_H(S)| \geq 2$. Hence, by Proposition \ref{prop:boundary} applied with $i = 5$,  
%
%
%
$\s(G') \geq \s(H) \ge 3$. Therefore, $\s_t(G) \ge 3$, which is a contradiction.

Hence, $G$ must be an acyclic graph of maximum degree at most $2$, i.e., a path.
\end{proof}

Since the inspection number of any path is at most $2$ and subdivisions of paths are themselves paths, Proposition~\ref{prop:paths} implies that paths are precisely the connected graphs with topological inspection number at most $2$. In the remainder of this paper we establish a characterization of the connected graphs $G$ with $\s_t(G) = 3$, which turns out to be significantly more complicated.

\section{Statement of the classification theorem}\label{sec:defns}


\begin{defn}[Bipaths] Let a \emphd{bipath} be the union of a pair of edge-disjoint paths $P_1$, $P_2$ which share at least three vertices, two of which are their endpoints, such that the shared vertices appear on $P_1$ and $P_2$ in the same order. 
(See Fig.~\ref{fig:bipath}.)

We call the shared vertices the \emphd{primary vertices} of the bipath and the number of primary vertices the \emphd{order} of the bipath. The shared endpoints of $P_1$ and $P_2$ will be referred to as the \emphd{endpoints} of the bipath. 

Note that each pair of consecutive primary vertices has a pair of internally vertex-disjoint paths between them contained in the bipath. These will be called \emphd{primary paths}.

\begin{figure}[h]
\begin{center}
\begin{tikzpicture}[scale=0.8,x=0.75pt,y=0.75pt,yscale=-1,xscale=1,scale=0.8]

\draw  [fill={rgb, 255:red, 0; green, 0; blue, 0 }  ,fill opacity=1 ] (100,125.72) .. controls (100,122.56) and (102.56,120) .. (105.72,120) .. controls (108.87,120) and (111.43,122.56) .. (111.43,125.72) .. controls (111.43,128.87) and (108.87,131.43) .. (105.72,131.43) .. controls (102.56,131.43) and (100,128.87) .. (100,125.72) -- cycle ;
\draw  [fill={rgb, 255:red, 0; green, 0; blue, 0 }  ,fill opacity=1 ] (206,126.72) .. controls (206,123.56) and (208.56,121) .. (211.72,121) .. controls (214.87,121) and (217.43,123.56) .. (217.43,126.72) .. controls (217.43,129.87) and (214.87,132.43) .. (211.72,132.43) .. controls (208.56,132.43) and (206,129.87) .. (206,126.72) -- cycle ;
\draw  [fill={rgb, 255:red, 0; green, 0; blue, 0 }  ,fill opacity=1 ] (271,126.72) .. controls (271,123.56) and (273.56,121) .. (276.72,121) .. controls (279.87,121) and (282.43,123.56) .. (282.43,126.72) .. controls (282.43,129.87) and (279.87,132.43) .. (276.72,132.43) .. controls (273.56,132.43) and (271,129.87) .. (271,126.72) -- cycle ;
\draw  [fill={rgb, 255:red, 0; green, 0; blue, 0 }  ,fill opacity=1 ] (367,127.72) .. controls (367,124.56) and (369.56,122) .. (372.72,122) .. controls (375.87,122) and (378.43,124.56) .. (378.43,127.72) .. controls (378.43,130.87) and (375.87,133.43) .. (372.72,133.43) .. controls (369.56,133.43) and (367,130.87) .. (367,127.72) -- cycle ;
\draw    (105.72,125.72) .. controls (129.15,89.18) and (192.43,93.47) .. (211.72,126.72) ;
\draw    (211.72,126.72) .. controls (228.43,97.47) and (273.43,106.47) .. (276.72,126.72) ;
\draw    (276.72,126.72) .. controls (293.43,97.47) and (349.43,92.47) .. (372.72,127.72) ;
\draw    (105.72,125.72) .. controls (135.43,157.47) and (186.43,159.47) .. (211.72,126.72) ;
\draw    (276.72,126.72) .. controls (306.43,158.47) and (347.43,160.47) .. (372.72,127.72) ;
\draw    (211.72,126.72) .. controls (226.43,152.47) and (261.43,157.47) .. (276.72,126.72) ;

\draw (96,100.4) node [anchor=north west][inner sep=0.75pt]    {$a$};
\draw (209,99.4) node [anchor=north west][inner sep=0.75pt]    {$b$};
\draw (274,98.4) node [anchor=north west][inner sep=0.75pt]    {$c$};
\draw (368,98.4) node [anchor=north west][inner sep=0.75pt]    {$d$};

\end{tikzpicture}
\caption{A bipath of order 4, with primary vertices $a$, $b$, $c$, $d$.}\label{fig:bipath}
\end{center}
\end{figure}
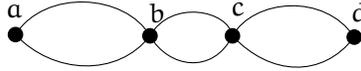
\end{defn}

From here, we construct three families of graphs to be used in our theorem:

\begin{defn}[Families $\mathcal{F}_1$, $\mathcal{F}_2$, $\mathcal{F}_3$]
Let $\mathcal{F}_1$ comprise all subdivisions of $K_4$.

Let $\mathcal{F}_2$ comprise the graphs consisting of 3 bipaths which share their endpoints but are otherwise pairwise vertex-disjoint.

Let $\mathcal{F}_3$ comprise the graphs consisting of 4 bipaths $B_{a}$, $B_{b}$, $B_{c}$, $B_{d}$ and two paths $P_1$, $P_2$, such that $B_{a}$, $B_{b}$ share endpoints $v_1$ and $v_2$, $B_c$, $B_d$ share endpoints $v_3$ and $v_4$, but the bipaths are otherwise pairwise vertex-disjoint. Furthermore, $P_1$ has endpoints $v_1$, $v_3$, and $P_2$ has endpoints $v_2$, $v_4$, but these paths have no other vertices in common with the bipaths (however, $P_1$ and $P_2$ may intersect each other).

Examples of graphs from the families $\mathcal{F}_1$, $\mathcal{F}_2$, $\mathcal{F}_3$ are shown in Fig.~\ref{fig:representatives1}.\label{defn:families}
\end{defn}

\begin{figure}[H]
\begin{center}
\begin{tikzpicture}[x=0.75pt,y=0.75pt,yscale=-1,xscale=1]
\draw    (105.31,142.67) .. controls (109.14,116.67) and (113.52,113.67) .. (125.57,103.67) ;
\draw    (145.83,148.67) .. controls (143.64,120.67) and (137.07,111.67) .. (125.57,103.67) ;
\draw    (125.57,187.67) .. controls (135.43,179.67) and (143.64,165.67) .. (145.83,148.67) ;
\draw    (125.57,187.67) .. controls (110.79,181.67) and (107.5,171.67) .. (105.31,142.67) ;
\draw    (105.31,142.67) .. controls (117.36,155.67) and (130.5,158.67) .. (145.83,148.67) ;
\draw    (125.57,187.67) .. controls (57.83,191.67) and (58.83,102.67) .. (125.57,103.67) ;
\draw    (215.38,144.4) .. controls (200.36,130.16) and (215.38,104.95) .. (238.99,101.67) ;
\draw    (216.45,145.5) .. controls (218.6,132.35) and (223.97,113.72) .. (238.99,101.67) ;
\draw    (236.84,187.13) .. controls (215.38,187.13) and (194.99,160.84) .. (216.45,145.5) ;
\draw    (236.84,187.13) .. controls (223.97,172.89) and (217.53,160.84) .. (216.45,145.5) ;
\draw    (238.99,143.3) .. controls (227.19,133.44) and (230.41,111.53) .. (238.99,101.67) ;
\draw    (236.84,187.13) .. controls (225.04,177.27) and (230.41,153.17) .. (238.99,143.3) ;
\draw    (238.99,143.3) .. controls (249.72,133.44) and (251.87,114.82) .. (238.99,101.67) ;
\draw    (236.84,187.13) .. controls (250.8,177.27) and (251.87,154.26) .. (238.99,143.3) ;
\draw    (265.43,144.4) .. controls (274.1,124.09) and (265.43,104.95) .. (236.84,101.67) ;
\draw    (264.13,145.5) .. controls (261.53,132.35) and (255.04,113.72) .. (236.84,101.67) ;
\draw    (236.84,187.13) .. controls (262.83,187.13) and (278.39,162.44) .. (264.13,145.5) ;
\draw    (236.84,187.13) .. controls (252.44,172.89) and (262.83,160.84) .. (264.13,145.5) ;
\draw    (324.67,131.01) .. controls (323.78,121.13) and (327.72,111.13) .. (341.98,106.67) ;
\draw    (324.67,131.01) .. controls (334.28,128.13) and (339.53,118.13) .. (343.8,106.67) ;
\draw    (341.98,154.13) .. controls (326.4,149.13) and (325.09,143.13) .. (324.67,131.01) ;
\draw    (341.98,154.13) .. controls (342.16,141.13) and (335.59,134.13) .. (324.67,131.01) ;
\draw    (360.05,131.01) .. controls (360.71,117.74) and (353.55,109.13) .. (342.41,106.67) ;
\draw    (360.05,131.01) .. controls (348.55,125.13) and (345.55,117.13) .. (342.41,106.67) ;
\draw    (342.41,154.13) .. controls (355.55,151.13) and (360.55,144.13) .. (360.05,131.01) ;
\draw    (342.41,154.13) .. controls (347.55,141.13) and (350.55,137.13) .. (360.05,131.01) ;
\draw    (381.67,158.01) .. controls (380.78,148.13) and (384.72,138.13) .. (398.98,133.67) ;
\draw    (381.67,158.01) .. controls (391.28,155.13) and (396.53,145.13) .. (400.8,133.67) ;
\draw    (398.98,181.13) .. controls (383.4,176.13) and (382.09,170.13) .. (381.67,158.01) ;
\draw    (398.98,181.13) .. controls (399.16,168.13) and (392.59,161.13) .. (381.67,158.01) ;
\draw    (417.05,158.01) .. controls (417.71,144.74) and (410.55,136.13) .. (399.41,133.67) ;
\draw    (417.05,158.01) .. controls (405.55,152.13) and (402.55,144.13) .. (399.41,133.67) ;
\draw    (399.41,181.13) .. controls (412.55,178.13) and (417.55,171.13) .. (417.05,158.01) ;
\draw    (399.41,181.13) .. controls (404.55,168.13) and (407.55,164.13) .. (417.05,158.01) ;
\draw    (343.8,106.67) .. controls (383.8,76.67) and (359.41,211.13) .. (399.41,181.13) ;
\draw    (342.41,154.13) .. controls (357.55,183.13) and (374.55,113.13) .. (400.8,133.67) ;

\draw (104,197.4) node [anchor=north west][inner sep=0.75pt]    {$\mathcal{F}_{1}$};
\draw (232,201.4) node [anchor=north west][inner sep=0.75pt]    {$\mathcal{F}_{2}$};
\draw (366,202.4) node [anchor=north west][inner sep=0.75pt]    {$\mathcal{F}_{3}$};
\end{tikzpicture}
\caption{Representative members of the families $\mathcal{F}_1$, $\mathcal{F}_2$, $\mathcal{F}_3$. Each curved segment represents a path which may be subdivided arbitrarily many times.}\label{fig:representatives1}
\end{center}
\end{figure}
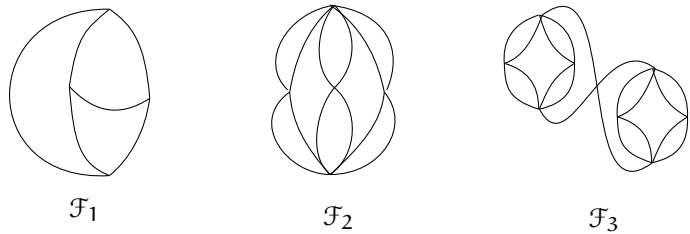

\begin{defn}[Generalized Series-Parallel Graphs]
A \emphd{graph with terminals} $\gwt{G}$ is an ordered triple of the form $(G, u, v)$ where $G$ is a graph and $u, v$ are distinct vertices in $V(G)$, called the \emphd{terminals} of $\gwt{G}$. Note the convention of denoting graphs with terminals by the boldface version of the letter used to denote the underlying graph. The \emphd{generalized series-parallel graphs}, or \emphd{GSP graphs} for short, are the minimal collection of graphs with terminals that contains all single-edge graphs with their two vertices as terminals, and is closed under the following operations:
\begin{itemize}[wide]
    \item The series operation $\circ_s$ can be performed on two GSP graphs $\gwt{G} = (G, u, v)$ and $\gwt{H} = (H, v, w)$ such that $V(G) \cap V(H) = \{v\}$. The result of this operation is the new GSP graph $\gwt{G} \circ_s \gwt{H} = (G \cup H, u, w)$. 
    \item The parallel operation $\circ_p$ can be performed on two GSP graphs $\gwt{G} = (G, u, v)$ and $\gwt{H} = (H, u, v)$ such that $V(G) \cap V(H) = \{u, v\}$. The result of this operation is the new GSP graph $\gwt{G} \circ_p \gwt{H} = (G \cup H, u, v)$. 
    \item The first branch operation $\circ_b$ can be performed on two GSP graphs $\gwt{G} = (G, u, v)$, $\gwt{H} = (H, u, w)$ such that $V(G) \cap V(H) = \{u\}$. The result of this operation is the new GSP graph $\gwt{G} \circ_b \gwt{H} = (G \cup H, u, v)$. 
    \item The second branch operation $\circ_{b'}$ can be performed on two GSP graphs $\gwt{G} = (G, u, v)$, $\gwt{H} = (H, v, w)$ such that $V(G) \cap V(H) = \{v\}$. The result of this operation is the new GSP graph $\gwt{G} \circ_{b'} \gwt{H} = (G \cup H, u, v)$. 
\end{itemize}
A graph $G$ is \emphd{GSP} if there is some choice of terminals $u$, $v \in V(G)$ such that $(G,u,v)$ is a GSP graph.\label{defn:GSP}
\end{defn}

The way a GSP graph is built up from individual edges is encoded in its GSP decomposition:

\begin{defn}[GSP Decompositions]
Given a GSP graph $\gwt{G} = (G, u, v)$, a \emphd{GSP decomposition} of $\gwt{G}$ is a rooted binary tree $T$ whose vertices are GSP graphs satisfying the following recursive conditions: 
\begin{itemize}[wide]
    \item The root node of $T$ is $\gwt{G}$.
    \item If $G$ comprises a single edge, then $T$ comprises a single node. 
    \item Otherwise, there exist two GSP graphs $\gwt{H}$ and $\gwt{K}$ such that $\gwt{G} = \gwt{H} \circ \gwt{K}$ for some operation $\circ \in \{\circ_s, \circ_p, \circ_b, \circ_{b'}\}$, and $T$ comprises GSP decompositions $T_H, T_K$ of $\gwt{H}$ and $\gwt{K}$, joined to the root node. 
\end{itemize}

\begin{figure}[b]
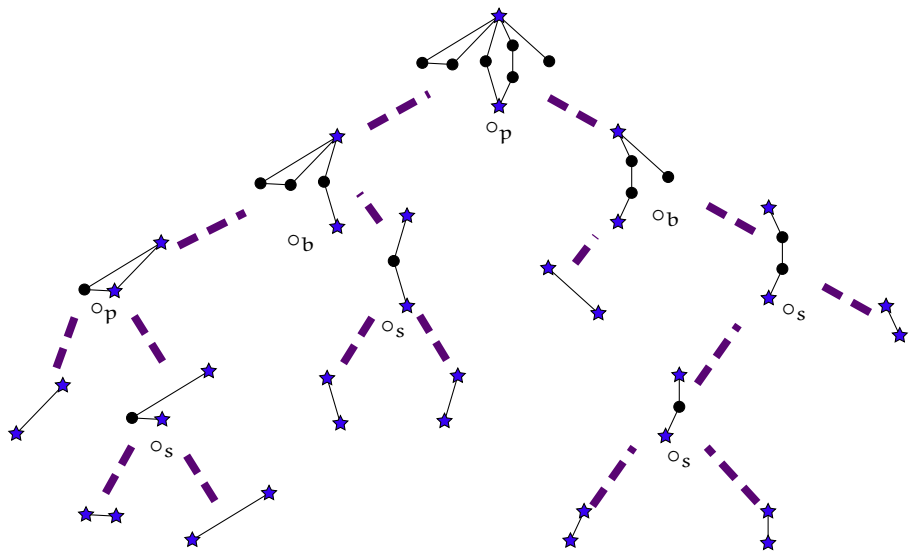

\begin{center}


\caption{A GSP decomposition $T$. The starred vertices are terminals of their respective graphs. Each node $\gwt{H}$ is labelled with $\circ_{T, \gwt{H}}$.\label{fig:gspdecomp}}
\end{center}
\end{figure}
It is clear that every GSP graph $\gwt{G}$ admits a GSP decomposition (not necessarily unique). Fig.~\ref{fig:gspdecomp} shows an example of a GSP decomposition.


Given a GSP decomposition $T$ of $\gwt{G}$ and $\gwt{H} \in V(T)$, the \emphd{induced decomposition} of $\gwt{H}$ is the subtree of $T$ rooted at $\gwt{H}$. If $\gwt{H}$ is not a leaf of $T$, then $\circ_{T, \gwt{H}}$ (or simply $\circ_{T,H}$, where $H$ is the underlying graph of $\gwt{H}$) denotes the operation in  $\{\circ_s, \circ_p, \circ_b, \circ_{b'}\}$ by which $\gwt{H}$ is obtained from its children.

We call $T$ a \emphd{series-parallel decomposition} if $\circ_{T, \gwt{H}} \in \{\circ_s, \circ_p\}$ for all $\gwt{H} \in V(T)$.
\label{defn:GSP_dec}
\end{defn}

Next we define the subclass of simple GSP decompositions.

\begin{defn}[Simple GSP Decompositions]
A graph with terminals is \emphd{bridged} if it contains an edge such that every path between its terminals includes that edge (i.e. a bridge separating its terminals).

We define the \emphd{complexity} $c(T)$ of a GSP decomposition $T$ of $\gwt{G}$ recursively as follows:
\begin{itemize}[wide]
    \item If $\circ_{T, \gwt{G}} = \circ_p$, $c(T)$ is the sum of the complexities of the two subtrees rooted at the children of $\gwt{G}$.
    \item If $\circ_{T, \gwt{G}} \in \{\circ_b, \circ_{b'}\}$, $c(T)$ is the complexity of the child of $\gwt{G}$ that shares its terminals with $\gwt{G}$.
    \item Otherwise, the complexity of the decomposition is $0$ if $\gwt{G}$ is bridged and $1$ otherwise.
\end{itemize}
A GSP decomposition $T$ is \emphd{simple} if the complexity of every subtree of $T$ is at most $1$. Otherwise, it is \emphd{complex}.\label{defn:simple}
\end{defn}

Finally, we have all the terminology needed to state our result:

\begin{theo} Let $G$ be a connected graph with $\abs{V(G)} \ge 2$. The following are equivalent:
\begin{enumerate}[label=\ep{\normalfont\arabic*}]
    \item $\s_t(G) \le 3$.
    \item $G$ does not contain as a subgraph any element of $\mathcal{F}_1, \mathcal{F}_2$, or $\mathcal{F}_3$.
    \item $G$ admits a simple GSP decomposition.
\end{enumerate}\label{theo:main}
\end{theo}
The remainder of the paper is devoted to proving Theorem~\ref{theo:main}.

\section{Forbidden graphs have topological inspection number at least 4}\label{sec:forb_lower}


\subsection{A technical lemma}

We begin by establishing a slightly technical result that will allow us to control sets of vertices with boundary of size at most $2$ 
in our subsequent analysis.

\begin{lemma} Suppose $G$ is a graph and $P_1, \ldots, P_k \subseteq V(G)$ are $G$-connected. For each $i$, let $Q_i = P_i \cup N_G(P_i)$, and let $P = \bigcup_i P_i$. Assume the following hypotheses: 
\begin{enumerate}[label=\ep{\upshape\textsf{H}\arabic*}]
\item\label{item:H1} $\bigcup_{i=0}^k Q_i = V(G)$.
\item\label{item:H2} For each $i \ne j$, $P_i \cap Q_j = \0$.
\item\label{item:H3} For each $i$, $G - P_i$ is connected and $G - P_i - v$ is connected for every $v \in V(G) \setminus P$. 
\item\label{item:H4} For each cut-set $\{a, b\} \subseteq P$ of $G$
, there exists some $i$ such that $\{a, b\} \subseteq P_i$.
\end{enumerate}
Then if $S \subseteq V(G)$ is such that $\abs{\partial_G(S)} \le 2$, at least one of the following conclusions holds:
\begin{enumerate}[label=\ep{\upshape\textsf{C}\arabic*}]
    \item\label{item:C1} $\abs{S} \le 2;$
    \item\label{item:C2} $S \subseteq P_i \cup \{v\}$ for some $i$ and $v \in V(G) \setminus P$;
    \item\label{item:C3} $S \supseteq V(G) \setminus (P_i \cup A)$ for some $i$ and $A \subseteq V(G) \setminus P$ with $\abs{A} \le 1$;
    \item\label{item:C4} $S = \bigcup_{i \in I}P_i \cup A$ for some $I \subseteq \{1, \ldots, k\}$ and $A \subseteq V(G) \setminus P$.
\end{enumerate}\label{claim:technical_lemma}
\end{lemma}
\begin{proof} Suppose $S \subseteq V(G)$ is a set with $\abs{\partial_G(S)} \le 2$. If $\abs{S} \le 2$, then \ref{item:C1} holds. Thus, we may assume $\abs{S} > 2$. Let $H = G - \partial_G(S)$. Then each connected component of $H$ is either contained in $S$ or is disjoint from $S$. Note that $V(H) \cap S \ne \0$ since $\abs{S} > \abs{\partial_G(S)}$.

If $H$ is connected then $S = V(G)$ and \ref{item:C4} holds with $I = \{1,\ldots,k\}$ and $A = V(G) \setminus P$. Thus, we may assume $H$ is disconnected.

Suppose $\partial_G(S) \subseteq V(G) \setminus P$. For every $i$, $\partial_G(S)$ is disjoint from $P_i$ and hence $P_i$ is either contained in $S$ or disjoint from $S$. Letting $I = \{i\mid P_i \subseteq S\}$ and $A = S \setminus P$ gives us $S = \bigcup_{i \in I} P_i \cup A$, so \ref{item:C4} holds.

Otherwise, $\abs{\partial_G(S) \setminus P} \leq 1$. Note that we can find some $i$ such that $\partial_G(S) \cap P \subseteq P_i$, as $\partial_G(S) \cap P$ either comprises a single vertex or $\partial_G(S) \subseteq P$ and we can use \ref{item:H4}. By \ref{item:H3}, the graph $G - P_i - \partial_G(S) = G - P_i - (\partial_G(S) \setminus P)$ is connected and is thus either contained in $S$ or disjoint from $S$. In the former case \ref{item:C3} holds, and in the latter case $S$ is contained in $P_i \cup (\partial_G(S) \setminus P)$, fulling \ref{item:C2}.
\end{proof}

\subsection{The family $\mathcal{F}_1$}

Here we show that the topological inspection number of $K_4$ is $4$:


\begin{theo}
    $\s_t(K_4) \ge 4$. \label{res:1}
\end{theo}

The remainder of this subsection contains the proof of Theorem~\ref{res:1}. Note that Theorem~\ref{res:1} implies that every graph in $\mathcal{F}_1$ (i.e., every subdivision of $K_4$) has topological inspection number at least $4$.

Let $G$ be some subdivison of $K_4$ such that each edge is subdivided at least $2^{12}$ times (we did not make an attempt to optimize the value $2^{12}$). We will show that $\s(G) \geq 4$. Let $T = \{a, b, c, d\}$ be the original vertices of the $K_4$ in this subdivision. Assume, towards a contradiction, that $\s(G) \leq 3$.

We begin by applying Lemma~\ref{claim:technical_lemma} in the context of $G$. Let $P_{ab}$, $P_{ac}$, \ldots, $P_{cd}$ be the sets of internal vertices of the paths replacing the corresponding edges of $K_4$. In the notation of Lemma~\ref{claim:technical_lemma}, each set $Q_{xy}$ would then be of the form $Q_{xy} = P_{xy} \cup \{x,y\}$. 

\begin{claim} The sets $P_{ab}$, $P_{ac}$, \ldots, $P_{cd}$ satisfy the assumptions of Lemma~\ref{claim:technical_lemma}.
\end{claim}
\begin{proof} 
It is easy to see that the sets $P_{xy}$ are $G$-connected and that statements \ref{item:H1} and \ref{item:H2} hold. 
For each pair $(x,y)$, $G - P_{xy}$ is a subdivision of a diamond graph, which is 2-connected, so \ref{item:H3} is satisfied. It thus remains to verify \ref{item:H4}.

If $\{v, w\}$ is a cut-set such that $\{v,w\} \subseteq P$---i.e., $\{v, w\} \cap T = \0$---then there exists a path through all $4$ vertices of $T$ which passes through neither $v$ nor $w$. Any vertex not on this path lies on a subdivided edge between two vertices of $T$. For such a vertex to be separated from $T$ by $\{v, w\}$, both $v$ and $w$ must be on this subdivided edge, and thus they both belong to some $P_{xy}$, as desired.
\end{proof}

Now consider the possible sizes of subsets of $V(G)$ with boundary of size at most $2$. Let $M = \max\abs{P_{xy}}$. Suppose $S \subseteq V(G)$ is a set with $\abs{\partial(S)} \le 2$ and $M + 1 < \abs{S} < \abs{V(G)} - M - 1$. Then, by Lemma~\ref{claim:technical_lemma},
\[
    S = \bigcup_{xy \in I} P_{xy} \cup T' \quad \text{for some } I \subseteq \{ab, ac, \ldots, cd\} \text{ and } T' \subseteq T.
\]
There are at most $2^{10}$ such sets $S$ (as we have $2^4$ options for $T'$ and $2^6$ options for $I$). On the other hand, by Proposition~\ref{prop:boundary}, for every $i \in (M + 1, \abs{V(G)} - M - 2)$, we must have such a set $S$ with $\abs{S} = i$ or $\abs{S} = i + 1$. Hence, $\abs{V(G)} - 2M - 4 \le 2 \cdot 2^{10}$, which yields
\begin{equation}\label{eq:2}
    \abs{V(G)} \le 2M + 2^{11} + 4. \tag{$\S$}
\end{equation}

Now suppose without loss of generality that $\abs{P_{ab}} = M$. Let
\[
    R_1 = \{a\} \cup P_{ac} \cup P_{ad}, \quad R_2 = \{b\} \cup P_{bc} \cup P_{bd}, \quad \text{and} \quad R_3 = P_{cd}.
\]
Without loss of generality, suppose $\abs{R_1} \le \abs{R_2}$. Let $N = \max_{i=1, 2, 3}\abs{R_i}$. Note that $N \leq |V(G)| - M$. Since every edge of $K_4$ is subdivided at least $2^{12}$ times, we have $|R_1| \geq 2^{12}$, so, by \eqref{eq:2},
\[N - \abs{R_1} \le \abs{V(G)} - M - 2^{12} \le M - 2^{11} + 4 < M.\]
Let $D$ be the induced subgraph of $G$ on the union of $R_1 \cup R_2 \cup R_3 \cup \{c, d\}$. Let $R_4$ be the set of the $N - \abs{R_1}$ vertices in $P_{ab}$ that are closest to $a$, and let $H = G[V(D) \cup R_4]$ (see Fig.~\ref{fig:subgraph}).

\begin{figure}[t]
\begin{center}
\begin{tikzpicture}[scale=0.8,x=0.75pt,y=0.75pt,yscale=-1,xscale=1]

\draw  [fill={rgb, 255:red, 0; green, 0; blue, 0 }  ,fill opacity=1 ] (144.53,120.73) .. controls (144.53,119.22) and (145.76,118) .. (147.27,118) .. controls (148.78,118) and (150,119.22) .. (150,120.73) .. controls (150,122.24) and (148.78,123.47) .. (147.27,123.47) .. controls (145.76,123.47) and (144.53,122.24) .. (144.53,120.73) -- cycle ;
\draw  [fill={rgb, 255:red, 0; green, 0; blue, 0 }  ,fill opacity=1 ] (188.53,51.73) .. controls (188.53,50.22) and (189.76,49) .. (191.27,49) .. controls (192.78,49) and (194,50.22) .. (194,51.73) .. controls (194,53.24) and (192.78,54.47) .. (191.27,54.47) .. controls (189.76,54.47) and (188.53,53.24) .. (188.53,51.73) -- cycle ;
\draw  [fill={rgb, 255:red, 0; green, 0; blue, 0 }  ,fill opacity=1 ] (187.53,190.73) .. controls (187.53,189.22) and (188.76,188) .. (190.27,188) .. controls (191.78,188) and (193,189.22) .. (193,190.73) .. controls (193,192.24) and (191.78,193.47) .. (190.27,193.47) .. controls (188.76,193.47) and (187.53,192.24) .. (187.53,190.73) -- cycle ;
\draw  [fill={rgb, 255:red, 0; green, 0; blue, 0 }  ,fill opacity=1 ] (233.24,120.73) .. controls (233.24,119.22) and (232.09,118) .. (230.67,118) .. controls (229.25,118) and (228.1,119.22) .. (228.1,120.73) .. controls (228.1,122.24) and (229.25,123.47) .. (230.67,123.47) .. controls (232.09,123.47) and (233.24,122.24) .. (233.24,120.73) -- cycle ;
\draw    (191.27,51.73) -- (190.27,190.73) ;
\draw    (147.27,120.73) .. controls (152.12,91.47) and (164.12,59.47) .. (188.53,51.73) ;
\draw    (230.67,120.73) .. controls (226.06,87.47) and (211.12,52.47) .. (191.27,51.73) ;
\draw    (147.27,122.73) .. controls (152.12,151.04) and (164.12,181.99) .. (188.53,189.47) ;
\draw    (230.67,122.73) .. controls (226.06,154.91) and (211.12,188.47) .. (191.27,189.47) ;
\draw  [color={rgb, 255:red, 12; green, 10; blue, 198 }  ,draw opacity=1 ][dash pattern={on 4.5pt off 4.5pt}] (135.12,121.47) .. controls (135.12,97.47) and (175.12,30.47) .. (183.12,53.47) .. controls (191.12,76.47) and (165.12,98.47) .. (165.12,121.47) .. controls (165.12,144.47) and (191.12,133.47) .. (183.12,188.47) .. controls (172.12,206.47) and (135.12,145.47) .. (135.12,121.47) -- cycle ;
\draw  [color={rgb, 255:red, 12; green, 10; blue, 198 }  ,draw opacity=1 ][dash pattern={on 4.5pt off 4.5pt}] (244.12,122.47) .. controls (244.12,98.47) and (204.91,33.47) .. (196.5,56.47) .. controls (188.1,79.47) and (214.12,97.47) .. (214.12,120.47) .. controls (214.12,143.47) and (187.12,131.47) .. (196.5,191.47) .. controls (208.06,209.47) and (244.12,146.47) .. (244.12,122.47) -- cycle ;
\draw    (70.12,120.47) -- (144.53,120.73) ;
\draw  [dash pattern={on 0.84pt off 2.51pt}]  (70.12,119.73) .. controls (25.12,80.73) and (91.12,33.73) .. (176.12,30.73) .. controls (261.12,27.73) and (334.12,96.73) .. (230.67,120) ;

\draw (152,111.4) node [anchor=north west][inner sep=0.75pt]    {$a$};
\draw (216,113.4) node [anchor=north west][inner sep=0.75pt]    {$b$};
\draw (187,31.4) node [anchor=north west][inner sep=0.75pt]    {$d$};
\draw (185,193.4) node [anchor=north west][inner sep=0.75pt]    {$c$};
\draw (159,144.4) node [anchor=north west][inner sep=0.75pt]  [color={rgb, 255:red, 36; green, 19; blue, 250 }  ,opacity=1 ]  {$R_{1}$};
\draw (198,146.4) node [anchor=north west][inner sep=0.75pt]  [color={rgb, 255:red, 36; green, 19; blue, 250 }  ,opacity=1 ]  {$R_{2}$};
\draw (94,121.4) node [anchor=north west][inner sep=0.75pt]  [color={rgb, 255:red, 0; green, 0; blue, 0 }  ,opacity=1 ]  {$R_{4}$};
\end{tikzpicture}
\end{center}
\caption{Construction of the subgraph $H$.}\label{fig:subgraph}
\end{figure}

\begin{claim}
    The sets $R_{1}$, $R_{2}$, $R_{3}$ satisfy the assumptions of Lemma~\ref{claim:technical_lemma} as subsets of $D$.
\end{claim}
\begin{proof}
It is easy to see that the set $R_{i}$ are connected and that statements \ref{item:H1} and \ref{item:H2} hold. 
For each $i$, $D - R_{i}$ is a cycle, which is 2-connected, so \ref{item:H3} holds. It thus remains to show \ref{item:H4}.

Note that no set $\{x, y\} \subseteq R_1 \cup R_2 \cup R_3$ may separate $c$ from $d$, as there are $3$ internally disjoint paths between them. Hence for $\{x, y\}$ to be a cut-set of $D$, it must separate some vertex $v$ from both $c$ and $d$. Whichever $R_i$ contains $v$ gives two internally disjoint paths from $v$ to each of $c$, $d$, so one of $x$, $y$ must lie on each of these paths. Hence, $x$ and $y$ are both contained in $R_i$, as desired.
\end{proof}
\begin{claim} Suppose a set $S \subseteq V(H)$ satisfies $\abs{\partial_H(S)} \le 2$. Then
 $\abs{S} \le N + 2$ or $\abs{S} \ge \abs{V(H)} - N - 2$.
\end{claim}
\begin{proof} Let $S' = S \cap V(D)$. Note that $\abs{\partial_D(S')} \le 2$ as any vertex in $\partial_D(S')$ is in $\partial_H(S)$ as well.

Since $D$ is 2-connected, either $\abs{S'} < 2$, or $V(D) \subseteq S$, or 
$\abs{\partial_D(S')} \ge 2$ (by Lemma~\ref{claim:connectedness_lemma}).

In the first case, $\abs{S} \le N - \abs{R_1} + 1 < N + 2$. In the second case, \[\abs{S} \ge \abs{V(D)} = \abs{V(H)} - N + \abs{R_1} > \abs{V(H)} - N - 2.\] The remaining case is where $\abs{\partial_D(S')} = 2$. This implies that $\partial_H(S) = \partial_D(S')$. In particular, $R_4$ does not contain a boundary vertex of $S$. Since $R_4$ is connected, it is either entirely contained in $S$ or is disjoint from it. Note that if $R_4 \subseteq S$, then $a \in S$ as otherwise the neighbor of $a$ in $R_4$ will be a boundary vertex, a contradiction. Also, if $S \cap R_4 = \0$, then either $a \not \in S$ or $a \in \partial_D(S')$.

By Lemma~\ref{claim:technical_lemma}, one of the following four cases must hold.

\smallskip

{\textbf{Case 1:}} \textsl{$\abs{S'} \leq 2$}. Then $\abs{S} \le N - \abs{R_1} + 2 < N + 2$.

\smallskip

{\textbf{Case 2:}} \textsl{$S'$ is a subset of $R_i \cup \{c, d\}$ for some $1 \leq i \leq 3$}. If $i = 1$, this implies $\abs{S} \le N - \abs{R_1} + \abs{R_1} + 2 = N + 2$. Otherwise, $a \not \in S$ so $S$ is disjoint from $R_4$. Then $\abs{S} \le \max \{\abs{R_2}, \abs{R_3}\} + 2 \le N + 2$.

\smallskip

{\textbf{Case 3:}} \textsl{$S'$ contains $V(D) \setminus (R_i \cup \{c, d\})$ for some $1 \leq i \leq 3$}. If $i \ne 1$, then $R_1 \subseteq S$, which implies that $a$ is in $S$ but not in $\partial_D(S')$. Therefore, $R_4 \subseteq S$, so $\abs{S} \ge \abs{V(H)} - \max \{\abs{R_2}, \abs{R_3}\} - 2 \ge \abs{V(H)} - N - 2$. Otherwise, $\abs{S} \ge \abs{V(H)} - (\abs{R_1} + \abs{R_4} + 2) = \abs{V(H)} - N - 2$.




\smallskip

{\textbf{Case 4:}} \textsl{$S'$ is the union of some subset of $\{R_1, R_2, R_3, \{c\}, \{d\}\}$}. If $R_2 \cup R_3 \subseteq S$, then $\abs{S} \ge \abs{V(H)} - N - 2$. On the other hand, if both $R_2$, $R_3 \not\subseteq S$, then $\abs{S} \le \abs{R_1} + \abs{R_4} + 2 = N+2$. The only remaining possibility is that exactly one of $R_2$, $R_3$ is a subset of $S$. Note that the cycle formed by $R_2 \cup R_3 \cup \{c, d\}$ contains two boundary vertices of $S$. Thus $R_1 \subseteq S$ if and only if $R_4 \subseteq S$, as otherwise one of $a$ or its neighbor in $R_4$ would be a third boundary vertex. If $R_1 \not\subseteq S$, then $\abs{S} \le \max\{\abs{R_2}, \abs{R_3}\} + 2 = N + 2$. If $R_1 \subseteq S$, then $\abs{S} \ge \abs{R_1} + \min\{\abs{R_2}, \abs{R_3}\} + \abs{R_4} \ge \abs{V(H)} - N - 2$.
\end{proof}

Note that 
\[\abs{V(H)} \ge 2N + \min\{\abs{R_2}, \abs{R_3}\} + 2 \ge 2N + 2^{12} + 2,\] so 
$\abs{V(H)} - N - 2 \ge N + 2^{12}$. Hence there does not exist a set $S \subseteq V(H)$ with $\abs{\partial_H(S)} \le 2$ and $N + 2 < \abs{S} < N+2^{12}$, and so by Proposition~\ref{prop:boundary}, $\s(G) \geq \s(H) \ge 4$, a contradiction.

Therefore, each subdivision of $K_4$ where every edge is subdivided at least $2^{12}$ times has inspection number at least $4$, which implies $\s_t(K_4) \ge 4$.

\subsection{The families $\mathcal{F}_2$ and $\mathcal{F}_3$}

To show the bound $\s_t \geq 4$ for graphs in $\mathcal{F}_2$ and $\mathcal{F}_3$, we first consider a configuration common to both families. Suppose $G$ is a graph comprising two vertices, $a$, $b$, two internally disjoint bipaths between them (the union of which we call $D$), and a path with $a$ as an endpoint which is otherwise disjoint from $D$ (see Fig.~\ref{fig:config}). We let $R$ be the set of the vertices on this path distinct from $a$. Let $I$ be the set of all pairs $\{x,y\}$ of consecutive primary vertices on the bipaths forming $D$. For each $\{x,y\} \in I$, let $P_{xy, 0}$ and $P_{xy, 1}$ be the sets of internal vertices on the two primary paths joining $x$ and $y$. We assume that all the sets $P_{xy, i}$ are nonempty.

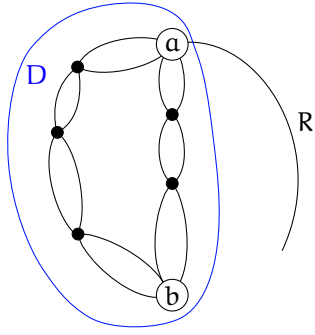
\begin{figure}[H]
\begin{center}
\begin{tikzpicture}[x=0.75pt,y=0.75pt,yscale=-1,xscale=1,scale=0.6]

\draw    (317.5,45) .. controls (384.83,28.67) and (452.83,126.67) .. (409.83,218.67) ;
\draw    (317.5,162.07) .. controls (305,154.47) and (300.5,119) .. (317.5,103) ;
\draw    (317.5,162.07) .. controls (333,152.23) and (329.07,114.77) .. (317.5,103) ;
\draw    (317.5,104.07) .. controls (305,96.47) and (300.5,61) .. (317.5,45) ;
\draw    (317.5,104.07) .. controls (333,94.23) and (329.07,56.77) .. (317.5,45) ;
\draw  [fill={rgb, 255:red, 0; green, 0; blue, 0 }  ,fill opacity=1 ] (312.57,104.07) .. controls (312.57,101.34) and (314.78,99.13) .. (317.5,99.13) .. controls (320.22,99.13) and (322.43,101.34) .. (322.43,104.07) .. controls (322.43,106.79) and (320.22,109) .. (317.5,109) .. controls (314.78,109) and (312.57,106.79) .. (312.57,104.07) -- cycle ;
\draw  [fill={rgb, 255:red, 0; green, 0; blue, 0 }  ,fill opacity=1 ] (312.57,162.07) .. controls (312.57,159.34) and (314.78,157.13) .. (317.5,157.13) .. controls (320.22,157.13) and (322.43,159.34) .. (322.43,162.07) .. controls (322.43,164.79) and (320.22,167) .. (317.5,167) .. controls (314.78,167) and (312.57,164.79) .. (312.57,162.07) -- cycle ;
\draw  [fill={rgb, 255:red, 0; green, 0; blue, 0 }  ,fill opacity=1 ] (233,64.07) .. controls (233,61.34) and (235.21,59.13) .. (237.93,59.13) .. controls (240.66,59.13) and (242.87,61.34) .. (242.87,64.07) .. controls (242.87,66.79) and (240.66,69) .. (237.93,69) .. controls (235.21,69) and (233,66.79) .. (233,64.07) -- cycle ;
\draw  [fill={rgb, 255:red, 0; green, 0; blue, 0 }  ,fill opacity=1 ] (216,119.07) .. controls (216,116.34) and (218.21,114.13) .. (220.93,114.13) .. controls (223.66,114.13) and (225.87,116.34) .. (225.87,119.07) .. controls (225.87,121.79) and (223.66,124) .. (220.93,124) .. controls (218.21,124) and (216,121.79) .. (216,119.07) -- cycle ;
\draw  [fill={rgb, 255:red, 0; green, 0; blue, 0 }  ,fill opacity=1 ] (233.36,204.62) .. controls (233.36,201.89) and (235.57,199.68) .. (238.29,199.68) .. controls (241.02,199.68) and (243.22,201.89) .. (243.22,204.62) .. controls (243.22,207.34) and (241.02,209.55) .. (238.29,209.55) .. controls (235.57,209.55) and (233.36,207.34) .. (233.36,204.62) -- cycle ;
\draw    (317.5,45) .. controls (294,32.23) and (244.43,42.47) .. (237.93,64.07) ;
\draw    (317.5,45) .. controls (309,63.23) and (260,75.23) .. (238.43,64.47) ;
\draw    (220.93,119.07) .. controls (240,109.23) and (242,83.23) .. (238.43,64.47) ;
\draw    (220.43,118.67) .. controls (216,98.23) and (220,77.23) .. (237.93,64.07) ;
\draw    (237.14,205.29) .. controls (248.34,186.54) and (237.75,129.86) .. (220.34,119.09) ;
\draw    (237.14,205.29) .. controls (221.99,195.18) and (203.98,143.03) .. (220.34,119.09) ;
\draw  (327,185.23) .. controls (324,174.23) and (324,171.23) .. (317.5,162) ;
\draw  (315.5,256) .. controls (334,241.23) and (331,202.23) .. (327,185.23) ;
\draw  (306.42,185.23) .. controls (309.83,174.23) and (309.83,171.23) .. (317.23,162) ;
\draw  (319.5,256) .. controls (298.46,241.23) and (301.87,202.23) .. (306.42,185.23) ;
\draw  (262.83,209.91) .. controls (252.06,206.17) and (249.58,204.47) .. (238.29,204.62) ;
\draw  (314.71,259.39) .. controls (312.98,235.78) and (279.11,216.22) .. (262.83,209.91) ;
\draw  (251.2,226.89) .. controls (244.05,217.86) and (241.58,216.16) .. (238.14,204.84) ;
\draw  (316.97,256.09) .. controls (292.9,265.1) and (262.65,240.25) .. (251.2,226.89) ;
\draw  [color={rgb, 255:red, 0; green, 27; blue, 255 }  ,draw opacity=1 ] (199.83,46.67) .. controls (214.83,29.67) and (232.15,16.29) .. (263.83,11.67) .. controls (295.52,7.04) and (324.5,13) .. (332.83,37.67) .. controls (341.17,62.33) and (345.83,64.67) .. (354.83,148.67) .. controls (363.83,232.67) and (342.83,266.67) .. (323.83,275.67) .. controls (304.83,284.67) and (266.83,285.67) .. (247.83,274.67) .. controls (228.83,263.67) and (217.83,247.67) .. (207.83,232.67) .. controls (197.83,217.67) and (182.96,183.67) .. (179.83,144.67) .. controls (176.71,105.67) and (184.83,63.67) .. (199.83,46.67) -- cycle ;

\node [draw,fill=white,circle,minimum size=12pt,outer sep=0pt,inner sep=0pt] (a) at (317.5,45) {$a$};
\node [draw,fill=white,circle,minimum size=12pt,outer sep=0pt,inner sep=0pt] (b) at (317.5,256) {$b$};

\draw (420,100) node [anchor=north west][inner sep=0.75pt]   [align=left] {$\displaystyle R$};
\draw (192,60) node [anchor=north west][inner sep=0.75pt]  [color={rgb, 255:red, 13; green, 0; blue, 255 }  ,opacity=1 ] [align=left] {$\displaystyle D$};

\end{tikzpicture}
\end{center}
\caption{The special configuration $H$.} \label{fig:config}
\end{figure}

\begin{claim} The sets $P_{xy, i}$ satisfy the hypotheses of Lemma \ref{claim:technical_lemma} applied to the graph $D$.
\end{claim}
\begin{proof}Evidently, the sets $P_{xy, i}$ are connected and both \ref{item:H1}, \ref{item:H2} hold. Removing any set $P_{xy, i}$ from $D$ keeps the remaining graph 2-connected, showing \ref{item:H3}. Finally, if $v$ and $w$ are two non-primary vertices that do not lie on the same primary path, then $D - \{v,w\}$ is connected, proving \ref{item:H4}. 
%
\end{proof}

Now we are going to make a few additional assumptions about the graph $H$. First, let $T$ be the sum of the orders of the bipaths comprising $D$ (recall that the order of a bipath is the number of its primary vertices). We shall assume that:
\begin{itemize}[wide]
    \item $\abs{P_{xy,i}} \geq 2^{4T}$ for all $\{x,y\} \in I$ and $i \in \{0,1\}$.
\end{itemize}
Next, let $N$ be the maximum of $\abs{P_{xy, i}}$ taken over all sets $P_{xy, i}$ with $a \not \in \{x,y\}$. Similarly, let $M$ be the maximum of $\abs{P_{xy, i}}$ taken over all sets $P_{xy, i}$ with $a \in \{x,y\}$. 
We assume that:
\begin{itemize}[wide]
    \item $N \geq M$ and $\abs{R} = N - M$. 
\end{itemize}
With these assumptions, we can show that the inspection number of $H$ is at least $4$:

\begin{prop} $\s(H) \ge 4$. \label{prop:donut_configuration}
\end{prop}
\begin{proof} Suppose $S \subseteq V(H)$ is a subset such that $\abs{\partial_G(H)} \le 2$ and
\[
    N+2 \leq |S| \leq \abs{V(H)} - N - 2.
\]
Let $S' = S \cap V(D)$. Note that $\partial_{D}(S') \subseteq \partial_H(S)$. 
Since $D$ is 2-connected, we have $\abs{S'} < 2$, or $V(D) \subseteq S$, or $\abs{\partial_D(S')} = 2$ (by Lemma~\ref{claim:connectedness_lemma}). In the first case, $\abs{S} \le |R| + 1 = N - M + 1 < N + 2$, a contradiction. In the second case, \[\abs{S} \ge \abs{V(D)} = \abs{V(H)} - |R| = \abs{V(H)} - N + M > \abs{V(H)} - N - 2,\]
which is again a contradiction. Therefore, $\abs{\partial_D(S')} = 2$, which implies that $\partial_H(S) = \partial_D(S')$. In particular, $R$ does not contain any boundary vertices of $S$. 
Since $R$ is $H$-connected, it is either entirely contained in $S$ or is disjoint from it. Note that if $R \subseteq S$, then $a \in S$ as otherwise the neighbor of $a$ in $R$ will be a boundary vertex, a contradiction. Also, if $S \cap R = \0$, then either $a \not \in S$ or $a \in \partial_D(S')$.

Since $\abs{\partial_D(S')} \le 2$, by Lemma~\ref{claim:technical_lemma}, one of the following four cases must hold.

\smallskip

{\textbf{Case 1:}} \textsl{$\abs{S'} \leq 2$}. Then $\abs{S} \le \abs{R} + 2 = N - M + 2 < N + 2$, a contradiction.

\smallskip

{\textbf{Case 2:}} \textsl{$S'$ is contained in $P_{xy,i} \cup \{z\}$ for some $P_{xy,i}$ and a primary vertex $z$}. If $a \in \{x,y\}$, then $|P_{xy,i}| \leq M$, so $\abs{S} \le \abs{R} + M + 1 = N + 1$, a contradiction. Thus, $a \not \in \{x,y\}$. Furthermore, $z = a$, as otherwise $R \cap S = \0$, which again implies $\abs{S} \leq N + 1$. Therefore, $a \in \partial_H(S)$. Consider the cycle $C$ in $H$ with $V(C) = P_{xy,0} \cup P_{xy, 1} \cup \{x,y\}$. It is 2-connected, not contained in $S$, and, since $a \not \in V(C)$, it contains at most one boundary vertex of $S$. Thus, $\abs{S \cap V(C)} \leq 1$, so $\abs{S} \le \abs{R} + 2 < N + 2$, a contradiction.

\smallskip

{\textbf{Case 3:}} \textsl{$S'$ contains $V(D) \setminus (P_{xy,i} \cup \{z\})$ for some $P_{xy,i}$ and a primary vertex $z$}. Then $S$ contains at least 3 sets $P_{uv,j}$ with $a \in \{u,v\}$. If $a \not\in S$, each of these sets includes a boundary vertex of $S$, a contradiction. Hence $a \in S$. If $a \in \{x,y\}$, then $\abs{S} \ge \abs{V(H)} - M - \abs{R} - 1 = \abs{V(H)} - N - 1$. Otherwise, considering the cycle $C$ in $H$ with vertex set $P_{xy,0} \cup P_{xy, 1} \cup \{x,y\}$ as we did in Case $2$ yields that $a \not \in \partial_H(S)$. Hence $R \subseteq S$, so $\abs{S} \ge \abs{V(H)} - N - 1$, a contradiction.



\smallskip

Therefore, any such set $S$ has to satisfy the following:

\smallskip

{\textbf{Case 4:}} \textsl{$S'$ is the union of a collection of sets $P_{xy, i}$ and some primary vertices}. Since there are $2T - 4$ sets $P_{xy,i}$ and $T - 2$ primary vertices, the number of such sets $S$ is at most
\[
    2^{2T - 4} \cdot 2^{T - 2} \cdot 2 = 2^{3T - 5},
\]
where we are multiplying by $2$ to account for the fact that $S$ either contains $R$ or is disjoint from $R$. If $\s(H) \le 3$, then, by Proposition~\ref{prop:boundary}, for each $i \in [N+2, \abs{V(H)} - N - 2)$, there exists some $S \subseteq V(H)$ with $\abs{\partial_H(S)} \le 2$ and $\abs{S} \in \{i, i+1\}$. Therefore,
\[\abs{V(H)} - 2N - 4 \le 2 \cdot 2^{3T - 5} = 2^{3T - 4} < 2^{3T}.\]
However, $V(H)$ contains as disjoint subsets $R$, a set $P_{xy,i}$ with $a \in \{x,y\}$ of size $M$, a set $P_{xy,i}$ with $a \not\in \{x,y\}$ of size $N$, and at least one other set $P_{xy,i}$, so
\[\abs{V(H)} \ge (N-M) + M + N + 2^{4T} > 2N + 4 + 2^{3T},\]
a contradiction.
\end{proof}

\begin{theo} Suppose $G_0 \in \mathcal{F}_2$. Then $\s_t(G_0) \ge 4$. \label{res:2} \end{theo}
\begin{proof}
Let $a$, $b$ be the shared vertices of the three bipaths forming $G_0$ (see Definition~\ref{defn:families}). Let $T$ be the sum of the orders of the bipaths comprising $G_0$, and suppose $G$ is some subdivision of $G_0$ such that each edge is subdivided at least $2^{4T}$ times. We will show that $\s(G) \geq 4$. Note that $G$ is in $\mathcal{F}_2$ and that every primary path in $G$ contains at least $2^{4T}$ internal vertices. Let $P$ be a primary path of maximum length in $G$. Let $D$ be the union of the two bipaths not containing $P$. Let $M$ (resp.~$L$) be the maximum number of internal vertices on a primary path in $D$ adjoining $a$ (resp.~$b$). Without loss of generality, assume that $L \geq M$ and let $N$ be the maximum number of internal vertices on a primary path in $D$ not adjoining $a$. Note that $N \geq L \geq M$ and $P$ has at least $N$ internal vertices. 

Let $Q$ be an arbitrary $ab$-path through $P$ disjoint from $D$. 
Let $R$ be the set of $N - M$ vertices in $V(Q) \setminus \{a\}$ closest to $a$. The induced subgraph $H$ of $G$ on $V(D) \cup R$ satisfies the assumptions of Proposition~\ref{prop:donut_configuration}. Therefore, $\s(G) \geq \s(H) \geq 4$, as desired.
%
\end{proof}

\begin{theo} Suppose $G_0 \in \mathcal{F}_3$. Then $\s_t(G_0) \ge 4$. \label{res:3} \end{theo}
\begin{proof}
    Recall that by Definition~\ref{defn:families}, $G_0$ consists of 4 bipaths $B_{a}$, $B_{b}$, $B_{c}$, $B_{d}$ and two paths $P_1$, $P_2$, such that $B_{a}$, $B_{b}$ share endpoints $v_1$ and $v_2$, $B_c$, $B_d$ share endpoints $v_3$ and $v_4$, but the bipaths are otherwise pairwise vertex-disjoint. Furthermore, $P_1$ has endpoints $v_1$, $v_3$, and $P_2$ has endpoints $v_2$, $v_4$, but these paths have no other vertices in common with the bipaths (however, $P_1$ and $P_2$ may intersect each other).

    Let $T$ be the sum of the orders of the bipaths $B_{a}$, $B_{b}$, $B_{c}$, and $B_{d}$, and suppose $G$ is some subdivision of $G_0$ such that each edge is subdivided at least $2^{4T}$ times. We will show that $\s(G) \geq 4$. Note that $G$ is in $\mathcal{F}_3$ and every primary path in $G$ has at least $2^{4T}$ internal vertices. Let $P$ be a primary path of maximum length in $G$. Suppose without loss of generality that $P$ belongs to a bipath with endpoints $v_3$, $v_4$. Let $D$ be the subgraph of $G$ comprising the two bipaths with endpoints $v_1$, $v_2$. Let $M$ (resp.~$L$) be the maximum number of internal vertices on a primary path in $D$ adjoining $v_1$ (resp.~$v_2$). Without loss of generality, assume that $L \geq M$ and let $N$ be the maximum number of internal vertices on a primary path in $D$ not adjoining $v_1$. Note that $N \geq L \geq M$ and $P$ has at least $N$ internal vertices.

    Let $Q$ be a path that starts at $v_1$, follows the path $P_1$ to $v_3$, and then continues to $v_4$ through $P$. Let $R$ be the set of $N - M$ vertices in $V(Q) \setminus \{v_1\}$ closest to $v_1$. The induced subgraph $H$ of $G$ on $V(D) \cup R$ satisfies the assumptions of Proposition~\ref{prop:donut_configuration} (with $v_1$ in place of $a$ and $v_2$ in place of $b$). Therefore, $\s(G) \geq \s(H) \geq 4$, as desired.
%
\end{proof}

\section{Building simple GSP decompositions}\label{sec:forb_dec}

\subsection{GSP decompositions of $K_4$-subdivision-free graphs}\label{subsec:K4-free}

The main result of this subsection is Corollary~\ref{corl:K4-free}, which says that connected $K_4$-subdivision-free graphs admit GSP decompositions (with some control over the choice of terminals). By Theorem~\ref{res:1}, every graph $G$ with $\s_t(G) \leq 3$ must be $K_4$-subdivision-free, so Corollary~\ref{corl:K4-free} implies that such $G$ has a GSP decomposition. In the subsequent subsections we shall explain how to make this GSP decomposition simple.

\begin{lemma} Suppose $G$ is a $K_4$-subdivision-free graph and $H$ is a connected subgraph of $G$ containing distinct vertices $a$, $b$ such that $a \not \in N_H(b)$, $\{a, b\}$ is not a cut-set of $H$, and there exists a path from $a$ to $b$ in $G - (H - \{a, b\})$. Then there exists a cut vertex of $H$ separating $a$ and $b$. \label{lemma:separating}
\end{lemma}
\begin{proof} Suppose $H$ contained internally disjoint paths $P$, $Q$ between $a$ and $b$. Since $a \not \in N_H(b)$, $P$ and $Q$ must each contain at least one intermediate vertex. As $H$ remains connected after removing $a$ and $b$, there must be some path connecting intermediate vertices of $P$ and $Q$ that is internally disjoint from $P$ and $Q$. Let $v$ and $w$ be the endpoints of such a path. Then there are pairwise disjoint paths connecting each of the pairs of vertices in $\{a, b, v, w\}$ (the path between $a$ and $b$ in particular is taken to be one contained in $G - (H - \{a, b\})$ given by the hypothesis). Thus $G$ contains a subdivision of $K_4$, a contradiction. Hence $H$ cannot contain a pair of internally disjoint paths between $a$ and $b$, and so the claim follows by Menger's theorem \cite[Corollary 3.3.5]{Diestel}.
\end{proof}
\begin{prop} Suppose $G$ is a $2$-connected $K_4$-subdivision-free graph, and $\{a, b\} \subseteq V(G)$ are either neighbors or a cut-set of $G$. Then $(G, a, b)$ admits a series-parallel decomposition. \label{prop:seriesparallel}
\end{prop}
\begin{proof} We shall prove this via induction on $\abs{V(G)}$. The base case where $\abs{V(G)} = 2$ is trivial.

To prove the general case, let the connected components of $G - \{a, b\}$ be $H_0'$, \ldots, $H_k'$, and for each $i$, let $H_i$ be obtained from $G[H_i' \cup \{a, b\}]$ by removing the edge between $a$ and $b$ if it exists. For each $i$, the graph $G - H_i'$ contains an $ab$-path $P_i$. Indeed, if $a$ and $b$ are neighbors, the edge between $a$ and $b$ is such a path. Otherwise, $\{a, b\}$ is a cut-set of $G$ so there are at least two connected components of $G - \{a, b\}$. Hence $k > 0$ and we can choose some $H_j'$ distinct from $H_i'$. Since $G$ is $2$-connected, $G[H_j' \cup \{a, b\}]$ is a connected subgraph of $G - H_i'$, and thus it contains an $ab$-path. By Lemma~\ref{lemma:separating}, we may conclude that $H_i$ must contain a cut-vertex separating $a$ and $b$.

Consider the block-cut tree of $H_i$. If it is not a single path from the block containing $a$ to the block containing $b$ (which are distinct by the above), we can find some leaf block $L$ with a cut-vertex $v$ such that $a$, $b \not\in V(L) \setminus \{v\}$. Then $v$ separates $V(L) \setminus \{v\}$ from the rest of $G$, which is impossible as $G$ is $2$-connected. 

Thus the block-cut tree of $H_i$ is a path from the block containing $a$ to the block containing $b$. Enumerate the blocks of $H_i$ as $B_0$, \ldots, $B_{\ell-1}$ in order of increasing distance from $a$, and let $v_j$ be the cut-vertex between $B_{j-1}$ and $B_{j}$. We also write $v_0 = a$, $v_\ell = b$. We can find a path from each $v_j$ to $v_{j+1}$ without passing through any other vertices of $B_j$: take a path from $v_j$ to $a$ through the blocks $B_{j-1}$, \ldots, $B_0$ in order, then $P_i$, then a path from $b$ to $v_{j+1}$ through the blocks $B_{\ell-1}$, \ldots, $B_{j+1}$ in order. Furthermore, $B_j$ cannot contain a cut-vertex separating $v_{j}$ from $v_{j+1}$, so by the contrapositive of Lemma~\ref{lemma:separating}, either $\{v_j, v_{j+1}\}$ is a cut-set of $B_j$, or $v_j \in N_{B_j}(v_{j+1})$. In either case we can apply the induction hypothesis to see that $(B_j, v_j, v_{j+1})$ admits a series-parallel decomposition.

Note that 
\[
    (H_i, a, b) = (B_0, v_0, v_1) \circ_s (B_1, v_1, v_2) \circ_s \cdots \circ_s (B_{\ell-1}, v_{\ell-1}, v_\ell),
\]
so this implies that $(H_i, a, b)$ admits a series-parallel decomposition. Furthermore, we have
\[
    (G, a, b) = \begin{cases}(H_0, a, b) \circ_p \cdots \circ_p (H_k, a, b) &\text{if } a \not \in N_G(b);\\
    (H_0, a, b) \circ_p \cdots \circ_p (H_k, a, b) \circ_p (e_{ab}, a, b) &\text{if } a \in N_G(b),
    \end{cases}
\]
where $e_{ab}$ is the graph comprising the single edge between $a$ and $b$, so $G$ admits a series-parallel decomposition, as desired.
\end{proof}
\begin{lemma} If $\gwt{G} = (G, a, b)$ admits a series-parallel decomposition $T$, then $G$'s block-cut tree comprises a single path \ep{possibly a single block}, and $a$ and $b$ are each contained in a leaf block. If we let $v_1$, \ldots, $v_{k-1}$ be the cut-vertices of $G$ in order of distance from $a$, and $v_0 = a$, $v_k = b$, then for every block $B$, $(B, v_i, v_{i+1}) \in V(T)$ where $i$ is the unique index such that $\{v_i, v_{i+1}\} \subseteq V(B)$. Furthermore, for any $v \in V(G)$, we can find a path from $a$ through $v$ to $b$. \label{lemma:series_parallel_blocks}
\end{lemma}
\begin{proof} We shall show this by induction on $\abs{V(T)}$. The claim is clearly true if $T$ is trivial, where $G$ must comprise a single edge.

If $\circ_{T, \gwt{G}} = \circ_s$, suppose $(H, a, c)$ and $(K, c, b)$ are the children of $\gwt{G}$ in $T$. Then $c$ separates $V(H) \setminus \{c\}$ from $V(K) \setminus \{c\}$ in $G$, so the block-cut tree of $G$ is simply the union of the block-cut trees of $H$ and $K$ with an additional edge between each of their blocks which contain $c$. One notes that this, combined with the induction hypothesis
, gives the desired structure of the block-cut graph of $G$. Each block $(B, v_i, v_{i+1})$ is either a block of $H$ or a block of $K$, so it is in $V(T)$ by the induction hypothesis. 
Finally, for any $v \in V(G)$, suppose without loss of generality that $v \in V(H)$. Then taking a path in $H$ from $a$ through $v$ to $c$ and extending it by any path in $K$ from $c$ to $b$ gives a desired path from $a$ through $v$ to $b$ in $G$.

If $\circ_{T, \gwt{G}} = \circ_p$, we will show that $G$ is $2$-connected. Let $(H, a, b)$ and $(K, a, b)$ be the children of $\gwt{G}$ in $T$. Consider any vertex $v \in V(G)$. We will show that the graph $G - v$ is connected. We may assume that $v \neq a$ \ep{the argument when $v = a$ is the same but with the roles of $a$ and $b$ exchanged}. We claim that every vertex $w \in V(G) \setminus \{v\}$ is reachable in $G - v$ by a path from $a$. Without loss of generality, we may assume that $w \in V(H)$. By the induction hypothesis, we can find a path $P$ from $a$ through $w$ to $b$ in $H$ and a path $Q$ from $a$ to $b$ in $K$. Write $P = P_1 \cup P_2$, where $P_1$ is an $aw$-path and $P_2$ is a $wb$-path, noting that $P_1$ and $P_2$ share only a single vertex $w$. If $v \not \in V(P_1)$, then $P_1$ is a desired $aw$-path in $G - v$. Otherwise, i.e., if $v \in V(P_1)$, a desired $aw$-path in $G - v$ is $Q \cup P_2$. 
Now that we have established that $G$ is $2$-connected, 
the block-cut graph of $G$ comprises a single block, and the desired properties in the claim follow immediately. To find a path from $a$ through $v$ to $b$ for any $v \in V(G)$, we simply take the path in whichever of $H$ or $K$ contains $v$ using the induction hypothesis.
\end{proof}

\begin{lemma} Suppose $\gwt{G} = (G, a, b)$ and $\gwt{H} = (H, c, d)$ admit GSP decompositions $T_G$, $T_H$ respectively, and $V(G) \cap V(H) = \{c\}$. Let $K$ be the union of $G$ and $H$. Then $\gwt{K} = (K, a, b)$ admits a GSP decomposition $T$ with $\gwt{H} \in V(T)$ such that the induced decomposition of $\gwt{H}$ in $T$ is $T_H$.

Furthermore, if both $T_G$ and $T_H$ are simple, then we can additionally further restrict $T$ to be simple and have the same complexity as $T_G$.\label{lemma:blockmerge}
\end{lemma}
\begin{proof} We shall proceed by induction on $\abs{V(G)}$. For the base case, $\abs{V(G)} = 2$ and $G$ comprises a single edge. If $c = a$, $\gwt{K} = \gwt{G} \circ_b \gwt{H}$. Similarly, if $c = b$, $\gwt{K} = \gwt{G} \circ_{b'} \gwt{H}$ and $T$ can be obtained directly by joining $T_G$ to $T_H$. In either case, the complexity of $T$ is equal to the complexity of $T_G$, so if $T_G$ and $T_H$ were simple, $T$ would be simple too.

In the general case, $T_G$ is nontrivial. Let $\gwt{G_0} = (G_0, s, t)$, $\gwt{G_1}$ be the children of $G$ in $T_G$. Without loss of generality we may assume that $c \in V(G_0)$. Then by the induction hypothesis, $(G_0 \cup H, s, t)$ admits a GSP decomposition $T'$ containing $\gwt{H}$ with induced decomposition $T_H$. Connecting $T'$ with the induced decomposition of $T_G$ on $\gwt{G_1}$ gives the desired decomposition $T$. If $T_G$ and $T_H$ were simple, the induced decompositions on $G_0$ and $G_1$ would be simple too. By the induction hypothesis, $T'$ is simple and has the same complexity as the induced decomposition on $G_0$ by $T_G$, so $T$ has the same complexity as $T_G$ and is simple.
\end{proof}
\begin{corl} Suppose $G$ is a connected $K_4$-subdivision-free graph, and $\{a, b\} \subseteq V(G)$ are either neighbors or a cut-set of a 2-connected component of $G$. Then $(G, a, b)$ is a GSP graph.\label{corl:K4-free}
\end{corl}
\begin{proof} We proceed by induction on the number of blocks in $G$. In the case where the number of blocks is $1$, this reduces to Proposition~\ref{prop:seriesparallel}. In the general case, take a leaf block $L$ which does not contain both $a$ and $b$. Let $v$ be the cut-vertex of $L$. By the induction hypothesis, $(G - (L - \{v\}), a, b)$ is a GSP graph, and by Proposition~\ref{prop:seriesparallel}, $(L, \{v, w\})$ is a GSP graph for any neighbor $w$ of $v$ in $L$. Hence by Lemma~\ref{lemma:blockmerge}, $(G, a, b)$ is a GSP graph.
\end{proof}

\subsection{Simple series-parallel decompositions of 2-connected graphs}

The bulk of the work in the proof of the implication \ref{item:families} $\Longrightarrow$ \ref{item:GSP} in Theorem~\ref{theo:classification} is done in this subsection, where we show that $2$-connected graphs $G$ without any member of $\mathcal{F}_1 \cup \mathcal{F}_2 \cup \mathcal{F}_3$ as a subgraph admit simple series-parallel decompositions (Proposition~\ref{prop:series_parallel_for_block}). 

\begin{lemma} Suppose $\gwt{G} = (G, a, b)$ has a GSP decomposition $T$ with complexity $k$. Then $G$ contains at least $k$ pairwise internally vertex-disjoint bipaths with endpoints $\{a, b\}$. \label{lemma:complexity} \end{lemma}
\begin{proof} We proceed by induction on $\abs{V(G)}$. If $T$ is trivial, we are done as then $k = 0$.

If $\circ_{T, \gwt{G}} = \circ_p$, $\circ_b$ or $\circ_{b'}$, then by definition $k$ is the sum of the complexities of the children of $\gwt{G}$ which have terminals $(a, b)$. By the induction hypothesis, each of these children have at least as many bipaths with terminals $\{a, b\}$ as their respective complexities, and since these children may only intersect at $\{a, b\}$, these bipaths are internally vertex-disjoint, proving the claim.

Otherwise, $\circ_{T, \gwt{G}} = \circ_s$. Then, by definition, either $\gwt{G}$ is bridged and $k = 0$ or $\gwt{G}$ is not bridged and $k = 1$. In the former case, there is nothing to prove, so we may suppose that $k = 1$. Consider the block-cut tree of $G$. Since $\circ_{T, \gwt{G}} = \circ_s$, there exists at least one cut-vertex separating $a$ from $b$, so $a$ and $b$ do not share a block. Let $P$ be the unique path in the block-cut tree from the block containing $a$ to the block containing $b$. If any block in $P$ comprises a single edge, this edge would be a bridge separating $a$ from $b$, contradicting the assumption that $\gwt{G}$ is not bridged. Hence every block in $P$ contains at least $3$ vertices, and thus has a cycle containing the cut-vertices it shares with its neighbors in $P$ (or in the case of the endpoint blocks, the one shared cut-vertex and whichever of $a$ or $b$ it contains). The union of these cycles is a bipath, as desired.
\end{proof}
We can note the following corollary, which immediately results from the previous lemma:
\begin{corl} Suppose $T$ is a complex GSP decomposition of a graph $G$. Then some $\gwt{H} \in V(T)$ contains as a subgraph a pair of internally disjoint bipaths between its terminals. \label{cor:complex_bipaths}
\end{corl}

\begin{lemma} Suppose $\gwt{G} = (G, a, b)$ admits a simple GSP decomposition $T$. Then $(G, b, a)$ also admits a simple GSP decomposition. Furthermore, if $T$ is a series-parallel decomposition then $(G, b, a)$ also admits a simple series-parallel decomposition.
\label{lemma:inversion}
\end{lemma}
\begin{proof} We proceed by induction on $\abs{V(G)}$. The base case where $G$ comprises a single edge is trivial. Suppose $(G_0, s_0, t_0)$ and $(G_1, s_1, t_1)$ are the children of $\gwt{G}$ in $T$. Their induced decompositions must be simple GSP decompositions, so by the induction hypothesis $(G_0, t_0, s_0)$, $(G_1, t_1, s_1)$ also admit simple GSP decompositions. If $\circ_{T, \gwt{G}} = \circ_s$ or $\circ_p$, then $\gwt{G} = (G_1, t_1, s_1) \circ_{T, \gwt{G}} (G_0, t_0, s_0)$. Otherwise, $\gwt{G} = (G_1, t_0, s_0) \circ_* (G_0, s_1, t_1)$ where $\circ_*$ is whichever of $\{\circ_b, \circ_{b'}$ is not $\circ_{T, \gwt{G}}$. In either case, connecting the two GSP decompositions given of the children gives a GSP decomposition of $G$ as desired.
\end{proof}
\begin{prop} Suppose $G$ is a $2$-connected graph and $(G, a, b) = (H, a, b) \circ_p (K, a, b)$, where both $(H, a, b)$ and $(K, a, b)$ admit simple series-parallel decompositions $T_H$, $T_K$ respectively. Then $G$ admits a simple series-parallel decomposition with $a$ as a terminal.\label{prop:rotation}
\end{prop}
\begin{proof} Suppose without loss of generality that $|V(H)| \leq |V(K)|$. We argue by induction on $|V(H)| + \abs{E(H)}$.

For the base case, where $\abs{V(H)} + \abs{E(H)} = 3$, i.e. $H \cong K_2$, directly combining $T_H$ and $T_K$ gives a series-parallel decomposition $T$ of $G$ with complexity at most $1$, since $H$ comprises a single edge and so $T_H$ is trivial, with complexity $0$. Every proper subtree of $T$ is a subtree of $T_H$ or $T_K$, and thus must have complexity at most $1$, so $T$ is simple as desired.

In the general case, $T_H$ is nontrivial. Suppose $\circ_{T_H, \gwt{H}} = \circ_p$ and $(H_0, a, b)$, $(H_1, a, b)$ are the children of $\gwt{H}$ in $T$. Since $T_H$ has complexity at most $1$, one of these children has an induced decomposition of complexity $0$; suppose without loss of generality that this is $H_1$. This naturally gives a simple series-parallel decomposition of $\pr{(H_1, a, b) \circ_p (K, a, b)}$. Then, note that
\[
    (G, a, b) = (H_0, a, b) \circ_p \pr{(H_1, a, b) \circ_p (K, a, b)}.
\]
Since $H_0$ is a proper subgraph of $H$, the induction hypothesis implies that $G$ has a simple series-parallel decomposition with $a$ as a terminal as desired.

Now suppose $\circ_{T_H, \gwt{H}} = \circ_s$. By Lemma~\ref{lemma:series_parallel_blocks}, the block-cut tree of $H$ comprises a path of blocks from the block containing $a$ to the block containing $b$, and each block is contained in $V(T_H)$. Let $B_0$, \ldots, $B_{k-1}$ be the blocks in order of increasing distance from $a$, and let $v_0$, \ldots, $v_k$ be defined as in Lemma~\ref{lemma:series_parallel_blocks}. We have $(B_i, v_i, v_{i+1}) \in V(T_H)$ for each $i$, and so they have simple induced series-parallel decompositions. By Lemma~\ref{lemma:inversion}, $(B_i, v_{i+1}, v_i)$ also each admit simple series-parallel decompositions. Now, note that
\[
    (G, a, v_1) = (B_0, a, v_1) \circ_p \pr{(K, a, b) \circ_s (B_{k-1}, v_k, v_{k-1}) \circ_s \ldots \circ_s (B_1, v_2, v_1)},
\]
and each graph with terminals admits a simple series-parallel decomposition. Since $B_0$ has strictly fewer vertices than $H$, the induction hypothesis implies that $G$ has a simple series-parallel decomposition with $a$ as a terminal as desired.
\end{proof}
\begin{lemma} If $T$ is a GSP decomposition of a graph $G$, and $(H, a, b) \in T$, then $\{a, b\}$ separates $V(H) \setminus \{a, b\}$ from the rest of $G$. \label{lemma:terminals_disconnect}
\end{lemma}
\begin{proof} We proceed by induction on $\abs{V(G)} - \abs{V(H)}$. If $H = G$, the claim is trivial. 
Otherwise, $H$ is a proper subgraph of $G$ and thus $\gwt{H}$ is a child of some $\gwt{K}$ in $T$. When $K = G$, the claim immediately follows from the definition of the operations $\circ_s$, $\circ_p$, $\circ_b$, $\circ_{b'}$. The remaining case is when $K$ is also a proper subgraph of $G$. By the induction hypothesis, substituting $K$ for $G$, the only vertices of $H$ which could have neighbors in $V(K) \setminus V(H)$ are $\{a, b\}$. If we instead substitute $K$ for $H$, we get that the only vertices of $K$ which could have neighbors in $V(G) \setminus V(K)$ are its terminals, and whichever of these terminals are contained in $H$ must themselves be terminals of $H$ (one can easily verify this by considering each possible choice of $\circ_{T, \gwt{K}}$). These two combined show the claim, completing the proof.
\end{proof}


We say that graphs with terminals $\gwt{G} = (G, a, b)$ and $\gwt{H} = (H, c, d)$ are \emphd{non-terminally disjoint} if $V(G)\cap V(H) \subseteq \{a,b,c,d\}$. 

\begin{prop} Suppose $G$ is a 2-connected graph that does not contain any member of $\mathcal{F}_3$ as a subgraph. Let $T_0$, $T_1$ be series-parallel decompositions of $G$. If $\gwt{H_0} \in V(T_0)$, $\gwt{H_1} \in V(T_1)$ are such that both their induced decompositions are complex, then they are not non-terminally disjoint. \label{prop:minial_complex_descendant}
\end{prop}
\begin{proof} Suppose otherwise, so $H_0$ and $H_1$ are non-terminally disjoint. By Corollary~\ref{cor:complex_bipaths}, the induced decompositions of $H_0$ and $H_1$ (in $T_0$ and $T_1$ respectively) contain $(K_0, a_0, b_0)$ and $(K_1, a_1, b_1)$ which themselves contain two disjoint bipaths between their respective terminals.

Since $G$ is $2$-connected, the sets $V(K_0)$ and $V(K_1)$ do not have a vertex separator of size $1$, so by Menger's theorem \cite[Theorem 3.3.1]{Diestel}, there exist two disjoint paths $P_0$, $P_1$ between them. Since $\gwt{K_0}$ and $\gwt{K_1}$ are non-terminally disjoint, by Lemma~\ref{lemma:terminals_disconnect}, we may assume that $P_0$ and $P_1$ intersect $K_0$ and $K_1$ only at the terminals. Then the union of $P_0$, $P_1$, $K_0$ and $K_1$ contains as a subgraph an element of $\mathcal{F}_3$, a contradiction.
%
%
\end{proof}

If $T$ is a series-parallel decomposition of a graph $G$, then for any two graphs $\gwt{H_0}$, $\gwt{H_1} \in V(T)$, either they are non-terminally disjoint or one of them contains the other as a subgraph. Therefore, Proposition~\ref{prop:minial_complex_descendant} implies that, given any complex series-parallel decomposition $T$ of a graph $G$ without any member of $\mathcal{F}_3$ as a subgraph, there is a unique inclusion-minimal $\gwt{H} \in V(T)$ whose induced decomposition is complex. We call this $\gwt{H}$ the \emphd{minimal complex descendant} of $G$ in $T$. 

\begin{lemma} If $\gwt{G} = (G, a, b)$ admits a series-parallel decomposition $T$, then $\gwt{G}$ can be obtained by combining, via $\circ_p$, the graphs with terminals $(H, a, b) \in V(T)$ such that no child of $(H, a, b)$ in $T$ has terminals $(a, b)$.\label{lemma:combining_parallel}
\end{lemma}
\begin{proof} Call the graphs with terminals $(H, a, b) \in V(T)$ such that no child of $(H, a, b)$ in $T$ has terminals $(a, b)$ \textit{good graphs}. We argue by induction on the number of good graphs. If there is a single good graph, then $(G, a, b)$ itself must be good, as each graph with terminals $(a, b)$ is either good or has two children with terminals $(a, b)$. Otherwise, $\gwt{G} = (H, a, b) \circ_p (K, a, b)$ for some subgraphs $H$, $K$. Applying the induction hypothesis on $H$ and $K$ and their induced decompositions shows that $\gwt{H}$ and $\gwt{K}$ can each be obtained by combining some subset of the good graphs via $\circ_p$, and hence so can $\gwt{G}$, proving the lemma.
\end{proof}


\begin{prop} Suppose $G$ is a 2-connected graph without any member of $\mathcal{F}_1 \cup \mathcal{F}_2 \cup \mathcal{F}_3$ as a subgraph, and let $T$ be a series-parallel decomposition of $G$. Then either $T$ is simple or $G$ admits a simple series-parallel decomposition sharing a terminal with the minimal complex descendant of $G$ in $T$.\label{prop:series_parallel_for_block}
\end{prop}
\begin{proof}\stepcounter{ForClaims} \renewcommand{\theForClaims}{\ref{prop:series_parallel_for_block}}
We may operate under the assumption that $T$ is not simple. Let $(H, a, b)$ be the minimal complex descendant of $G$ in $T$.

\begin{smallclaim}
    The graph with terminals $(G,a,b)$ admits a series-parallel decomposition $T'$.
\end{smallclaim}
\begin{claimproof}
    If $V(H) = V(G)$, we can take $T' = T$. Otherwise, by Lemma~\ref{lemma:terminals_disconnect}, $\{a, b\}$ separates $V(H) \setminus \{a, b\}$ from $V(G) \setminus V(H)$, so $(G,a,b)$ has a series-parallel decomposition by Proposition~\ref{prop:seriesparallel}.
\end{claimproof}


    Let $T'$ be a series-parallel decomposition of $(G, a, b)$.
    
    \begin{smallclaim}
        Suppose $(K, a, b) \in V(T')$ has no child in $T'$ with terminals $(a,b)$. Then $K$ is a subgraph of either $H$ or $G - (H - \{a,b\})$.
    \end{smallclaim}
    \begin{claimproof}
        If $K$ comprises a single edge, then it is a subgraph of $H$ or of $G - (H - \{a, b\})$ depending on whether $a$ and $b$ are neighbors in $H$. Otherwise, the induced decomposition on $\gwt{K}$ in $T'$ is nontrivial, and we must have $\circ_{T', \gwt{K}} = \circ_s$ as otherwise $\gwt{K}$ would have a child with terminals $(a, b)$. By Lemma~\ref{lemma:series_parallel_blocks}, the block-cut graph of $K$ comprises a path, with the blocks containing $a$ and $b$ as distinct endpoints. If $K$ contains vertices of both $H - \{a, b\}$ and $G - H$, one of these blocks, say $B$, must contain vertices from both. By Lemma~\ref{lemma:terminals_disconnect} applied to $\gwt{H}$ with respect to $T$, $\{a, b\}$ must separate $H - \{a, b\}$ from $G - H$, but $B$ cannot contain both $a$ and $b$, so it remains connected after the removal of $\{a, b\}$, a contradiction.
    \end{claimproof}

    By Lemma~\ref{lemma:combining_parallel}, we can write
\[
    (G, a, b) = (K_1, a, b) \circ_p \cdots \circ_p (K_\ell, a, b),
\]
where $\gwt{K_1}$, \ldots, $\gwt{K_\ell} \in V(T')$ are such that their children do not have terminals $(a, b)$. We may order these graphs so that for some $1 \le m \leq \ell$, the $K_i$ with $i < m$ are subgraphs of $G - (H - \{a, b\})$ and the $K_i$ with $i \geq m$ are subgraphs of $H$. Then $\bigcup_{i = m}^\ell K_i = H$, so
\[
    (G, a, b) = (K_1, a, b) \circ_p \ldots \circ_p (K_{m-1}, a, b) \circ_p (H, a, b).
\]
Since $(H, a, b)$ is the minimal complex descendant of $G$ in $T$, the induced decompositions of its children are simple and $\circ_{T, \gwt{H}} = \circ_p$. Let $(H_0, a, b)$ and $(H_1, a, b)$ be its children in $T$. Then
\[
    (G, a, b) = \pr{(K_1, a, b) \circ_p \cdots \circ_p (K_{m-1}, a, b) \circ_p (H_0, a, b)} \circ_p (H_1, a, b).
\]
Consider any $K_i$ with $i < m$. Since it is a subgraph of $G - (H - \{a, b\})$, its induced decomposition in $T'$ cannot be complex, as otherwise $H$ and $K_i$ 
contradict Proposition~\ref{prop:minial_complex_descendant}. Furthermore, $(K_i, a, b)$ must be bridged, as otherwise, by Lemma~\ref{lemma:complexity}, it would contain a bipath between $a$ and $b$ which is internally disjoint from the two bipaths between $a$ and $b$ contained in $H$ (also given by Lemma~\ref{lemma:complexity}), and so $G$ would contain as a subgraph a member of $\mathcal{F}_2$.

Hence each $(K_i, a, b)$ is bridged and has a simple induced decomposition in $T'$ of complexity $0$. By connecting each of these decompositions with the induced decomposition of $(H_0, a, b)$ from $T$, we obtain a simple series-parallel decomposition of $(K_1, a, b) \circ_p \ldots \circ_p (K_{m-1}, a, b) \circ_p (H_0, a, b)$. Finally, we may apply Proposition~\ref{prop:rotation} to show that $G$ admits a simple series-parallel decomposition with $a$ as a terminal.
\end{proof}

\subsection{The general case} In this subsection we complete the proof of implication \ref{item:families} $\Longrightarrow$ \ref{item:GSP} in Theorem~\ref{theo:classification} by handling non-$2$-connected graphs.

\begin{lemma} Let $G$ be a connected graph which does not contain any member of $\mathcal{F}_3$ as a subgraph. Suppose $H$, $K$ are distinct leaf blocks of $G$ with respective cut-vertices $v_H$, $v_K$. If $T_H$, $T_K$ are series-parallel decompositions of $H$ and $K$ respectively, and $(H', a, b) \in T_H$ has a complex induced decomposition with $v_H \not \in V(H')$, then any $\gwt{K'} \in T_K$ with a complex induced decomposition must have $v_K \in V(K')$. \label{lemma:disjoint_complex_blocks}
\end{lemma}
\begin{proof} By 2-connectedness of $H$, there exists a path from $a$ to $v_H$ in $G - \{b\}$. By Lemma~\ref{lemma:terminals_disconnect}, any such path of minimum length cannot contain any non-terminal vertex of $H'$. Likewise, we can find a path from $b$ to $v_H$ which does not contain a non-terminal vertex of $H'$.

If $K'$ does not contain $v_K$, we can similarly find paths from $v_K$ to each of its terminals which does not contain the non-terminal vertices of $K'$. Finally, we can take any path from $v_H$ to $v_K$ to construct paths between each pair of terminals of $H'$ and $K'$ each avoiding the non-terminal vertices of $H'$ and $K'$. Since each of $H'$, $K'$ contain a pair of bipaths between their terminals (Lemma~\ref{lemma:complexity}), $G$ will contain as a subgraph a member of $\mathcal{F}_3$, contradicting the assumption.
\end{proof}
\begin{theo} Suppose $G$ is a connected graph which does not contain any member of $\mathcal{F}_1 \cup \mathcal{F}_2 \cup \mathcal{F}_3$ as a subgraph. Then $G$ has a simple GSP decomposition.
\end{theo}
\begin{proof} We proceed by induction on the number of blocks in $G$. The base case when $G$ has only one block is handled by Proposition~\ref{prop:series_parallel_for_block}. 
Now suppose $G$ has at least $2$ blocks and let $H$, $K$ be distinct leaf blocks with respective cut-vertices $v_H$, $v_K$. By Proposition~\ref{prop:seriesparallel}, we can find series-parallel decompositions $T_H$, $T_H$ of $H$ and $K$ respectively with each having $v_H$, $v_K$ respectively as a terminal (we can make the terminals be $v_H/v_K$ and a neighbor, for instance).

If at least one of $T_H$, $T_K$ is simple, then by the inductive hypothesis we can find a simple decomposition of the rest of the graph, and then, by Lemma~\ref{lemma:blockmerge}, we get a simple decomposition of $G$. Otherwise, both $T_H$ and $T_K$ are complex, so Lemma~\ref{lemma:disjoint_complex_blocks} implies that at least one (without loss of generality, let it be $H$) has its minimal complex descendant contain $v_H$. Since $\gwt{H}$ has $v_H$ as a terminal, $v_H$ must be a terminal of this minimal complex descendant. Hence by Proposition~\ref{prop:series_parallel_for_block}, we can find a simple decomposition of $H$ with $v_H$ as a terminal, then again we can find a simple decomposition of the rest of the graph, and by Lemma~\ref{lemma:blockmerge}, we have a simple decomposition of $G$.
\end{proof}

\section{Searching subdivisions}\label{sec:dec_ts}

\subsection{Introduction}

In this section we prove implication \ref{item:GSP} $\Longrightarrow$ \ref{item:ts3} of Theorem~\ref{theo:classification}; i.e., we show that every graph $G$ with a simple GSP decomposition has topological inspection number at most $3$. We begin by noting that if a graph $G$ has a simple GSP decomposition, then so do all its subdivisions:

\begin{lemma}
Let $\gwt{G} = (G, a,b)$ be a graph with terminals and let $T$ be a GSP decomposition of $\gwt{G}$. Suppose that $G'$ is a subdivision of $G$. For each subgraph $D$ of $G$, let $D'$ denote the corresponding subgraph of $G'$ obtained by subdividing $D$.

Then $\gwt{G'} = (G, a,b)$ has a GSP decomposition $T'$ such that for all $\gwt{D} = (D,s,t) \in V(T)$, we have $\gwt{D'} = (D',s,t) \in V(T')$ and $\circ_{T, \gwt{D}} = \circ_{T', \gwt{D'}}$. Additionally, if $T$ is simple, then $T'$ is simple too.\label{lemma:subd_GSP}
\end{lemma}
\begin{proof} Replace each node $(D,s,t)$ in $T$ with $(D',s,t)$. The leaf nodes now correspond to subdivided edges (i.e., paths), which have series-parallel decompositions using only the series operation. Replacing each leaf node with its series-parallel decomposition gives the desired GSP decomposition $T'$. Note that the series-parallel decompositions replacing leaf nodes have complexity $0$. Since whether a graph with terminals is bridged is not changed by subdivision, the complexity of each subtree in the decomposition remains the same as before, and hence if $T$ is simple, then $T'$ is simple as well.
\end{proof}

In view of Lemma \ref{lemma:subd_GSP}, to prove implication \ref{item:GSP} $\Longrightarrow$ \ref{item:ts3} of Theorem~\ref{theo:classification}, we just need to argue that every graph $G$ with a simple GSP decomposition $T$ has a subdivision $G'$ such that $\s(G') \leq 3$. Indeed, by Lemma~\ref{lemma:subd_GSP}, this would imply that every subdivision of $G$ has a further subdivision with inspection number at most $3$, as desired.

\subsection{Aligned searches and quotient graphs}

Given a graph $G$ with a simple GSP decomposition $T$, we shall construct a subdivision $G'$ of $G$ with $\s(G') \leq 3$ inductively by ``amalgamating'' successful $3$-searches on subdivisions of the descendants of $G$ in $T$. In order for this amalgamation process to work, we need these $3$-searches to satisfy a certain property defined as follows:

\begin{defn}[Aligned Searches] A $k$-search $\mathcal{S}$ of length $\ell$ on a graph $G$ is \emphd{aligned} to an ordered pair of vertices $(a, b) \in V(G)^2$ if $a \in \PC_t(\mathcal{S})$ for each $1 \leq t \leq \ell$, and $b \not\in \FC_t(\mathcal{S})$ for each $0 \leq t < \ell$.\label{defn:aligned}

If $A \subseteq V(G)$ and $\mathcal{S}$ is a $k$-search of length $\ell$, we say that $\mathcal{S}$ is \emphd{aligned to $(a, b) \in V(G)^2$ over $A$} if $a \in \PC_t(\mathcal{S}, A)$ for each $1 \leq t \leq \ell$, and $b \not\in \FC_i(\mathcal{S}, A)$ for each $0 \leq t < \ell$.

Here $\PC_t(\mathcal{S}, A)$ and $\FC_t(\mathcal{S}, A)$ refer to the sets of pre-cleared and fully cleared vertices if the set $A$ is initially cleared (see Definition \ref{defn:search}).
\end{defn}



The next definition provides the terminology and notation that will be used in describing the amalgamation process: 

\begin{defn}[Quotient Graphs]
Let $G$ be a graph and let $\mathcal{E}$ be an equivalence relation on $V(G)$. The \emphd{quotient graph} $G/\mathcal{E}$ is the graph with vertex set $V(G)/\mathcal{E}$ that includes an edge between distinct $\mathcal{E}$-equivalence classes $c_1$, $c_2 \in V(G)/\mathcal{E}$ if and only if there are adjacent vertices $v_1 \in c_1$ and $v_2 \in c_2$. 

A subset $X \subseteq V(G)$ is \emphd{$\mathcal{E}$-invariant} if it is a union of $\mathcal{E}$-equivalence classes. For an $\mathcal{E}$-invariant set $X \subseteq V(G)$, we write $X/\mathcal{E}$ to denote the set of all $\mathcal{E}$-equivalence classes contained in $X$.

Define two maps $\wedge_\mathcal{E}$, $\vee_\mathcal{E} \colon \mathcal{P}(V(G)) \to \mathcal{P}(V(G)/\mathcal{E})$ via
\[
    \wedge_\mathcal{E}(X) = \{c \in V(G)/\mathcal{E} \,:\, c \subseteq X\} \quad \text{and} \quad  \vee_\mathcal{E}(X) = \{c \in V(G)/\mathcal{E} \,:\, c \cap X \neq \0\}.
\]
Note that if $X$ is $\mathcal{E}$-invariant, then $\wedge_\mathcal{E}(X) = \vee_\mathcal{E}(X) = X/\mathcal{E}$.

Let $\mathcal{S}$ be a $k$-search of length $\ell$ on $G$ and let $A \subseteq V(G)$. We say that $\mathcal{S}$ is \emphd{$\mathcal{E}$-invariant over $A$} if for all $1 \leq t \leq \ell$, the set $\PC_t(\mathcal{S}, A)$ is $\mathcal{E}$-invariant. Given a $k$-search $\mathcal{S} = (S_t)_{t=1}^\ell$ on $G$, we let $\vee_\mathcal{E}(\mathcal{S}) = (\vee_\mathcal{E}(S_t))_{t=1}^\ell$. 
\end{defn}

If $\mathcal{E}$ is an equivalence relation on $V(G)$ and $\mathcal{S}$ is a $k$-search on $G$, then $\vee_\mathcal{E}(\mathcal{S})$ is a $k$-search on $G/\mathcal{E}$. The next proposition allows us to compute the sets of pre-cleared and fully cleared vertices for it.

\begin{prop} Let $G$ be a graph and let $\mathcal{E}$ be an equivalence relation on $V(G)$. Suppose that $A \subseteq V(G)$ and $\mathcal{S}$ is a $k$-search of length $\ell$ on $G$ that is $\mathcal{E}$-invariant over $A$. Then:
\begin{itemize}
    \item $\FC_{t}(\vee_\mathcal{E}(\mathcal{S}), \wedge_\mathcal{E}(A)) = \wedge_\mathcal{E}(\FC_{t}(\mathcal{S}, A))$ for all $0 \leq t \leq \ell$;
    
    \item $\PC_t(\vee_\mathcal{E}(\mathcal{S}), \wedge_\mathcal{E}(A))= \PC_t(\mathcal{S}, A)/\mathcal{E}$ for all $1 \leq t \leq \ell$.
\end{itemize}
\label{prop:quotient_search}
\end{prop}
\begin{proof} We shall prove the following two implications for all $1 \leq t \leq \ell$:
\begin{enumerate}[label=\ep{\emph{\alph*}}]
    \item\label{item:quot_a} If $\FC_{t-1}(\vee_\mathcal{E}(\mathcal{S}), \wedge_\mathcal{E}(A)) = \wedge_\mathcal{E}(\FC_{t-1}(\mathcal{S}, A))$, then $\PC_t(\vee_\mathcal{E}(\mathcal{S}), \wedge_\mathcal{E}(A))= \PC_t(\mathcal{S}, A)/\mathcal{E}$.
    
    \item If $\PC_t(\vee_\mathcal{E}(\mathcal{S}), \wedge_\mathcal{E}(A))= \PC_t(\mathcal{S}, A)/\mathcal{E}$, then $\FC_{t}(\vee_\mathcal{E}(\mathcal{S}), \wedge_\mathcal{E}(A)) = \wedge_\mathcal{E}(\FC_{t}(\mathcal{S}, A))$.
\end{enumerate}
Together with the observation that $\FC_0(\vee_\mathcal{E}(\mathcal{S}), \wedge_\mathcal{E}(A)) = \wedge_\mathcal{E}(A) = \wedge_\mathcal{E}(\FC_{0}(\mathcal{S}, A))$, these immediately yield the desired result.

First suppose that $\FC_{t-1}(\vee_\mathcal{E}(\mathcal{S}), \wedge_\mathcal{E}(A)) = \wedge_\mathcal{E}(\FC_{t-1}(\mathcal{S}, A))$. Then, by definition,
\[
    \PC_t(\vee_\mathcal{E}(\mathcal{S}), \wedge_\mathcal{E}(A)) = \wedge_\mathcal{E}(\FC_{t-1}(\mathcal{S}, A)) \cup \vee_\mathcal{E}(S_t) \quad \text{and} \quad \PC_t(\mathcal{S}, A) = \FC_{t-1}(\mathcal{S}, A) \cup S_t.
\]
Since the set $\PC_t(\mathcal{S}, A)$ is $\mathcal{E}$-invariant, every $\mathcal{E}$-equivalence class $c$ is either contained in it or disjoint from it. Therefore,
\begin{align*}
    c \in \PC_t(\mathcal{S}, A)/\mathcal{E} \, &\Longleftrightarrow \, c \subseteq \FC_{t-1}(\mathcal{S}, A)  \text{ or } c \cap S_t \neq \0 \\
    & \Longleftrightarrow \, c \in \wedge_\mathcal{E}(\FC_{t-1}(\mathcal{S}, A))  \text{ or } c \in \vee_\mathcal{E}(S_t)  \\
    & \Longleftrightarrow \, c \in \PC_t(\vee_\mathcal{E}(\mathcal{S}), \wedge_\mathcal{E}(A)).
\end{align*}

Now suppose that $\PC_t(\vee_\mathcal{E}(\mathcal{S}), \wedge_\mathcal{E}(A))= \PC_t(\mathcal{S}, A)/\mathcal{E}$. Then $\FC_t(\vee_\mathcal{E}(\mathcal{S}), \wedge_\mathcal{E}(A))$ comprises exactly those equivalence classes $c \in \PC_t(\mathcal{S}, A)/\mathcal{E}$ that satisfy $c \cap \partial_G(\PC_t(\mathcal{S}, A)) = \0$, i.e., that are fully contained in $\FC_t(\mathcal{S},A)$. Therefore, $\FC_t(\vee_\mathcal{E}(\mathcal{S}), \wedge_\mathcal{E}(A)) = \wedge_\mathcal{E}(\FC_t(\mathcal{S},A))$, as desired.
\end{proof}

\subsection{Clearing balls}

Recall that for $v \in V(G)$ and $r \in \mathbb{N}$, $B_G(v, r)$ denotes the ball of radius $r$ around $v$ in $G$. In this subsection we establish two useful auxiliary constructions that allow us to clear a ball of large radius around a vertex in a subdivision of a given graph.

\begin{lemma} Let $G$ be a graph and let $v$, $w \in V(G)$ be distinct vertices. Let $d$ be the degree of $w$ in $G$. Fix some $r \in \mathbb{N}$ and let $H$ be a subdivision of $G$ such that each edge incident to $w$ is subdivided at least $2^{d-1}r + 1$ times. Then there exists a $3$-search $\mathcal{S}$ on $H$ of length $\ell = (2^{d} - 1)r$ such that $w \in S_t$ for all $1 \leq t \leq \ell$, $\mathcal{S}$ is aligned to $(w, v)$, and $\FC_*(\mathcal{S}) \supseteq B_H(w, r)$. \label{lemma:clearing_from_bottom}
\end{lemma}
\begin{proof}
    Let $P_1$, \ldots, $P_d$ be the paths in $H$ replacing the edges of $G$ incident to $w$. Denote the endpoint of $P_i$ distinct from $w$ by $u_i$. Let $H'$ be the graph obtained from $H$ by removing all the vertices in $V(P_1) \cup \ldots \cup V(P_d)$ except for $u_1$, \ldots, $u_d$, and let $H^\ast$ be the graph obtained as a disjoint union of $H'$ and $d$ paths $P_1^\ast$, \ldots, $P_d^\ast$, where each $P_i^\ast$ is of the same length as $P_i$. Denote the two endpoints of $P_i^\ast$ by $w_i$ and $u_i^\ast$ and define an equivalence relation $\mathcal{E}$ on $V(H^\ast)$ by making the following sets the only nontrivial equivalence classes: $\{w_1, \ldots, w_d\}$, $\{u_1, u_1^\ast\}$, \ldots, $\{u_d, u_d^\ast\}$. 
    Then there is an obvious isomorphism $H \cong H^\ast/\mathcal{E}$, which allows us to identify $H$ with $H^\ast/\mathcal{E}$.

    For $1 \leq i \leq d$, denote by $V_i[a, b]$ the set of all vertices in $P_i^\ast$ at distance between $a$ and $b$ \ep{inclusive} from $w_i$. Let $\mathcal{S}_i$ be the $(d+2)$-search on $H^\ast$ of length $2^{d-i}r$, where for each $1 \le t \le 2^{d-i}r$,
    \[
        (S_i)_t = \{w_1, \ldots, w_d\} \cup V_i[t, t+1].
    \]
    Note that since the length of $P_i^\ast$ is at least $2^{d-1}r+2$, $u_i^\ast \not \in (S_i)_t$. Let $\mathcal{S}^\ast = \mathcal{S}_1 \concat \cdots \concat \mathcal{S}_d$, where $\concat$ denotes concatenation of sequences, and define $\mathcal{S} = \vee_\mathcal{E}(\mathcal{S}^\ast)$. We claim that $\mathcal{S}$ is as desired.
    
    It is clear that $\mathcal{S}$ is a $3$-search \ep{since the vertices $w_1$, \ldots, $w_d$ represent the same equivalence class}. The length of $\mathcal{S}$ is equal to the length of $\mathcal{S}^\ast$, which is
    \[
        2^{d-1}r + 2^{d-2}r + \cdots + r = (2^d - 1)r.
    \]
    Also, for each $t$, $w$ is the unique vertex of $G$ in $S_t$. This implies that $\mathcal{S}$ is aligned to $(w,v)$. 
    
    Next we observe that the search $\mathcal{S}^\ast$ is $\mathcal{E}$-invariant. Indeed, the equivalence class $\{w_1, \ldots, w_d\}$ is contained in every $S^\ast_t$, so all its elements are always pre-cleared. On the other hand, the vertices in an equivalence class of the form $\{u_i, u_i^\ast\}$ are never pre-cleared.

    It is straightforward to verify that $V_i[0, 2^{d-i}r] \subseteq \FC_\ast(\mathcal{S}_i)$ and hence $V_i[0, 2^{d-i}r] \subseteq \FC_{(2^d - 2^{d-i})r}(\mathcal{S}^\ast)$. It follows that 
    for each $t \ge 1$,
    \[V_i[0, 2^{d-i}r -t + 1] \subseteq \PC_{(2^d - 2^{d-i})r + t}(\mathcal{S}^\ast) \quad \text{and} \quad V_i[0, 2^{d-i}r -t] \subseteq \FC_{(2^{d} - 2^{d-i})r + t}(\mathcal{S}^\ast).\]
    By applying this with $t = (2^{d-i} - 1)r$, we get that $V_i[0, r] \subseteq \FC_*(\mathcal{S}^\ast)$ for each $i$. Hence, by Proposition~\ref{prop:quotient_search}, $\FC_\ast(\mathcal{S})$ contains all the vertices of $H$ at distance at most $r$ from $w$, as desired.
\end{proof}

\begin{lemma} Let $G$ be a graph and $v$, $w \in V(G)$ be distinct vertices. Let $d$ be the degree of $v$ in $G$. Fix some $r \in \mathbb{N}$ and let $H$ be a subdivision of $G$ such that each edge incident to $v$ is subdivided at least $2^{d-1}r$ times. Let $A = V(H) \setminus B_H(v, r)$. Then there exists a $3$-search $\mathcal{S}$ on $H$ of length $\ell=(2^d-1)r$ such that $v \in S_t$ for all $1 \leq t \leq \ell$, $\mathcal{S}$ is aligned to $(w, v)$ over $A$, and $\FC_*(\mathcal{S}, A) = V(H)$. \label{lemma:clearing_to_top}
\end{lemma}
\begin{proof}
    Let $P_1$, \ldots, $P_d$ be the paths in $H$ replacing the edges of $G$ incident to $v$. Denote the endpoint of $P_i$ distinct from $v$ by $u_i$. Let $H'$ be the graph obtained from $H$ by removing all the vertices in $V(P_1) \cup \ldots \cup V(P_d)$ except for $u_1$, \ldots, $u_d$, and let $H^\ast$ be the graph obtained as a disjoint union of $H'$ and $d$ paths $P_1^\ast$, \ldots, $P_d^\ast$, where each $P_i^\ast$ is of the same length as $P_i$. Denote the two endpoints of $P_i^\ast$ by $v_i$ and $u_i^\ast$ and define an equivalence relation $\mathcal{E}$ on $V(H^\ast)$ by making the following sets the only nontrivial equivalence classes: $\{v_1, \ldots, v_d\}$, $\{u_1, u_1^\ast\}$, \ldots, $\{u_d, u_d^\ast\}$. 
    Then there is an obvious isomorphism $H \cong H^\ast/\mathcal{E}$, which allows us to identify $H$ with $H^\ast/\mathcal{E}$. Let $A^\ast \subseteq V(H^\ast)$ be the $\mathcal{E}$-invariant set such that $A^\ast/\mathcal{E} = A$. Explicitly, $A^\ast$ comprises $V(H')$ together with the vertices on each of the paths $P_i$ at distance more than $r$ from $v_i$.
    

    For $1 \leq i \leq d$, denote by $V_i[a, b]$ the set of all vertices in $P_i^\ast$ at distance between $a$ and $b$ \ep{inclusive} from $v_i$. Let $\mathcal{S}_i$ be the $(d+2)$-search on $H^\ast$ of length $2^{i - 1}r$ where for each $1 \le t \le 2^{i - 1}r$,
    \[
        (S_i)_t = \{v_1, \ldots, v_d\} \cup V_i[2^{i - 1}r - t, 2^{i - 1}r - t + 1].
    \]
    Let $\mathcal{S}^\ast = \mathcal{S}_1 \concat \cdots \concat \mathcal{S}_d$ and $\mathcal{S} = \vee_\mathcal{E}(\mathcal{S}^\ast)$. We claim that $\mathcal{S}$ is as desired. 
    
    It is clear that $\mathcal{S}$ is a $3$-search \ep{since the vertices $v_1$, \ldots, $v_d$ represent the same equivalence class}. The length of $\mathcal{S}$ is equal to the length of $\mathcal{S}^\ast$, which is
    \[
        r + 2r + 2^2 r + \cdots + 2^{d-1}r = (2^d - 1)r.
    \]
    By construction, $v \in S_t$ for all $t$.
    
    For $t < (2^{i-1}-1)r$ (the first turn of $\mathcal{S}_i$ in $\mathcal{S}'$), we have $V(P_i) \setminus \FC_{t}(\mathcal{S}^\ast, A^\ast) = V_i[0, r + t]$. Thus,
    \[
        V(P_i) \setminus \FC_{(2^{i-1}-1)r}(\mathcal{S}^\ast, A^\ast) = V_i[0, 2^{i-1}r].
    \]
    By turn $(2^i-1)r$, all the vertices in $P_i$ will become fully cleared and will remain so for the remainder of $\mathcal{S}^\ast$. Since $V(H') \subseteq A^\ast$, we conclude that $\FC_*(\mathcal{S}^\ast, A^\ast) = V(H^\ast)$.

    Note that each $v_i$ is in $S_t^\ast$ for all $t$, while $u_i$ and $u_i^\ast$ are in $\PC_t(\mathcal{S}^\ast, A^\ast)$ for all $t$. This shows that $\mathcal{S}^\ast$ is $\mathcal{E}$-invariant over $A^\ast$. Hence, by Proposition~\ref{prop:quotient_search}, $\FC_\ast(\mathcal{S}, A) = V(H)$, as desired. 
    Finally, $\mathcal{S}$ is aligned to $(w, v)$ over $A$, since every vertex of $G$ except $v$ is fully cleared throughout, while $v$ does not become fully cleared until the last move.
\end{proof}

    \subsection{Amalgamation} In this subsection we combine aligned searches by the GSP operations. 

    \begin{prop} Let $(G_0, a, c)$ and $(G_1, c, b)$ be connected graphs with terminals such that $V(G_0) \cap V(G_1) = \{c\}$ and let $(G, a, b) = (G_0, a, c) \circ_s (G_1, c, b)$. Suppose that for each $i \in \{0, 1\}$, there is a subdivision $H_i$ of $G_i$ admitting a successful $3$-search $\mathcal{S}_i$ aligned to the corresponding terminals. Then there is a subdivision of $G$ admitting a successful $3$-search aligned to $(a, b)$.
    \label{prop:amalgamation_s}
    \end{prop}
\begin{figure}[t]
\begin{center}
\begin{tikzpicture}[scale=0.9,x=0.75pt,y=0.75pt,yscale=-1,xscale=1]

\draw    (73.12,251.47) .. controls (46.12,233.47) and (33.12,179.47) .. (73.12,149.47) ;
\draw    (73.12,251.47) .. controls (102.12,225.47) and (104.12,181.47) .. (73.12,149.47) ;
\draw  [fill={rgb, 255:red, 0; green, 0; blue, 0 }  ,fill opacity=1 ] (70.12,250.69) .. controls (70.12,249.15) and (71.36,247.91) .. (72.9,247.91) .. controls (74.43,247.91) and (75.68,249.15) .. (75.68,250.69) .. controls (75.68,252.22) and (74.43,253.47) .. (72.9,253.47) .. controls (71.36,253.47) and (70.12,252.22) .. (70.12,250.69) -- cycle ;
\draw  [fill={rgb, 255:red, 0; green, 0; blue, 0 }  ,fill opacity=1 ] (69.56,150.47) .. controls (69.56,148.93) and (70.8,147.69) .. (72.34,147.69) .. controls (73.87,147.69) and (75.12,148.93) .. (75.12,150.47) .. controls (75.12,152) and (73.87,153.25) .. (72.34,153.25) .. controls (70.8,153.25) and (69.56,152) .. (69.56,150.47) -- cycle ;
\draw    (72.9,150.69) .. controls (45.9,132.69) and (32.9,78.69) .. (72.9,48.69) ;
\draw    (73.12,151.47) .. controls (102.12,125.47) and (104.12,81.47) .. (73.12,49.47) ;
\draw  [fill={rgb, 255:red, 0; green, 0; blue, 0 }  ,fill opacity=1 ] (69.56,50.47) .. controls (69.56,48.93) and (70.8,47.69) .. (72.34,47.69) .. controls (73.87,47.69) and (75.12,48.93) .. (75.12,50.47) .. controls (75.12,52) and (73.87,53.25) .. (72.34,53.25) .. controls (70.8,53.25) and (69.56,52) .. (69.56,50.47) -- cycle ;
\draw [color={rgb, 255:red, 24; green, 61; blue, 247 }  ,draw opacity=1, very thick ]   (85,249) .. controls (111.71,240.59) and (128.6,203.6) .. (106.17,172.87) ;
\draw [shift={(105.12,171.47)}, rotate = 412.25] [color={rgb, 255:red, 24; green, 61; blue, 247 }  ,draw opacity=1, very thick ][line width=0.75]    (10.93,-3.29) .. controls (6.95,-1.4) and (3.31,-0.3) .. (0,0) .. controls (3.31,0.3) and (6.95,1.4) .. (10.93,3.29)   ;
\draw [color={rgb, 255:red, 24; green, 61; blue, 247 }  ,draw opacity=1, very thick ]   (75,239) .. controls (95.59,212.15) and (90.4,184.87) .. (74.37,165.91) ;
\draw [shift={(73.12,164.47)}, rotate = 408.18] [color={rgb, 255:red, 24; green, 61; blue, 247 }  ,draw opacity=1, very thick ][line width=0.75]    (10.93,-3.29) .. controls (6.95,-1.4) and (3.31,-0.3) .. (0,0) .. controls (3.31,0.3) and (6.95,1.4) .. (10.93,3.29)   ;
\draw [color={rgb, 255:red, 24; green, 61; blue, 247 }  ,draw opacity=1, very thick ]   (56,246) .. controls (26.71,219.02) and (26.13,183.9) .. (48.71,164.63) ;
\draw [shift={(50.12,163.47)}, rotate = 501.63] [color={rgb, 255:red, 24; green, 61; blue, 247 }  ,draw opacity=1, very thick ][line width=0.75]    (10.93,-3.29) .. controls (6.95,-1.4) and (3.31,-0.3) .. (0,0) .. controls (3.31,0.3) and (6.95,1.4) .. (10.93,3.29)   ;
\draw [line width=3]    (134,154) -- (206.12,154.44) ;
\draw [shift={(211.12,154.47)}, rotate = 180.35] [color={rgb, 255:red, 0; green, 0; blue, 0 }  ][line width=3]    (20.77,-6.25) .. controls (13.2,-2.65) and (6.28,-0.57) .. (0,0) .. controls (6.28,0.57) and (13.2,2.66) .. (20.77,6.25)   ;
\draw [color={rgb, 255:red, 24; green, 61; blue, 247 }  ,draw opacity=1, very thick ]   (264.12,251.47) .. controls (237.12,233.47) and (224.12,179.47) .. (264.12,149.47) ;
\draw [color={rgb, 255:red, 24; green, 61; blue, 247 }  ,draw opacity=1, very thick ]   (264.12,251.47) .. controls (293.12,225.47) and (295.12,181.47) .. (264.12,149.47) ;
\draw  [color={rgb, 255:red, 24; green, 61; blue, 247 }  ,draw opacity=1, very thick ][fill={rgb, 255:red, 24; green, 61; blue, 247 }  ,fill opacity=1 ] (261.12,250.69) .. controls (261.12,249.15) and (262.36,247.91) .. (263.9,247.91) .. controls (265.43,247.91) and (266.68,249.15) .. (266.68,250.69) .. controls (266.68,252.22) and (265.43,253.47) .. (263.9,253.47) .. controls (262.36,253.47) and (261.12,252.22) .. (261.12,250.69) -- cycle ;
\draw  [fill={rgb, 255:red, 0; green, 0; blue, 0 }  ,fill opacity=1 ] (260.56,150.47) .. controls (260.56,148.93) and (261.8,147.69) .. (263.34,147.69) .. controls (264.87,147.69) and (266.12,148.93) .. (266.12,150.47) .. controls (266.12,152) and (264.87,153.25) .. (263.34,153.25) .. controls (261.8,153.25) and (260.56,152) .. (260.56,150.47) -- cycle ;
\draw    (263.9,150.69) .. controls (236.9,132.69) and (223.9,78.69) .. (263.9,48.69) ;
\draw    (264.12,151.47) .. controls (293.12,125.47) and (295.12,81.47) .. (264.12,49.47) ;
\draw  [fill={rgb, 255:red, 0; green, 0; blue, 0 }  ,fill opacity=1 ] (260.56,50.47) .. controls (260.56,48.93) and (261.8,47.69) .. (263.34,47.69) .. controls (264.87,47.69) and (266.12,48.93) .. (266.12,50.47) .. controls (266.12,52) and (264.87,53.25) .. (263.34,53.25) .. controls (261.8,53.25) and (260.56,52) .. (260.56,50.47) -- cycle ;
\draw [color={rgb, 255:red, 24; green, 61; blue, 247 }  ,draw opacity=1, very thick ]   (277,148) .. controls (303.71,139.59) and (320.6,102.6) .. (298.17,71.87) ;
\draw [shift={(297.12,70.47)}, rotate = 412.25] [color={rgb, 255:red, 24; green, 61; blue, 247 }  ,draw opacity=1, very thick ][line width=0.75]    (10.93,-3.29) .. controls (6.95,-1.4) and (3.31,-0.3) .. (0,0) .. controls (3.31,0.3) and (6.95,1.4) .. (10.93,3.29)   ;
\draw [color={rgb, 255:red, 24; green, 61; blue, 247 }  ,draw opacity=1, very thick ]   (267,138) .. controls (287.59,111.15) and (282.4,83.87) .. (266.37,64.91) ;
\draw [shift={(265.12,63.47)}, rotate = 408.18] [color={rgb, 255:red, 24; green, 61; blue, 247 }  ,draw opacity=1, very thick ][line width=0.75]    (10.93,-3.29) .. controls (6.95,-1.4) and (3.31,-0.3) .. (0,0) .. controls (3.31,0.3) and (6.95,1.4) .. (10.93,3.29)   ;
\draw [color={rgb, 255:red, 24; green, 61; blue, 247 }  ,draw opacity=1, very thick ]   (248,145) .. controls (218.71,118.02) and (218.13,82.9) .. (240.71,63.63) ;
\draw [shift={(242.12,62.47)}, rotate = 501.63] [color={rgb, 255:red, 24; green, 61; blue, 247 }  ,draw opacity=1, very thick ][line width=0.75]    (10.93,-3.29) .. controls (6.95,-1.4) and (3.31,-0.3) .. (0,0) .. controls (3.31,0.3) and (6.95,1.4) .. (10.93,3.29)   ;
\draw [line width=3]    (343,154) -- (415.12,154.44) ;
\draw [shift={(420.12,154.47)}, rotate = 180.35] [color={rgb, 255:red, 0; green, 0; blue, 0 }  ][line width=3]    (20.77,-6.25) .. controls (13.2,-2.65) and (6.28,-0.57) .. (0,0) .. controls (6.28,0.57) and (13.2,2.66) .. (20.77,6.25)   ;
\draw [color={rgb, 255:red, 24; green, 61; blue, 247 }  ,draw opacity=1, very thick ]   (473.12,251.47) .. controls (446.12,233.47) and (433.12,179.47) .. (473.12,149.47) ;
\draw [color={rgb, 255:red, 24; green, 61; blue, 247 }  ,draw opacity=1, very thick ]   (473.12,251.47) .. controls (502.12,225.47) and (504.12,181.47) .. (473.12,149.47) ;
\draw  [color={rgb, 255:red, 24; green, 61; blue, 247 }  ,draw opacity=1, very thick ][fill={rgb, 255:red, 24; green, 61; blue, 247 }  ,fill opacity=1 ] (470.12,250.69) .. controls (470.12,249.15) and (471.36,247.91) .. (472.9,247.91) .. controls (474.43,247.91) and (475.68,249.15) .. (475.68,250.69) .. controls (475.68,252.22) and (474.43,253.47) .. (472.9,253.47) .. controls (471.36,253.47) and (470.12,252.22) .. (470.12,250.69) -- cycle ;
\draw  [color={rgb, 255:red, 24; green, 61; blue, 247 }  ,draw opacity=1, very thick ][fill={rgb, 255:red, 24; green, 61; blue, 247 }  ,fill opacity=1 ] (469.56,150.47) .. controls (469.56,148.93) and (470.8,147.69) .. (472.34,147.69) .. controls (473.87,147.69) and (475.12,148.93) .. (475.12,150.47) .. controls (475.12,152) and (473.87,153.25) .. (472.34,153.25) .. controls (470.8,153.25) and (469.56,152) .. (469.56,150.47) -- cycle ;
\draw [color={rgb, 255:red, 24; green, 61; blue, 247 }  ,draw opacity=1, very thick ]   (472.9,150.69) .. controls (445.9,132.69) and (432.9,78.69) .. (472.9,48.69) ;
\draw [color={rgb, 255:red, 24; green, 61; blue, 247 }  ,draw opacity=1, very thick ]   (473.12,151.47) .. controls (502.12,125.47) and (504.12,81.47) .. (473.12,49.47) ;
\draw  [color={rgb, 255:red, 24; green, 61; blue, 247 }  ,draw opacity=1, very thick ][fill={rgb, 255:red, 24; green, 61; blue, 247 }  ,fill opacity=1 ] (469.56,50.47) .. controls (469.56,48.93) and (470.8,47.69) .. (472.34,47.69) .. controls (473.87,47.69) and (475.12,48.93) .. (475.12,50.47) .. controls (475.12,52) and (473.87,53.25) .. (472.34,53.25) .. controls (470.8,53.25) and (469.56,52) .. (469.56,50.47) -- cycle ;

\draw (58,192.4) node [anchor=north west][inner sep=0.75pt]    {$H_{0}$};
\draw (62,92.4) node [anchor=north west][inner sep=0.75pt]    {$H_{1}$};
\draw (58,246.4) node [anchor=north west][inner sep=0.75pt]    {$a$};
\draw (54,144.4) node [anchor=north west][inner sep=0.75pt]    {$c$};
\draw (58,33.4) node [anchor=north west][inner sep=0.75pt]    {$b$};
\draw (116,222.4) node [anchor=north west][inner sep=0.75pt]  [color={rgb, 255:red, 24; green, 61; blue, 247 }  ,opacity=1 ]  {$\mathcal{S}_{0}$};
\draw (252,192.4) node [anchor=north west][inner sep=0.75pt]  [color={rgb, 255:red, 24; green, 61; blue, 247 }  ,opacity=1 ]  {$H_{0}$};
\draw (249,92.4) node [anchor=north west][inner sep=0.75pt]    {$H_{1}$};
\draw (249,246.4) node [anchor=north west][inner sep=0.75pt]  [color={rgb, 255:red, 24; green, 61; blue, 247 }  ,opacity=1 ]  {$a$};
\draw (245,144.4) node [anchor=north west][inner sep=0.75pt]    {$c$};
\draw (249,33.4) node [anchor=north west][inner sep=0.75pt]    {$b$};
\draw (315,104.4) node [anchor=north west][inner sep=0.75pt]  [color={rgb, 255:red, 24; green, 61; blue, 247 }  ,opacity=1 ]  {$\mathcal{S}_{1}$};
\draw (461,192.4) node [anchor=north west][inner sep=0.75pt]  [color={rgb, 255:red, 24; green, 61; blue, 247 }  ,opacity=1 ]  {$H_{0}$};
\draw (462,92.4) node [anchor=north west][inner sep=0.75pt]  [color={rgb, 255:red, 24; green, 61; blue, 247 }  ,opacity=1 ]  {$H_{1}$};
\draw (458,246.4) node [anchor=north west][inner sep=0.75pt]  [color={rgb, 255:red, 24; green, 61; blue, 247 }  ,opacity=1 ]  {$a$};
\draw (454,144.4) node [anchor=north west][inner sep=0.75pt]  [color={rgb, 255:red, 24; green, 61; blue, 247 }  ,opacity=1 ]  {$c$};
\draw (458,33.4) node [anchor=north west][inner sep=0.75pt]  [color={rgb, 255:red, 24; green, 61; blue, 247 }  ,opacity=1 ]  {$b$};
\end{tikzpicture}
\end{center}
\caption{A pictorial representation of the construction in Proposition~\ref{prop:amalgamation_s}.}\label{fig:amalgam_s}
\end{figure}
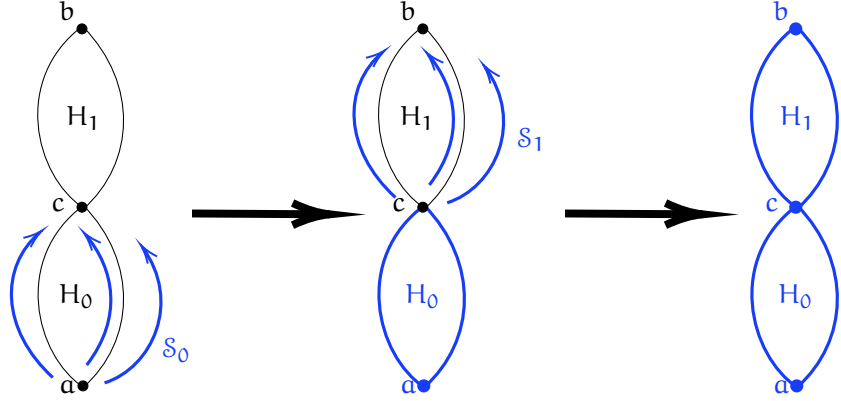
\begin{proof}
    The proof is illustrated in Fig.~\ref{fig:amalgam_s}. The desired subdivision of $G$ is simply the union of $H_0$ and $H_1$, and the desired $3$-search is the concatenation $\mathcal{S}_0 \concat \mathcal{S}_1$. To aid in our analysis, we let $H$ be the \emph{disjoint} union of the graphs $H_0$ and $H_1$ and denote the copy of $c$ in each $H_i \subset H$ by $c_i$.  
    Let $\mathcal{E}$ be the equivalence relation on $V(H)$ with $\{c_0, c_1\}$ as the only nontrivial equivalence class. Then $H/\mathcal{E}$ can be naturally identified with the \ep{non-disjoint} union of $H_0$ and $H_1$. 

    Construct searches $\mathcal{S}_0'$ and $\mathcal{S}_1'$ on $H$ from $\mathcal{S}_0$ and $\mathcal{S}_1$ respectively by replacing every appearance of $c$ by both $c_0$ and $c_1$, 
    and let $\mathcal{S} = \mathcal{S}_0' \concat \mathcal{S}_1'$. Then $\lor_{\mathcal{E}}(\mathcal{S})$ is a $3$-search on $H / \mathcal{E}$. \ep{If we identify $H/\mathcal{E}$ with the union of $H_0$ and $H_1$, then $\lor_{\mathcal{E}}(\mathcal{S}) = \mathcal{S}_0 \concat \mathcal{S}_1$.}

    We now check that $\mathcal{S}$ is $\mathcal{E}$-invariant. To this end, we need to show that for all $t$, $c_0 \in \PC_t(\mathcal{S}) \,\Longleftrightarrow \, c_1 \in \PC_t(\mathcal{S})$. 
    We claim that for $t \leq \abs{\mathcal{S}_0}$, 
        \[
            c_0 \in \PC_t(\mathcal{S}) \, \Longleftrightarrow \, c_0 \in S_t \, \Longleftrightarrow \, c_1 \in S_t \, \Longleftrightarrow \, c_1 \in \PC_t(\mathcal{S}).
        \]
    The first equivalence follows since $\mathcal{S}_0$ is aligned to $(a, c)$ and hence $c_0 \not\in \FC_{t-1}(\mathcal{S})$ for any $t \leq \abs{\mathcal{S}_0}$. The second equivalence holds by the construction of $\mathcal{S}$. The final equivalence holds since no neighbor of $c_1$ is pre-cleared during the first $\abs{\mathcal{S}_0}$ moves. 
        
    Now suppose $t > \abs{\mathcal{S}_0}$. Then both $c_0$ and $c_1$ are pre-cleared. Indeed, $c_0$ is fully cleared since $\mathcal{S}_0$ is a successful search on $H_0$, while $c_1$ is pre-cleared since $\mathcal{S}_1$ is aligned to $(c, b)$.

    Since $\mathcal{S}$ is $\mathcal{E}$-invariant, we may apply Proposition~\ref{prop:quotient_search} to conclude that $\lor_{\mathcal{E}}(\mathcal{S})$ is successful and aligned to $(a, b)$, as desired.
\end{proof}

    \begin{prop}
        Let $(G_0, a, c)$ and $(G_1, a, c)$ be connected graphs with terminals such that $V(G_0) \cap V(G_1) = \{a\}$ and let $(G, a, b) = (G_0, a, b) \circ_b (G_1, a, c)$. Suppose that there is a subdivision $H_1$ of $G_1$ with a successful $3$-search $\mathcal{S}_1$ aligned to $(a, c)$, and every subdivision of $G_0$ has a further subdivision with a successful $3$-search aligned to $(a, b)$. Then there is a subdivision of $G$ admitting a successful $3$-search aligned to $(a, b)$.
    \label{prop:amalgamation_b}
    \end{prop}
    
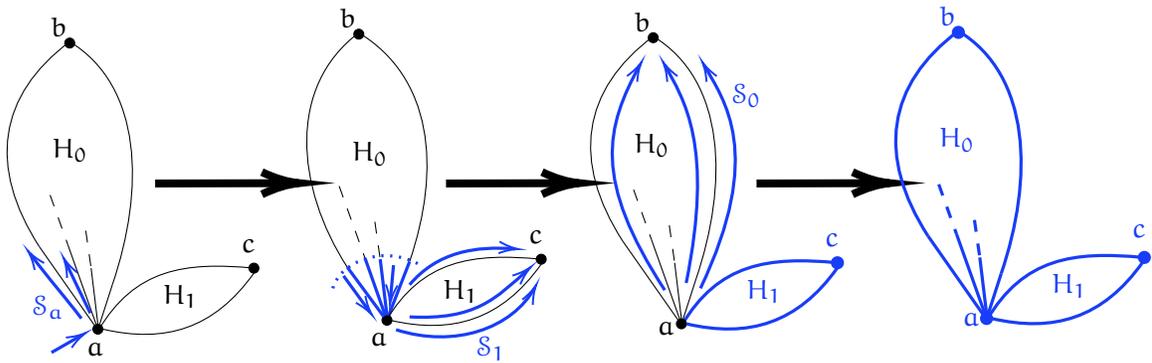
\begin{figure}[b]
\begin{center}
\begin{tikzpicture}[scale=0.9,x=0.75pt,y=0.75pt,yscale=-1,xscale=1]

\draw [color={rgb, 255:red, 24; green, 61; blue, 247 }  ,draw opacity=1, very thick]   (387.12,198.47) .. controls (407.58,163.8) and (435.58,159.8) .. (474.34,163.47) ;
\draw    (27.12,154.47) .. controls (-1.88,109.47) and (4.12,69.47) .. (44.12,39.47) ;
\draw    (68.12,175.47) .. controls (86.12,122.47) and (85.12,64.47) .. (44.12,39.47) ;
\draw  [fill={rgb, 255:red, 0; green, 0; blue, 0 }  ,fill opacity=1 ] (57.12,200.69) .. controls (57.12,199.15) and (58.36,197.91) .. (59.9,197.91) .. controls (61.43,197.91) and (62.68,199.15) .. (62.68,200.69) .. controls (62.68,202.22) and (61.43,203.47) .. (59.9,203.47) .. controls (58.36,203.47) and (57.12,202.22) .. (57.12,200.69) -- cycle ;
\draw  [fill={rgb, 255:red, 0; green, 0; blue, 0 }  ,fill opacity=1 ] (41.34,40.25) .. controls (41.34,38.71) and (42.58,37.47) .. (44.12,37.47) .. controls (45.65,37.47) and (46.9,38.71) .. (46.9,40.25) .. controls (46.9,41.78) and (45.65,43.02) .. (44.12,43.02) .. controls (42.58,43.02) and (41.34,41.78) .. (41.34,40.25) -- cycle ;
\draw  [fill={rgb, 255:red, 0; green, 0; blue, 0 }  ,fill opacity=1 ] (144.56,166.47) .. controls (144.56,164.93) and (145.8,163.69) .. (147.34,163.69) .. controls (148.87,163.69) and (150.12,164.93) .. (150.12,166.47) .. controls (150.12,168) and (148.87,169.25) .. (147.34,169.25) .. controls (145.8,169.25) and (144.56,168) .. (144.56,166.47) -- cycle ;
\draw    (62.68,200.69) .. controls (97.58,208.8) and (125.12,200.47) .. (147.34,169.25) ;
\draw    (60.12,201.47) .. controls (80.58,166.8) and (108.58,162.8) .. (147.34,166.47) ;
\draw    (27.12,154.47) -- (58.12,198.69) ;
\draw    (43.12,153.47) -- (58.9,197.91) ;
\draw    (68.12,175.47) -- (60.9,197.91) ;
\draw    (56.12,166.47) -- (59.9,197.91) ;
\draw  [dash pattern={on 4.5pt off 4.5pt}]  (33.12,125.47) -- (43.12,153.47) ;
\draw  [dash pattern={on 4.5pt off 4.5pt}]  (53.12,145.47) -- (56.12,166.47) ;
\draw [color={rgb, 255:red, 24; green, 61; blue, 247 }  ,draw opacity=1, very thick]   (50.67,194.73) .. controls (41.98,183.15) and (32.37,172.5) .. (21.82,158.29) ;
\draw [shift={(20.67,156.73)}, rotate = 413.75] [color={rgb, 255:red, 24; green, 61; blue, 247 }  ,draw opacity=1, very thick][line width=0.75]    (10.93,-3.29) .. controls (6.95,-1.4) and (3.31,-0.3) .. (0,0) .. controls (3.31,0.3) and (6.95,1.4) .. (10.93,3.29)   ;
\draw [color={rgb, 255:red, 24; green, 61; blue, 247 }  ,draw opacity=1, very thick]   (55.67,191.73) .. controls (50.03,180.45) and (46.15,171.82) .. (41.56,162.52) ;
\draw [shift={(40.67,160.73)}, rotate = 423.43] [color={rgb, 255:red, 24; green, 61; blue, 247 }  ,draw opacity=1, very thick][line width=0.75]    (10.93,-3.29) .. controls (6.95,-1.4) and (3.31,-0.3) .. (0,0) .. controls (3.31,0.3) and (6.95,1.4) .. (10.93,3.29)   ;
\draw [line width=3]    (92,119) -- (167.58,118.81) ;
\draw [shift={(172.58,118.8)}, rotate = 539.86] [color={rgb, 255:red, 0; green, 0; blue, 0 }  ][line width=3]    (20.77,-6.25) .. controls (13.2,-2.65) and (6.28,-0.57) .. (0,0) .. controls (6.28,0.57) and (13.2,2.66) .. (20.77,6.25)   ;
\draw    (199.58,165.8) .. controls (170.58,120.8) and (166.12,64.47) .. (206.12,34.47) ;
\draw    (234.58,162.8) .. controls (252.58,109.8) and (247.12,59.47) .. (206.12,34.47) ;
\draw  [fill={rgb, 255:red, 0; green, 0; blue, 0 }  ,fill opacity=1 ] (219.12,195.69) .. controls (219.12,194.15) and (220.36,192.91) .. (221.9,192.91) .. controls (223.43,192.91) and (224.68,194.15) .. (224.68,195.69) .. controls (224.68,197.22) and (223.43,198.47) .. (221.9,198.47) .. controls (220.36,198.47) and (219.12,197.22) .. (219.12,195.69) -- cycle ;
\draw  [fill={rgb, 255:red, 0; green, 0; blue, 0 }  ,fill opacity=1 ] (203.34,35.25) .. controls (203.34,33.71) and (204.58,32.47) .. (206.12,32.47) .. controls (207.65,32.47) and (208.9,33.71) .. (208.9,35.25) .. controls (208.9,36.78) and (207.65,38.02) .. (206.12,38.02) .. controls (204.58,38.02) and (203.34,36.78) .. (203.34,35.25) -- cycle ;
\draw [color={rgb, 255:red, 24; green, 61; blue, 247 }  ,draw opacity=1, very thick]   (199.58,165.8) -- (220.12,193.69) ;
\draw [color={rgb, 255:red, 24; green, 61; blue, 247 }  ,draw opacity=1, very thick]   (209.58,160.8) -- (220.9,192.91) ;
\draw [color={rgb, 255:red, 24; green, 61; blue, 247 }  ,draw opacity=1, very thick]   (234.58,162.8) -- (222.9,192.91) ;
\draw [color={rgb, 255:red, 24; green, 61; blue, 247 }  ,draw opacity=1, very thick]   (218.12,161.47) -- (221.9,192.91) ;
\draw  [dash pattern={on 4.5pt off 4.5pt}]  (195.12,120.47) -- (209.58,160.8) ;
\draw  [dash pattern={on 4.5pt off 4.5pt}]  (215.12,140.47) -- (218.12,161.47) ;
\draw [color={rgb, 255:red, 24; green, 61; blue, 247 }  ,draw opacity=1, very thick]   (227,201) .. controls (259.91,211.58) and (292.26,200.27) .. (304.84,177.23) ;
\draw [shift={(305.58,175.8)}, rotate = 476.57] [color={rgb, 255:red, 24; green, 61; blue, 247 }  ,draw opacity=1, very thick][line width=0.75]    (10.93,-3.29) .. controls (6.95,-1.4) and (3.31,-0.3) .. (0,0) .. controls (3.31,0.3) and (6.95,1.4) .. (10.93,3.29)   ;
\draw [color={rgb, 255:red, 24; green, 61; blue, 247 }  ,draw opacity=1, very thick]   (234.58,193.8) .. controls (271.63,195.75) and (283.02,183.44) .. (301.17,167.07) ;
\draw [shift={(302.58,165.8)}, rotate = 498.18] [color={rgb, 255:red, 24; green, 61; blue, 247 }  ,draw opacity=1, very thick][line width=0.75]    (10.93,-3.29) .. controls (6.95,-1.4) and (3.31,-0.3) .. (0,0) .. controls (3.31,0.3) and (6.95,1.4) .. (10.93,3.29)   ;
\draw [color={rgb, 255:red, 24; green, 61; blue, 247 }  ,draw opacity=1, very thick]   (234.58,175.8) .. controls (250.18,157.28) and (271.49,153.96) .. (293.86,155.66) ;
\draw [shift={(295.58,155.8)}, rotate = 184.97] [color={rgb, 255:red, 24; green, 61; blue, 247 }  ,draw opacity=1, very thick][line width=0.75]    (10.93,-3.29) .. controls (6.95,-1.4) and (3.31,-0.3) .. (0,0) .. controls (3.31,0.3) and (6.95,1.4) .. (10.93,3.29)   ;
\draw [color={rgb, 255:red, 24; green, 61; blue, 247 }  ,draw opacity=1, very thick]   (225.58,164.8) -- (223.76,184.81) ;
\draw [shift={(223.58,186.8)}, rotate = 275.19] [color={rgb, 255:red, 24; green, 61; blue, 247 }  ,draw opacity=1, very thick][line width=0.75]    (10.93,-3.29) .. controls (6.95,-1.4) and (3.31,-0.3) .. (0,0) .. controls (3.31,0.3) and (6.95,1.4) .. (10.93,3.29)   ;
\draw [color={rgb, 255:red, 24; green, 61; blue, 247 }  ,draw opacity=1, very thick]   (197.58,170.8) -- (212.41,191.18) ;
\draw [shift={(213.58,192.8)}, rotate = 233.97] [color={rgb, 255:red, 24; green, 61; blue, 247 }  ,draw opacity=1, very thick][line width=0.75]    (10.93,-3.29) .. controls (6.95,-1.4) and (3.31,-0.3) .. (0,0) .. controls (3.31,0.3) and (6.95,1.4) .. (10.93,3.29)   ;
\draw [color={rgb, 255:red, 24; green, 61; blue, 247 }  ,draw opacity=1, very thick] [dash pattern={on 0.84pt off 2.51pt}]  (191.58,177.8) .. controls (195.58,162.8) and (220.58,152.8) .. (240.58,164.8) ;
\draw  [fill={rgb, 255:red, 0; green, 0; blue, 0 }  ,fill opacity=1 ] (305.56,161.47) .. controls (305.56,159.93) and (306.8,158.69) .. (308.34,158.69) .. controls (309.87,158.69) and (311.12,159.93) .. (311.12,161.47) .. controls (311.12,163) and (309.87,164.25) .. (308.34,164.25) .. controls (306.8,164.25) and (305.56,163) .. (305.56,161.47) -- cycle ;
\draw    (223.68,195.69) .. controls (258.58,203.8) and (286.12,195.47) .. (308.34,164.25) ;
\draw    (221.12,196.47) .. controls (241.58,161.8) and (269.58,157.8) .. (308.34,161.47) ;
\draw    (354.12,151.47) .. controls (325.12,106.47) and (331.12,66.47) .. (371.12,36.47) ;
\draw [color={rgb, 255:red, 0; green, 0; blue, 0 }  ,draw opacity=1 ]   (395.12,172.47) .. controls (413.12,119.47) and (412.12,61.47) .. (371.12,36.47) ;
\draw  [fill={rgb, 255:red, 0; green, 0; blue, 0 }  ,fill opacity=1 ] (384.12,197.69) .. controls (384.12,196.15) and (385.36,194.91) .. (386.9,194.91) .. controls (388.43,194.91) and (389.68,196.15) .. (389.68,197.69) .. controls (389.68,199.22) and (388.43,200.47) .. (386.9,200.47) .. controls (385.36,200.47) and (384.12,199.22) .. (384.12,197.69) -- cycle ;
\draw  [fill={rgb, 255:red, 0; green, 0; blue, 0 }  ,fill opacity=1 ] (368.34,37.25) .. controls (368.34,35.71) and (369.58,34.47) .. (371.12,34.47) .. controls (372.65,34.47) and (373.9,35.71) .. (373.9,37.25) .. controls (373.9,38.78) and (372.65,40.02) .. (371.12,40.02) .. controls (369.58,40.02) and (368.34,38.78) .. (368.34,37.25) -- cycle ;
\draw  [color={rgb, 255:red, 24; green, 61; blue, 247 }  ,draw opacity=1, very thick][fill={rgb, 255:red, 24; green, 61; blue, 247 }  ,fill opacity=1 ] (471.56,163.47) .. controls (471.56,161.93) and (472.8,160.69) .. (474.34,160.69) .. controls (475.87,160.69) and (477.12,161.93) .. (477.12,163.47) .. controls (477.12,165) and (475.87,166.25) .. (474.34,166.25) .. controls (472.8,166.25) and (471.56,165) .. (471.56,163.47) -- cycle ;
\draw [color={rgb, 255:red, 24; green, 61; blue, 247 }  ,draw opacity=1, very thick]   (389.68,197.69) .. controls (424.58,205.8) and (452.12,197.47) .. (474.34,166.25) ;
\draw    (354.12,151.47) -- (385.12,195.69) ;
\draw    (370.12,150.47) -- (385.9,194.91) ;
\draw    (395.12,172.47) -- (387.9,194.91) ;
\draw    (383.12,163.47) -- (386.9,194.91) ;
\draw  [dash pattern={on 4.5pt off 4.5pt}]  (360.12,122.47) -- (370.12,150.47) ;
\draw  [dash pattern={on 4.5pt off 4.5pt}]  (380.12,142.47) -- (383.12,163.47) ;
\draw [line width=3]    (255,119) -- (324.58,118.81) ;
\draw [shift={(329.58,118.8)}, rotate = 539.85] [color={rgb, 255:red, 0; green, 0; blue, 0 }  ][line width=3]    (20.77,-6.25) .. controls (13.2,-2.65) and (6.28,-0.57) .. (0,0) .. controls (6.28,0.57) and (13.2,2.66) .. (20.77,6.25)   ;
\draw [line width=3]    (429,119) -- (498.58,118.81) ;
\draw [shift={(503.58,118.8)}, rotate = 539.85] [color={rgb, 255:red, 0; green, 0; blue, 0 }  ][line width=3]    (20.77,-6.25) .. controls (13.2,-2.65) and (6.28,-0.57) .. (0,0) .. controls (6.28,0.57) and (13.2,2.66) .. (20.77,6.25)   ;
\draw [color={rgb, 255:red, 24; green, 61; blue, 247 }  ,draw opacity=1, very thick]   (397.58,177.8) .. controls (433.66,101.75) and (410.79,72.28) .. (399.43,52.32) ;
\draw [shift={(398.58,50.8)}, rotate = 421.19] [color={rgb, 255:red, 24; green, 61; blue, 247 }  ,draw opacity=1, very thick][line width=0.75]    (10.93,-3.29) .. controls (6.95,-1.4) and (3.31,-0.3) .. (0,0) .. controls (3.31,0.3) and (6.95,1.4) .. (10.93,3.29)   ;
\draw [color={rgb, 255:red, 24; green, 61; blue, 247 }  ,draw opacity=1, very thick]   (389.58,174.8) .. controls (405.18,97.78) and (388.46,72.08) .. (377.42,52.31) ;
\draw [shift={(376.58,50.8)}, rotate = 421.19] [color={rgb, 255:red, 24; green, 61; blue, 247 }  ,draw opacity=1, very thick][line width=0.75]    (10.93,-3.29) .. controls (6.95,-1.4) and (3.31,-0.3) .. (0,0) .. controls (3.31,0.3) and (6.95,1.4) .. (10.93,3.29)   ;
\draw [color={rgb, 255:red, 24; green, 61; blue, 247 }  ,draw opacity=1, very thick]   (377.58,178.8) .. controls (337.2,116.75) and (344.35,87.68) .. (363.69,52.42) ;
\draw [shift={(364.58,50.8)}, rotate = 479.05] [color={rgb, 255:red, 24; green, 61; blue, 247 }  ,draw opacity=1, very thick][line width=0.75]    (10.93,-3.29) .. controls (6.95,-1.4) and (3.31,-0.3) .. (0,0) .. controls (3.31,0.3) and (6.95,1.4) .. (10.93,3.29)   ;
\draw [color={rgb, 255:red, 24; green, 61; blue, 247 }  ,draw opacity=1, very thick]   (558.12,195.47) .. controls (578.58,160.8) and (606.58,156.8) .. (645.34,160.47) ;
\draw [color={rgb, 255:red, 24; green, 61; blue, 247 }  ,draw opacity=1, very thick]   (525.12,148.47) .. controls (496.12,103.47) and (502.12,63.47) .. (542.12,33.47) ;
\draw [color={rgb, 255:red, 24; green, 61; blue, 247 }  ,draw opacity=1, very thick]   (566.12,169.47) .. controls (584.12,116.47) and (583.12,58.47) .. (542.12,33.47) ;
\draw  [color={rgb, 255:red, 24; green, 61; blue, 247 }  ,draw opacity=1, very thick][fill={rgb, 255:red, 24; green, 61; blue, 247 }  ,fill opacity=1 ] (555.12,194.69) .. controls (555.12,193.15) and (556.36,191.91) .. (557.9,191.91) .. controls (559.43,191.91) and (560.68,193.15) .. (560.68,194.69) .. controls (560.68,196.22) and (559.43,197.47) .. (557.9,197.47) .. controls (556.36,197.47) and (555.12,196.22) .. (555.12,194.69) -- cycle ;
\draw  [color={rgb, 255:red, 24; green, 61; blue, 247 }  ,draw opacity=1, very thick][fill={rgb, 255:red, 24; green, 61; blue, 247 }  ,fill opacity=1 ] (539.34,34.25) .. controls (539.34,32.71) and (540.58,31.47) .. (542.12,31.47) .. controls (543.65,31.47) and (544.9,32.71) .. (544.9,34.25) .. controls (544.9,35.78) and (543.65,37.02) .. (542.12,37.02) .. controls (540.58,37.02) and (539.34,35.78) .. (539.34,34.25) -- cycle ;
\draw  [color={rgb, 255:red, 24; green, 61; blue, 247 }  ,draw opacity=1, very thick][fill={rgb, 255:red, 24; green, 61; blue, 247 }  ,fill opacity=1 ] (643.56,160.47) .. controls (643.56,158.93) and (644.8,157.69) .. (646.34,157.69) .. controls (647.87,157.69) and (649.12,158.93) .. (649.12,160.47) .. controls (649.12,162) and (647.87,163.25) .. (646.34,163.25) .. controls (644.8,163.25) and (643.56,162) .. (643.56,160.47) -- cycle ;
\draw [color={rgb, 255:red, 24; green, 61; blue, 247 }  ,draw opacity=1, very thick]   (560.68,194.69) .. controls (595.58,202.8) and (623.12,194.47) .. (645.34,163.25) ;
\draw [color={rgb, 255:red, 24; green, 61; blue, 247 }  ,draw opacity=1, very thick]   (525.12,148.47) -- (556.12,192.69) ;
\draw [color={rgb, 255:red, 24; green, 61; blue, 247 }  ,draw opacity=1, very thick]   (541.12,147.47) -- (556.9,191.91) ;
\draw [color={rgb, 255:red, 24; green, 61; blue, 247 }  ,draw opacity=1, very thick]   (566.12,169.47) -- (558.9,191.91) ;
\draw [color={rgb, 255:red, 24; green, 61; blue, 247 }  ,draw opacity=1, very thick]   (554.12,160.47) -- (557.9,191.91) ;
\draw [color={rgb, 255:red, 24; green, 61; blue, 247 }  ,draw opacity=1, very thick] [dash pattern={on 4.5pt off 4.5pt}]  (531.12,119.47) -- (541.12,147.47) ;
\draw [color={rgb, 255:red, 24; green, 61; blue, 247 }  ,draw opacity=1, very thick] [dash pattern={on 4.5pt off 4.5pt}]  (551.12,139.47) -- (554.12,160.47) ;
\draw [color={rgb, 255:red, 24; green, 61; blue, 247 }  ,draw opacity=1, very thick]   (33.9,213.8) -- (51.36,203.79) ;
\draw [shift={(53.1,202.8)}, rotate = 510.19] [color={rgb, 255:red, 24; green, 61; blue, 247 }  ,draw opacity=1, very thick][line width=0.75]    (10.93,-3.29) .. controls (6.95,-1.4) and (3.31,-0.3) .. (0,0) .. controls (3.31,0.3) and (6.95,1.4) .. (10.93,3.29)   ;

\draw (33,92.4) node [anchor=north west][inner sep=0.75pt]    {$H_{0}$};
\draw (52.1,207.2) node [anchor=north west][inner sep=0.75pt]    {$a$};
\draw (32,24.4) node [anchor=north west][inner sep=0.75pt]    {$b$};
\draw (95,174.4) node [anchor=north west][inner sep=0.75pt]    {$H_{1}$};
\draw (139.34,148.87) node [anchor=north west][inner sep=0.75pt]    {$c$};
\draw (22,180.4) node [anchor=north west][inner sep=0.75pt]  [color={rgb, 255:red, 24; green, 61; blue, 247 }  ,opacity=1 ]  {$\mathcal{S}_{a}$};
\draw (201,94.4) node [anchor=north west][inner sep=0.75pt]    {$H_{0}$};
\draw (211,201.4) node [anchor=north west][inner sep=0.75pt]    {$a$};
\draw (194,19.4) node [anchor=north west][inner sep=0.75pt]    {$b$};
\draw (271,203.4) node [anchor=north west][inner sep=0.75pt]  [color={rgb, 255:red, 24; green, 61; blue, 247 }  ,opacity=1 ]  {$\mathcal{S}_{1}$};
\draw (252,171.4) node [anchor=north west][inner sep=0.75pt]    {$H_{1}$};
\draw (300.34,143.87) node [anchor=north west][inner sep=0.75pt]    {$c$};
\draw (359,89.4) node [anchor=north west][inner sep=0.75pt]    {$H_{0}$};
\draw (372,195.4) node [anchor=north west][inner sep=0.75pt]    {$a$};
\draw (359,21.4) node [anchor=north west][inner sep=0.75pt]    {$b$};
\draw (422,171.4) node [anchor=north west][inner sep=0.75pt]  [color={rgb, 255:red, 24; green, 61; blue, 247 }  ,opacity=1 ]  {$H_{1}$};
\draw (466.34,145.87) node [anchor=north west][inner sep=0.75pt]  [color={rgb, 255:red, 24; green, 61; blue, 247 }  ,opacity=1 ]  {$c$};
\draw (414,61.4) node [anchor=north west][inner sep=0.75pt]  [color={rgb, 255:red, 24; green, 61; blue, 247 }  ,opacity=1 ]  {$\mathcal{S}_{0}$};
\draw (530,86.4) node [anchor=north west][inner sep=0.75pt]  [color={rgb, 255:red, 24; green, 61; blue, 247 }  ,opacity=1 ]  {$H_{0}$};
\draw (543,190.4) node [anchor=north west][inner sep=0.75pt]  [color={rgb, 255:red, 24; green, 61; blue, 247 }  ,opacity=1 ]  {$a$};
\draw (530,18.4) node [anchor=north west][inner sep=0.75pt]  [color={rgb, 255:red, 24; green, 61; blue, 247 }  ,opacity=1 ]  {$b$};
\draw (593,168.4) node [anchor=north west][inner sep=0.75pt]  [color={rgb, 255:red, 24; green, 61; blue, 247 }  ,opacity=1 ]  {$H_{1}$};
\draw (638.34,140.87) node [anchor=north west][inner sep=0.75pt]  [color={rgb, 255:red, 24; green, 61; blue, 247 }  ,opacity=1 ]  {$c$};
\end{tikzpicture}
\end{center}
\caption{A pictorial representation of the construction in Proposition~\ref{prop:amalgamation_b}.}\label{fig:amalgam_b}
\end{figure}
\begin{proof}
    The proof is illustrated in Fig.~\ref{fig:amalgam_b}.
    Let $d$ be the degree of $a$ in $G_0$ and let $H_0$ be a subdivision of $G_0$ such that every edge incident to $a$ is subdivided at least $2^{d-1}\abs{\mathcal{S}_1} + 1$ times and there is a successful $3$-search $\mathcal{S}_0$ on $H_0$ aligned to $(a, b)$. By Lemma~\ref{lemma:clearing_from_bottom}, there exists a $3$-search $\mathcal{S}_a$ on $H_0$ aligned to $(a, b)$ with $a \in (S_a)_t$ for all $t$ such that $B_{H_0}(a, \abs{\mathcal{S}_1}) \subseteq \FC_*(\mathcal{S}_a)$. The desired subdivision of $G$ is the union of $H_0$ and $H_1$, and the desired $3$-search is the concatenation $\mathcal{S}_a \concat \mathcal{S}_1 \concat \mathcal{S}_0$.

    Let $H$ be the {disjoint} union of $H_0$ and $H_1$ and denote the copy of $a$ in each $H_i \subset H$ by $a_i$. 
    Let $\mathcal{E}$ be the equivalence relation on $V(H)$ with $\{a_0, a_1\}$ as the only nontrivial equivalence class. Then $H/\mathcal{E}$ can be naturally identified with the \ep{non-disjoint} union of $H_0$ and $H_1$. 

    Construct searches $\mathcal{S}_a'$, $\mathcal{S}_0'$, and $\mathcal{S}_1'$ on $H$ from $\mathcal{S}_a$, $\mathcal{S}_0$, and $\mathcal{S}_1$ respectively 
    by replacing every appearance of $a$ with both $a_0$ and $a_1$. Let $\mathcal{S} = \mathcal{S}_a'\concat\mathcal{S}_1'\concat\mathcal{S}_0'$. For brevity, let $t_0 = \abs{\mathcal{S}_a}$ and $t_1 = t_0 + \abs{\mathcal{S}_1}$.

    We now check that $\mathcal{S}$ is $\mathcal{E}$-invariant. To this end, we need to show that for all $t$, $a_0 \in \PC_t(\mathcal{S}) \,\Longleftrightarrow \, a_1 \in \PC_t(\mathcal{S})$. By construction, $a_0 \in S_t$ for $t \le t_0$. Since $B_H(a_0, \abs{\mathcal{S}_1}) \subseteq \FC_{t_0}(\mathcal{S})$, we have $a_0 \in \PC_t(\mathcal{S})$ for $t_0 < t \le t_1$. Since $\mathcal{S}_0$ is aligned to $(a, b)$, $a_0 \in \PC_t(\mathcal{S})$ for $t \ge t_1$. Thus, $a_0 \in \PC_t(\mathcal{S})$ for every $t$.

    For $t \le t_0$, we have $a_1 \in S_t$ by construction. Since $\mathcal{S}_1$ is aligned to $(a, c)$, $a_1 \in \PC_t(\mathcal{S})$ for $t_0 < t \le t_1$. Since $\mathcal{S}_1$ is successful, $a_1 \in \FC_t(\mathcal{S})$ for $t > t_1$. Thus $a_1 \in \PC_t(\mathcal{S})$ for every $t$ too.

    We have thus shown that $\mathcal{S}$ is $\mathcal{E}$-invariant. Note that $\lor_{\mathcal{E}}(\mathcal{S})$ is a $3$-search on $H / \mathcal{E}$ \ep{if we identify $H/\mathcal{E}$ with the union of $H_0$ and $H_1$, then $\lor_{\mathcal{E}}(\mathcal{S}) = \mathcal{S}_a \concat \mathcal{S}_1 \concat \mathcal{S}_0$}, and furthermore, by Proposition~\ref{prop:quotient_search}, $\lor_{\mathcal{E}}(\mathcal{S})$ is successful and aligned to $(a, b)$, as desired.
\end{proof}

    \begin{prop}
        Let $(G_0, a, b)$ and $(G_1, b, c)$ be connected graphs with terminals such that $V(G_0) \cap V(G_1) = \{b\}$ and let $(G, a, b) = (G_0, a, b) \circ_{b'} (G_1, b, c)$. Suppose that there is a subdivision $H_1$ of $G_1$ with a successful $3$-search $\mathcal{S}_1$ aligned to $(c, b)$, and every subdivision of $G_0$ has a further subdivision with a successful $3$-search aligned to $(a, b)$. Then there is a subdivision of $G$ admitting a successful $3$-search aligned to $(a, b)$.
    \label{prop:amalgamation_b2}
    \end{prop}
\begin{figure}[b]
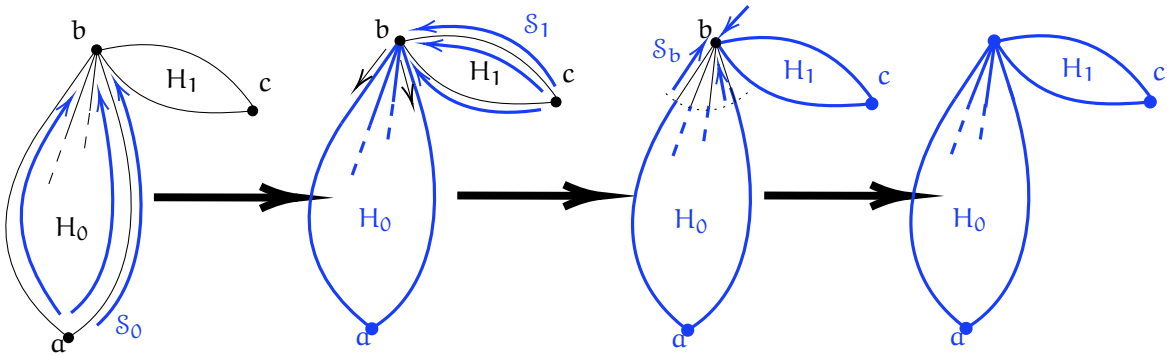

\begin{center}

\end{center}
\caption{A pictorial representation of the construction in Proposition~\ref{prop:amalgamation_b2}.}\label{fig:amalgam_b2}
\end{figure}
\begin{proof}
    The proof is illustrated in Fig.~\ref{fig:amalgam_b2}. Let $d$ be the degree of $b$ in $G_0$ and let $H_0$ be a subdivision of $G_0$ such that every edge incident to $b$ is subdivided at least $2^{d-1}(\abs{\mathcal{S}_1} + 1)$ times and there is a successful $3$-search $\mathcal{S}_0$ on $H_0$ aligned to $(a, b)$. For our analysis, it will be convenient to modify $\mathcal{S}_0$ so that the vertex $b$ remains uncleared even after the last step. To this end, we let $\mathcal{S}_0^\ast$ be the search on $H_0$ obtained from $\mathcal{S}_0$ by removing $b$ from the set $(S_0)_{\abs{\mathcal{S}_0}}$ \ep{i.e., from the last searched set in $\mathcal{S}_0$}.
    
    Let $A = V(H_0) \setminus B_{H_0}(b, \abs{\mathcal{S}_1} + 1)$. By Lemma~\ref{lemma:clearing_to_top}, there exists a $3$-search $\mathcal{S}_b$ on $H_0$ aligned to $(a, b)$ over $A$ with $b \in (S_b)_t$ for all $t$ such that $\FC_*(\mathcal{S}_b, A) = V(H_0)$. The desired subdivision of $G$ is the union of $H_0$ and $H_1$, and the desired $3$-search is the concatenation $\mathcal{S}_0^\ast \concat \mathcal{S}_1 \concat \mathcal{S}_b$. 
    
    Let $H$ be the {disjoint} union of $H_0$ and $H_1$ and denote the copy of $b$ in each $H_i \subset H$ by $b_i$. Let $\mathcal{E}$ be the equivalence relation on $V(H)$ with $\{b_0, b_1\}$ as the only nontrivial equivalence class. Then $H/\mathcal{E}$ can be naturally identified with the \ep{non-disjoint} union of $H_0$ and $H_1$.

    We construct searches $\mathcal{S}_0'$, $\mathcal{S}_1'$, and $\mathcal{S}_b'$ on $H$ from $\mathcal{S}_0^\ast$, $\mathcal{S}_1$, and $\mathcal{S}_b$ respectively by replacing every appearance of $b$ with both $b_0$ and $b_1$. Let $\mathcal{S} = \mathcal{S}_0' \concat \mathcal{S}_1' \concat \mathcal{S}_b'$. For brevity, let $t_0 = \abs{\mathcal{S}_0}$, and $t_1 = t_0 + \abs{\mathcal{S}_1}$.

    We now check that $\mathcal{S}$ is $\mathcal{E}$-invariant. To this end, we claim that for all $t$,
        \[
            b_0 \in \PC_t(\mathcal{S}) \, \Longleftrightarrow \, b_0 \in S_t \, \Longleftrightarrow \, b_1 \in S_t \, \Longleftrightarrow \, b_1 \in \PC_t(\mathcal{S}).
        \]
    The second equivalence holds by the construction of $\mathcal{S}$, so we only need to establish the first and third equivalences. We start by considering the case $t \leq t_0$. The first equivalence then follows since $\mathcal{S}_0$ is aligned to $(a, b)$ and hence $b_0 \not\in \FC_{t-1}(\mathcal{S})$, while the third equivalence holds since no neighbor of $b_1$ is pre-cleared during the first $t_0$ moves.
    

    Next we consider the case $t_0 < t \leq t_1$. Since $\mathcal{S}_0$ is aligned to $(a, b)$ and, by construction, $b \not \in (S^\ast_0)_{t_0}$, we have $b \not \in \PC_{t_0}(\mathcal{S})$. Since $\PC_*(\mathcal{S}_0) = V(H_0)$, we have $\PC_{t_0}(\mathcal{S}) = V(H_0) \setminus \{b_0\}$ and hence $\FC_{t_0}(\mathcal{S}) = V(H_0) \setminus B_{H_0}(b_0, 1)$. Therefore, if $t_0 < t \leq t_1$, then no neighbor of $b_0$ is in $\PC_{t}(\mathcal{S})$. Thus $b_0 \not\in\FC_t(\mathcal{S})$ for such $t$, which yields $b_0 \in \PC_t(\mathcal{S}) \Longleftrightarrow b_0 \in S_t$. Since $\mathcal{S}_1$ is aligned to $(c, b)$, $b_1 \not \in \FC_t(\mathcal{S})$ for $t < t_1$, so $b_1 \in \PC_{t}(\mathcal{S}) \Longleftrightarrow b_1 \in S_t$ as well. 

    Finally, if $t > t_1$, then $b_0$, $b_1 \in (S_b)_{t} \subseteq \PC_t(\mathcal{S})$ always. We have thus shown that the search $\mathcal{S}$ is $\mathcal{E}$-invariant. Note that $\lor_{\mathcal{E}}(\mathcal{S})$ is a $3$-search on $H / \mathcal{E}$, and furthermore, by Proposition~\ref{prop:quotient_search}, it is successful and aligned to $(a, b)$, proving the claim.
\end{proof}

\begin{prop} Let $(G_0, a, b)$, $(G_1, a, c)$, and $(G_2, d, b)$ be connected graphs with terminals such that $V(G_0)\cap V(G_1) = \{a\}$, $V(G_0) \cap V(G_2) = \{b\}$, and $V(G_1) \cap V(G_2) = \0$. Let
    \[(G, a, b) = (G_0, a, b) \circ_{p} ((G_1, a, c) \circ_s (e_{cd}, c, d) \circ_s (G_2, d, b)),\]
where $e_{cd}$ is the graph comprising a single edge between $c$ and $d$. Suppose that for each $i \in \{1,2\}$, there exists a subdivision $H_i$ of $G_i$ with a successful $3$-search $\mathcal{S}_i$ aligned to $(a,c)$ or $(d,b)$ respectively, and every subdivision of $G_0$ has a further subdivision with a successful $3$-search aligned to $(a, b)$. Then there is a subdivision of $G$ admitting a successful $3$-search aligned to $(a, b)$.
\label{prop:amalgamation_p}
\end{prop}
\begin{figure}[t]
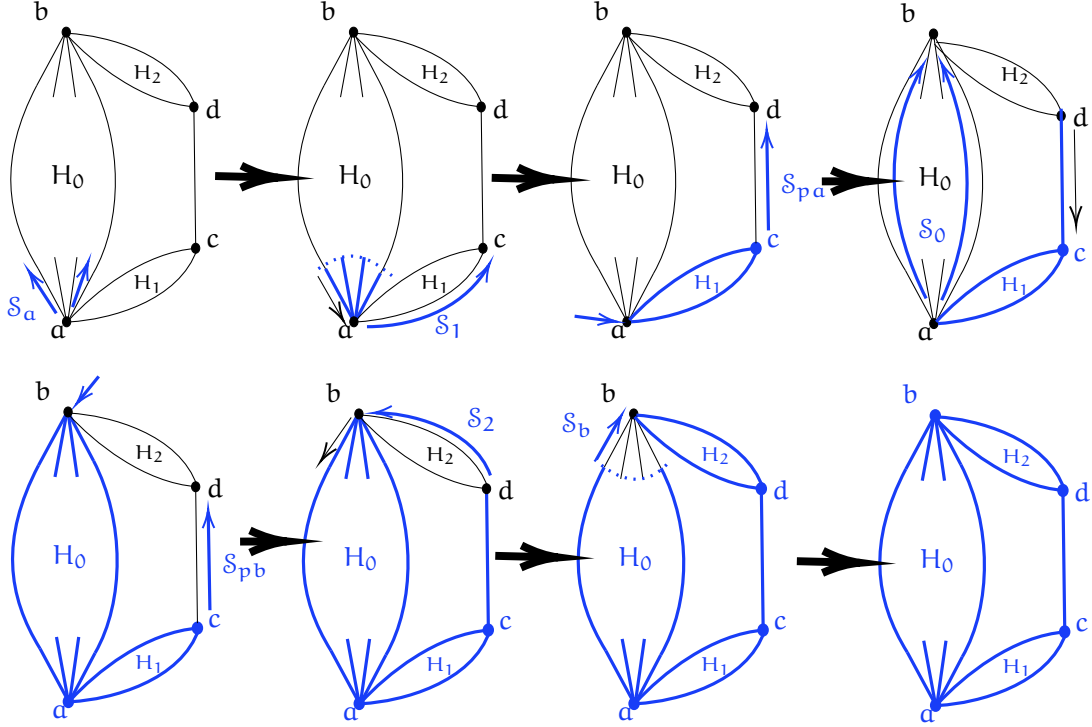

\begin{center}

\end{center}
\caption{A pictorial representation of the construction in Proposition~\ref{prop:amalgamation_p}.}\label{fig:amalgam_p}
\end{figure}
\begin{proof}
    The proof is illustrated in Fig.~\ref{fig:amalgam_p}. Let $\Delta$ be the maximum of the degrees of $a$ and $b$ in $G_0$ and let $H_0$ be a subdivision of $G_0$ such that every edge incident to $a$ is subdivided at least $2^{\Delta - 1}\abs{\mathcal{S}_1} + 1$ times, every edge incident to $b$ is subdivided at least $2^{\Delta - 1}(\abs{\mathcal{S}_2} + 1)$ times, and there is a successful $3$-search $\mathcal{S}_0$ on $H_0$ aligned to $(a, b)$. By Lemma~\ref{lemma:clearing_from_bottom}, there exists a $3$-search $\mathcal{S}_a$ on $H_0$ aligned to $(a, b)$ with $a \in (S_{a})_t$ for each $t$ such that $B_{H_0}(a, \abs{\mathcal{S}_1}) \subseteq \FC_*(\mathcal{S}_a)$. Also, by Lemma~\ref{lemma:clearing_to_top}, there is a $3$-search $\mathcal{S}_{b}$ on $H_0$ aligned to $(a, b)$ with $b \in (S_{b})_t$ for each $t$ such that $\FC_*(\mathcal{S}_b, V(H_0) \setminus B_{H_0}(b, \abs{\mathcal{S}_2} + 1)) = V(H_0)$.

    Let $H_3$ be a path from $c$ to $d$ of length $\abs{\mathcal{S}_0} + 5$ that is disjoint from $H_0 \cup H_1 \cup H_2$ except for the vertices $c$ and $d$. The desired subdivision of $G$ is the union $H_0 \cup H_1 \cup H_2 \cup H_3$.
    
    Let $\mathcal{S}_{p a}$ be a $3$-search on $H_0 \cup H_1 \cup H_2 \cup H_3$ of length $\abs{\mathcal{S}_0} + 2$, where $(S_{p a})_t$ comprises the vertices of distance $t-1$ and $t$ from $c$ in $H_3$, as well as $a$. Similarly, let $\mathcal{S}_{p b}$ be a search of length $\abs{\mathcal{S}_0} + 2$, where $(S_{p b})_t$ comprises the vertices of distance $t+2$ and $t+3$ from $c$ in $H_3$, as well as $b$. \ep{In particular, the last set in $\mathcal{S}_{pb}$ contains $d$.} The desired successful $3$-search on $H_0 \cup H_1 \cup H_2 \cup H_3$ is the concatenation
    \[
        \mathcal{S}_a \concat \mathcal{S}_1 \concat \mathcal{S}_{pa} \concat \mathcal{S}_0 \concat \mathcal{S}_{pb} \concat (\{d\}) \concat \mathcal{S}_2 \concat \mathcal{S}_b.
    \]
    Notice that in this search, after performing $\mathcal{S}_{pb}$ we just check the one-element set $\{d\}$. This is not strictly speaking necessary but makes the analysis a bit simpler.

    Let $H$ be the disjoint union of $H_0$, $H_1$, $H_2$, and $H_3$ and denote the copy of each vertex $x \in \{b, c, d\}$ in $H_i \subset H$ by $x_i$. Let $\mathcal{E}$ be the equivalence relation on $V(H)$ relating $a_0$ with $a_1$, $b_0$ with $b_2$, $c_1$ with $c_3$, and $d_2$ with $d_3$. Then the quotient graph $H/\mathcal{E}$ can be naturally identified with the \ep{non-disjoint} union of $H_0$, $H_1$, $H_2$, and $H_3$. 

    We construct searches $\mathcal{S}_{a}'$, $\mathcal{S}_{b}'$, $\mathcal{S}_0'$, $\mathcal{S}_1'$, $\mathcal{S}_2'$, $\mathcal{S}_{p a}'$, and $\mathcal{S}_{p b}'$ on $H$ from $\mathcal{S}_{a}$, $\mathcal{S}_{b}$, $\mathcal{S}_0$, $\mathcal{S}_1$, $\mathcal{S}_2$, $\mathcal{S}_{p a}$, and $\mathcal{S}_{p b}$ respectively by replacing every appearance of $b$, $c$, or $d$ with both of their copies in $H$. Define 
    \[
        \mathcal{S} = \mathcal{S}'_{a}\concat\mathcal{S}_1'\concat\mathcal{S}_{pa}' \concat \mathcal{S}_0' \concat \mathcal{S}_{pb}',
    \]
    and $\mathcal{R} = (\{d_2, d_3\}) \concat \mathcal{S}_2' \concat \mathcal{S}'_{b}$. For brevity, define
    \[
        t_0 = \abs{\mathcal{S}_{a}}, \quad t_1 = t_0 + \abs{\mathcal{S}_1}, \quad t_2 = t_1 + \abs{\mathcal{S}_{pa}}, \quad t_3 = t_2 + \abs{\mathcal{S}_0}, \quad t_4 = t_3 + \abs{\mathcal{S}_{pb}}, \quad  t_5 = 1 + \abs{\mathcal{S}_2}.
    \]

    We now check that $\mathcal{S}$ is $\mathcal{E}$-invariant. For $t \le t_0$, since $a \in (S_{a})_t$, both $a_0$ and $a_1$ belong to $S_t$, and thus $a_0$, $a_1 \in \PC_t(\mathcal{S})$. It is also clear that none of $b_0$, $b_2$, $c_1$, $c_3$, $d_2$, $d_3$ are in $\PC_t(\mathcal{S})$ for $t \leq t_0$.

    For $t_0 < t \le t_1$, since $B_{H}(a_0, \abs{\mathcal{S}_1}) \subseteq \FC_{t_0}(\mathcal{S})$, we have $N_{H}(a_0) \cup \{a_0\} \subseteq \FC_{t - 1}(\mathcal{S})$, and so $a_0 \in \PC_t(\mathcal{S})$. Since $\mathcal{S}_1$ is aligned to $(a, c)$, we also have $a_1 \in \PC_t(\mathcal{S})$ and $c \not \in \FC_{t - 1}(\mathcal{S})$ so 
    \[c_1 \in \PC_{t}(\mathcal{S}) \Longleftrightarrow c_1 \in S_t \Longleftrightarrow c_3 \in S_t \Longleftrightarrow c_3 \in \PC_t(\mathcal{S}).\]
    It is also clear that none of $b_0$, $b_2$, $d_2$, $d_3$ are in $\PC_t(\mathcal{S})$ for $t_0 < t \leq t_1$.

    Since $\mathcal{S}_1$ is successful, $V(H_1) \subseteq \FC_{t_1}(\mathcal{S})$, and hence $\{a_1, c_1\} \subseteq \PC_{t}(\mathcal{S})$ for all $t \geq t_1$. Furthermore, if $t_1 < t \leq t_2$, then $a_0 \in \PC_t(\mathcal{S})$ since $\{a_0, a_1\} \subseteq S_t$. Note also that if $t_1 < t \leq t_2$, then $\PC_{t}(\mathcal{S}) \cap V(H_3) = B_{H}(c_3, t-t_1)$, so $c_3 \in \PC_t(\mathcal{S})$. Finally, none of $b_0$, $b_2$, $d_2$, $d_3$ are in $\PC_t(\mathcal{S})$ for $t_1 < t \leq t_2$. 

    For $t_2 < t \le t_3$, we have $\{a_1, c_1\} \subseteq \PC_t(\mathcal{S})$. Since $\mathcal{S}_0$ is aligned to $(a, c)$, $a_0 \in \PC_t(\mathcal{S})$ and
    \[b_0 \in \PC_{t}(\mathcal{S}) \Longleftrightarrow b_0 \in S_t \Longleftrightarrow b_2 \in S_t \Longleftrightarrow b_2 \in \PC_t(\mathcal{S}).\]
    Since $B_{H}(c_3, \abs{\mathcal{S}_0} + 1) \subseteq \FC_{t_2}(\mathcal{S})$, we have $c_3 \in \PC_t(\mathcal{S})$. Finally, neither $d_2$ nor $d_3$ are in $\PC_t(\mathcal{S})$.

    For $t_3 < t \le t_4$, we have $\{a_1, c_1\} \subseteq \PC_t(\mathcal{S})$. Since $\mathcal{S}_0$ is successful, $V(H_0) \subseteq \PC_{t_3}(\mathcal{S})$, and hence $\{a_0, b_0\} \subseteq \PC_t(\mathcal{S})$. We note that $\FC_{t_2}(\mathcal{S})$ contains $B_H(c_3, \abs{\mathcal{S}_0} + 1)$ by construction, so $\FC_{t_3}(\mathcal{S})$ contains $B_H(c_3, 1)$. We then see that $\FC_t(\mathcal{S})$ includes $B_H(c_3, 1+t- t_3)$, and in particular it contains $c_3$. Both $b_0$ and $b_2$ are in $\PC_t(\mathcal{S})$ since $\{b_0, b_1\} \subseteq S_t$. Finally, it is clear that
    \[
        d_2 \in \PC_{t}(\mathcal{S}) \Longleftrightarrow d_3 \in \PC_t(\mathcal{S}) \Longleftrightarrow t = t_4.
    \]

    We have thus shown the $\mathcal{E}$-invariance of $\mathcal{S}$. It is straightforward to check that $\FC_*(\mathcal{S}) = V(H_0) \sqcup V(H_1) \sqcup V(H_3)$, so, by Proposition~\ref{prop:quotient_search}, $\lor_\mathcal{E}(\mathcal{S})$ is a $3$-search on $H/\mathcal{E}$ with
    \[
        \FC_*(\lor_\mathcal{E}(\mathcal{S})) = V(H_0) \cup V(H_1) \cup V(H_3) \setminus \{b, d\}
    \]
    \ep{here the union is non-disjoint}. Let
    \[
        A = V(H_0) \sqcup V(H_1) \sqcup V(H_3) \setminus \{b_0, d_3\}.
    \]
    Then $A$ is an $\mathcal{E}$-invariant subset of $V(H)$ and $A / \mathcal{E} = \FC_*(\lor_\mathcal{E}(\mathcal{S}))$. We now wish to show that the search $\mathcal{R}$ is $\mathcal{E}$-invariant over $A$.

    Note that $\{a_0, a_1, c_1, c_3, d_2, d_3\} \subseteq \PC_1(\mathcal{R}, A)$ and $\{b_0, b_2\} \cap \PC_1(\mathcal{R},A) = \0$. Furthermore, $\FC_1(\mathcal{R}, A) = V(H_0) \sqcup V(H_1) \sqcup V(H_3) \setminus B_H(b_0, 1)$.
    
    For $1 < t \le t_5$, $\{a_0, a_1, c_1, c_3, d_3\} \subseteq \FC_t(\mathcal{R}, A)$. Since $\mathcal{S}_2$ is aligned to $(d, b)$, $d_2 \in \PC_t(\mathcal{R}, A)$ and
    \[b_0 \in \PC_{t}(\mathcal{R}, A) \Longleftrightarrow b_0 \in R_t \Longleftrightarrow b_2 \in R_t \Longleftrightarrow b_2 \in \PC_t(\mathcal{R}, A).\]
    Since $\mathcal{S}_2$ is successful, $V(H_2) \subseteq \FC_{t_5}(\mathcal{R}, A)$.

    For $t > t_5$, we have $\{a_1, c_1, b_2, d_2, c_3, d_3\} \subseteq \FC_t(\mathcal{R}, A)$. Since $\mathcal{S}_{b}$ is aligned to $(a, b)$, we have $a_0 \in \PC_t(\mathcal{R}, A)$, and by construction $b_0 \in R_t \subseteq \PC_t(\mathcal{R}, A)$.
    
    We have thus shown that $\mathcal{R}$ is $\mathcal{E}$-invariant over $A$. Hence, $\lor_\mathcal{E}(\mathcal{R})$ is a successful $3$-search with $A / \mathcal{E}$ as its initial cleared set. It is also clear that $\lor_\mathcal{E}(\mathcal{R})$ is aligned to $(a, b)$, so $\lor_\mathcal{E}(\mathcal{S}) \concat \lor_\mathcal{E}(\mathcal{R})$ is a successful $3$-search on $H / \mathcal{E}$ aligned to $(a, b)$ as desired.
\end{proof}

    \subsection{Searching graphs with simple GSP decompositions}
    
    We are now ready to complete the proof of the implication \ref{item:GSP} $\Longrightarrow$ \ref{item:ts3} in Theorem~\ref{theo:main}.

    \begin{theo} Suppose $G$ is a connected graph with a simple GSP decomposition. Then $\s_t(G) \le 3$. \label{theo:searching}
    \end{theo}
\begin{proof}\stepcounter{ForClaims} \renewcommand{\theForClaims}{\ref{theo:searching}}
We shall show by induction on $\abs{V(G)}$ that if $\gwt{G} = (G, a, b)$ admits a simple GSP decomposition $T$, then every subdivision of $G$ has a further subdivision which has a successful $3$-search aligned to $(a, b)$. The base case, where $G$ is a single edge $ab$, is clear.

    Now assume the GSP decomposition $T$ is nontrivial. If $\circ_{T, \gwt{G}} \in \{\circ_s, \circ_b\}$, then we are done by Propositions~\ref{prop:amalgamation_s} and \ref{prop:amalgamation_b} respectively and the inductive hypothesis. If $\circ_{T, \gwt{G}} = \circ_{b'}$, i.e., $\gwt{G} = \gwt{G_0} \circ_{b'} \gwt{G_1}$, then we can use Lemma~\ref{lemma:inversion} to switch the order of the terminals in $\gwt{G_1}$ and then apply Proposition~\ref{prop:amalgamation_b2}  and the inductive hypothesis. 
    
    It remains to consider the case when $\circ_{T, \gwt{G}} = \circ_p$, i.e., $(G, a,b) = (G_0, a, b) \circ_p (K, a, b)$. Since $T$ is simple, we may assume, without loss of generality, that $K$ has a bridge $cd$ separating $a$ and $b$, with $c$ closer to $a$ and $d$ closer to $b$. To apply Proposition~\ref{prop:amalgamation_p}, we need to decompose $(K, a, b)$ as $(G_1, a, c) \circ_s (e_{cd}, c, d) \circ_s (G_2, d, b)$, where $e_{cd}$ is the graph comprising the single edge $cd$.
    

    \begin{smallclaim} Suppose $(K, a, b)$ contains a bridge $cd$ separating its terminals \ep{with $c$ closer to $a$ and $d$ closer to $b$} and admits a simple GSP decomposition $T'$. Let $G_0$ be the subgraph of $K$ induced by the vertices that are separated from $b$ by $d$, and let $G_1$ be the subgraph of $K$ induced by the vertices that are separated from $a$ by $c$. Then $(G_0, a, c)$ and $(G_1, d, b)$ each either comprise a single vertex or admit simple GSP decompositions.
    \end{smallclaim}
    \begin{claimproof}
    We shall induct on $\abs{V(K)}$. The base case where $K$ comprises a single edge is trivial. Now suppose the decomposition $T'$ is nontrivial. Note that $\circ_{T', \gwt{K}} \neq \circ_p$, since otherwise $K$ would contain two internally disjoint $ab$-paths. If $\circ_{T', \gwt{K}} = \circ_b$, then $(K, a, b) = (K_0, a, b) \circ_b (K_1, a, x)$ for some vertex $x$. Then $cd$ is a bridge separating the terminals of $(K_0, a, b)$, so, by the inductive hypothesis, the subgraphs $H_0$ and $H_1$ of $K_0$ comprising the vertices separated from each terminal by $d$, $c$ respectively have simple GSP decompositions. But $G_1 = H_1$ and $(G_0, a, c) = (H_0, a, c) \circ_b (K_1, a, x)$. This yields simple GSP decompositions of $G_0$ and $G_1$, as desired. The case where $\circ_{T', \gwt{K}} = \circ_{b'}$ follows similarly, \emph{mutatis mutandis}.

    Finally, suppose $\circ_{T', \gwt{K}} = \circ_s$, i.e., $(K, a, b) = (K_0, a, x) \circ_s (K_1, x, b)$ for some vertex $x$. Since $c$ and $d$ are adjacent, they must both be contained in either $K_0$ or $K_1$; without loss of generality, say $c$, $d \in V(K_0)$. By the inductive hypothesis, the subgraphs $H_0$ and $H_1$ of $K_0$ comprising the vertices separated from each terminal by $d$, $c$ respectively have simple GSP decompositions. But $G_0 = H_0$, and $(G_1, d, b) = (H_1, d, x) \circ_s (K_1, x, b)$, which yields simple GSP decompositions of $G_0$ and $G_1$, as desired.
    \end{claimproof}

    If necessary, we may replace $K$ with a subdivision and assume that $\{a,b\} \cap \{c,d\} = \0$. By the above claim, we may then write $(K, a, b) = (G_0, a, c) \circ_s (e_{cd}, c ,d) \circ_s (G_1, d, b)$, where $(G_0, a, c)$ and $(G_1, d, b)$ admit simple GSP decompositions. We can then apply Proposition~\ref{prop:amalgamation_p} to finish the proof.
\end{proof}

\section{Further investigations}\label{sec:conclusion}
    
Our main result, Theorem~\ref{theo:classification}, provides a complete characterization of graphs with topological inspection number at most $3$. A natural next step would be to consider graphs with topological inspection number at most $k$ for larger values of $k$. In this case, it is perhaps too optimistic to seek a complete characterization; however, even many rather fundamental specific questions remain open. For instance, we do not know the values of $\s_t$ for such basic families of graphs as complete graphs or complete bipartite graphs:

\begin{prob}
    Determine $\s_t(K_n)$ and $\s_t(K_{s,t})$.
\end{prob}

Since $\s_t(K_2) = 2$, $\s_t(K_3) = 3$, and \ep{by Theorem~\ref{res:1}} $\s_t(K_4) = 4$, a natural guess would be that $\s_t(K_n) = n$ for all $n$. This, however, turns out not to be the case, as the following bound shows:
    
\begin{prop} $\s_t(K_n) \le \lceil\frac{n}{3}\rceil + 2$.
\end{prop}

\begin{proof}[Proof sketch] Given any subdivision $G$ of $K_n$, partition the vertices of $K_n$ in $G$ into three parts $A$, $B$, $C$, each of size at most $\lceil\frac{n}{3}\rceil$. Let $\mathcal{E}$ be the equivalence relation on $V(G)$ where $A$, $B$, and $C$ are the nontrivial equivalence classes \ep{that is, every vertex of $G$ that is not one of the original vertices of $K_n$ forms its own equivalence class}. Now consider $G / \mathcal{E}$. This graph comprises three vertices $a$, $b$, $c$, which are the images of $A$, $B$, $C$ respectively under the quotient map, along with several internally vertex-disjoint paths between each of them, as well as some cycles which only intersect the rest of the graph at one of $a$, $b$, or $c$. Let $P_{ab}$, $P_{bc}$, $P_{ac}$ be the subgraphs of $G / \mathcal{E}$ comprising the paths between their respective subscripted vertices, and let $C_a$, $C_b$, $C_c$ be the subgraphs comprising the cycles containing their respective subscripted vertices.

We will construct a successful $3$-search on some subdivision of $G / \mathcal{E}$ in which at most one of $a$, $b$, or $c$ is contained in each search set. Note that 
\[(G / \mathcal{E}, a, c) = (P_{ac}, a, c) \circ_p ((P_{ab}, a, b) \circ_s ((P_{bc}, b, c) \circ_b (C_b, b, y)) \circ_b (C_a, a, x) \circ_{b'} (C_c, c, z),\]
where $x$, $y$, and $z$ are arbitrary vertices of $C_a$, $C_b$, and $C_c$ respectively distinct from $a$, $b$, and $c$. 
By Theorem~\ref{theo:searching}, every subdivision of each of the graphs $(P_{ac}, a, c)$, $(P_{ab}, a, b)$, $(P_{bc}, b, c)$, $(C_b, b, y)$, $(C_a, a, x)$, $(C_c, c, z)$ admits a further subdivision that has a successful $3$-search aligned to the corresponding terminals; furthermore, it is easy to verify that if the graphs are sufficiently subdivided, then the searches constructed in the proof of Theorem~\ref{theo:searching} include at most one terminal vertex in each search set. 
Combining these searches via Propositions~\ref{prop:amalgamation_s}, \ref{prop:amalgamation_b}, \ref{prop:amalgamation_b2}, and \ref{prop:amalgamation_p} gives a successful $3$-search $\mathcal{S}$ on some subdivision of $G/\mathcal{E}$ in which at most one of $a$, $b$, or $c$ is contained in each search set, as desired.

Now note that every subdivision of $G / \mathcal{E}$ is the quotient graph under $\mathcal{E}$ of some subdivision of $G$. Let $G'$ be the subdivision of $G$ such that $G' / \mathcal{E}$ is the graph on which $\mathcal{S}$ is defined. Let $\mathcal{S}'$ be the search on $G'$ obtained by taking $\mathcal{S}$ and replacing each instance of $a$, $b$, or $c$ with the corresponding equivalence class in $\mathcal{E}$. Since at most one of $a$, $b$, or $c$ is in each $S_t$, we have $\abs{S_t'} \le \lceil\frac{n}{3}\rceil + 2$ and thus $\mathcal{S}'$ is a $(\lceil\frac{n}{3}\rceil + 2)$-search. By construction, $\mathcal{S}'$ is $\mathcal{E}$-invariant, and thus by Proposition~\ref{prop:quotient_search} it is successful, proving the claim.
\end{proof}

In particular, $\s_t(K_5) = \s_t(K_6) = 4$. It also follows that $4 \leq \s_t(K_7) \leq 5$, but we do not know which one is the correct value. We conjecture that $\lim_{n \to \infty} \s_t(K_n) = \infty$, but the best lower bound that we can currently prove is $4$. Indeed, we do not know is there exist \emph{any} graphs $G$ with $\s_t(G) > 4$!

\begin{prob}
    Is there $n \in \mathbb{N}$ such that $\s_t(K_n) > 4$? Is it true that in fact $\lim_{n \to \infty} \s_t(K_n) = \infty$? 
\end{prob}

Another interesting class of graphs to consider comprises planar graphs. It follows from Theorem~\ref{theo:classification} that outerplanar graphs have topological inspection number at most $3$, and that some planar graphs \ep{e.g., $K_4$} have topological inspection number $4$. We do not know if there is a constant upper bound that works for all planar graphs:

\begin{prob}
    Is there a constant $k$ such that $\s_t(G) \leq k$ for all planar graphs $G$? Does $k = 4$ work?
\end{prob}

    We finish the paper with an enticing conjecture that may provide a characterization of the inspection number in terms of vertex sets with small boundary. Observe that every lower bound on the inspection number established in this paper is proved using Proposition~\ref{prop:boundary} \ep{this includes Theorems \ref{theo:grids}, \ref{theo:trees}, \ref{res:1}, \ref{res:2}, and \ref{res:3}}. Namely, to show that $\s(G) > k$, we exhibit a subgraph $H$ of $G$ that does not contain a subset $C \subseteq V(H)$ of size close to a certain value and with $\abs{\partial_H(C)} < k$. Since this is the only method we have for proving lower bounds on $\s(G)$, it is natural to wonder if Proposition~\ref{prop:boundary} has a converse: 
    \begin{prob}
        If $\s(G) > k$, do there necessarily exist a subgraph $H \subseteq G$ and an integer $1 \leq i \leq \abs{V(H)}$ such that there is no subset $C \subseteq V(H)$ with $\abs{\partial_H(C)} < k$ and $i-k < \abs{C} < i$?
    \end{prob}
\probend
\printbibliography
\probend
\end{document}